\documentclass{amsart}

\usepackage{hyperref}
 \hypersetup{
     linktocpage=true,  
     colorlinks=true,
     linkcolor=blue,
     citecolor=blue,
     urlcolor=blue,
 }
\makeatletter
\newcommand\org@maketitle{}
\newcommand\@authors{}
\let\org@maketitle\maketitle
\def\maketitle{%
	\let\@authors\authors
	\nxandlist{; }{ and }{; }\@authors
	\hypersetup{
		pdftitle={\@title},
                pdfauthor={\@authors},
                pdfsubject={\subjclassname. \@subjclass},
		pdfkeywords={\@keywords}
	}%
	\org@maketitle
}
\makeatother

\usepackage[alphabetic,msc-links,initials]{amsrefs}
\usepackage{amssymb}
\usepackage{mathtools}
\usepackage{enumitem}
\usepackage{rsfso}
\usepackage{microtype}

\usepackage[all]{xy}

\usepackage{doi}
\renewcommand{\PrintDOI}[1]{\doi{#1}}

\let\arXiv\arxiv

\numberwithin{equation}{section}
\newtheorem{maintheorem}{Theorem}

\newtheorem{theorem}{Theorem}[section]
\newtheorem{lemma}[theorem]{Lemma}
\newtheorem{proposition}[theorem]{Proposition}
\newtheorem{corollary}[theorem]{Corollary}
\theoremstyle{definition}
\newtheorem{definition}[theorem]{Definition}
\theoremstyle{remark}
\newtheorem{remark}[theorem]{Remark}
\newtheorem{example}[theorem]{Example}
\newcommand{\cR}{\mathcal{R}}
\newcommand{\cC}{\mathcal{C}}
\newcommand{\cK}{\mathcal{K}}

\newcommand{\mfK}{\mathfrak{K}}
\newcommand{\mfa}{\mathfrak{a}}

\newcommand{\e}{\varepsilon}
\newcommand{\g}{\gamma}
\newcommand{\de}{\delta}
\newcommand{\la}{\lambda}

\newcommand{\La}{\Delta}
\newcommand{\Ga}{\Gamma}
\newcommand{\Ld}{\Lambda}
\newcommand{\Si}{\Sigma}
\newcommand{\D}{\nabla}
\newcommand{\ra}{\rightarrow}
\newcommand{\pa}{\partial}

\newcommand{\dv}{\operatorname{div}}
\newcommand{\R}{\mathbb{R}}
\newcommand{\vp}{\varphi}

\newcommand{\al}{\alpha}
\newcommand{\be}{\beta}
\newcommand{\ka}{\kappa}
\newcommand{\si}{\sigma}
\newcommand{\sm}{\setminus}
\newcommand{\Om}{\Omega}
\newcommand{\om}{\omega}

\newcommand{\K}{\widetilde{K}}

\newcommand{\wDh}{\widehat{\D h}}

\newcommand{\wDU}{\widehat{\D U}}

\newcommand{\bx}{\bar{x}}
\newcommand{\by}{\bar{y}}

\newcommand{\loc}{\mathrm{loc}}

\newcommand{\dist}{\operatorname{dist}}
\newcommand{\mean}[1]{\langle{#1}\rangle}
\renewcommand{\Re}{\operatorname{Re}}

\def\Xint#1{\mathchoice
  {\XXint\displaystyle\textstyle{#1}}%
  {\XXint\textstyle\scriptstyle{#1}}%
  {\XXint\scriptstyle\scriptscriptstyle{#1}}%
  {\XXint\scriptscriptstyle\scriptscriptstyle{#1}}%
  \!\int}
\def\XXint#1#2#3{{\setbox0=\hbox{$#1{#2#3}{\int}$}
    \vcenter{\hbox{$#2#3$}}\kern-.5\wd0}}

\def\dashint{\Xint-}

\mathtoolsset{centercolon}

\allowdisplaybreaks[4]

\author{Seongmin Jeon}
\address{Department of Mathematics, Purdue University, West Lafayette,
  IN 47907, USA}
\email{jeon54@purdue.edu}
\thanks{S.J. is supported in part by Purdue Research Foundation}
\author{Arshak Petrosyan}
\address{Department of Mathematics, Purdue University, West Lafayette,
  IN 47907, USA}
\email{arshak@purdue.edu}
\thanks{A.P. is is supported in part by NSF Grant DMS-1800527.}
\author{Mariana Smit Vega Garcia}
\address{Department of Mathematics, Western Washington University, Bellingham, WA 98225, USA}
\email{Mariana.SmitVegaGarcia@wwu.edu}

\title[Almost minimizers for the thin obstacle problem with var.\ coeff.]{Almost minimizers for the thin obstacle problem with variable coefficients}

\begin{document}
\begin{abstract} We study almost minimizers for the thin obstacle problem with
  variable H\"older continuous coefficients and zero thin obstacle and establish their $C^{1,\beta}$
  regularity on the either side of the thin space. Under an additional
  assumption of quasisymmetry, we establish the optimal growth of almost
  minimizers as well as the regularity of the regular set and a
  structural theorem on the singular set. The proofs are based on the
  generalization of Weiss- and Almgren-type monotonicity formulas for
  almost minimizers established earlier in the case of constant
  coefficients. 
\end{abstract}
\maketitle
\tableofcontents

\newpage
\section{Introduction and Main Results}
\subsection{The thin obstacle (or Signorini) problem with variable coefficients} Let $D$ be a domain in $\R^n$, $n\geq 2$, and $\Pi$ a smooth hypersurface (the \emph{thin space}), that
splits $D$ into two subdomains $D^\pm$: $D\setminus\Pi=D^+\cup
D^-$. Let $\psi:\Pi\to\R$ be a certain (smooth) function (the \emph{thin
  obstacle}) and $g:\partial D\to \R$ (the \emph{boundary
  values}).
Let also $A(x)=(a_{ij}(x))$ be an $n\times n$ symmetric uniformly elliptic matrix,
$\alpha$-H\"older continuous as a function of $x\in D$, for some
$0<\alpha<1$, with ellipticity constants $0<\lambda\leq1\leq\Lambda<\infty$:
$$
\lambda|\xi|^2\leq \langle{A(x)\xi,\xi}\rangle\leq
\Lambda|\xi|^2,\quad x\in D,\ \xi\in\R^n.
$$
Then consider the minimizer $U$ of the energy functional
$$
\mathcal{J}_{A,D}(V)=\int_{D}\langle A(x)\nabla V,\nabla V\rangle dx,
$$
over a closed convex set $\mfK_{\psi,g}(D,\Pi)\subset W^{1,2}(D)$ defined by
$$
\mfK_{\psi,g}(D,\Pi):=\{V\in W^{1,2}(D): \text{$V=g$ on $\partial D$,
  $V\geq \psi$ on $\Pi\cap D$}\}.
$$
Because of the unilateral constraint on the thin space $\Pi$, the
problem is known as the \emph{thin obstacle problem}.
Away from $\Pi$, the minimizer solves a uniformly elliptic divergence
form equation with variable coefficients
$$
\dv (A(x)\nabla U)=0\quad\text{in }D^+\cup D^-.
$$
On the thin space, the minimizers satisfy
\begin{multline*}
U\geq \psi,\quad \langle{A\nabla U,\nu^+}\rangle+\langle{A\nabla
  U,\nu^-}\rangle\geq 0,\\*
(U-\psi)(\langle{A\nabla U,\nu^+}\rangle+\langle{A\nabla
  U,\nu^-}\rangle)=0\quad\text{on }D\cap\Pi,
\end{multline*}
in a certain weak sense, where $\nu^\pm$ are the exterior normals to $D^\pm$ on $\Pi$ and
$\langle{A\nabla U,\nu^\pm}\rangle$ are understood as the limits from inside
$D^\pm$. These are known as the \emph{Signorini complementarity conditions} and therefore the problem is often
referred to as the \emph{Signorini problem} with variable
coefficients (or \emph{$A$-Signorini problem}, for short).
One of the main objects of the study is the \emph{free boundary}
$$
\Gamma(U)=\partial_\Pi \{x\in \Pi: U(x)=\psi(x)\}\cap D,
$$
which separates the \emph{coincidence set}
$\{U=\psi\}$ from the 
\emph{noncoincidence set} $\{U>\psi\}$ in $D\cap \Pi$. The set $\Gamma(U)$ is also
called a \emph{thin} free boundary as it lives in $\Pi$ and is expected to be of
codimension two with respect to the domain $D$.

These types of problems go back to the original Signorini problem in
elastostatics \cite{Sig59}, but also appear in many applications ranging from math
biology (semipermeable membranes) to boundary heat control \cite{DuvLio76} or more
recently in math finance, with connection to the obstacle problem for
the fractional Laplacian, through the Caffarelli-Silvestre extension
\cites{CafSil07}. The presence of the free
boundary makes the problem particularly challenging and while the
$C^{1,\beta}$ regularity of the minimizers (on the either side of the
thin space) was known already in \cites{Caf79,Kin81,Ura85}, the study of the free
boundary became possible only after the breakthrough work of
\cite{AthCaf04} on the optimal $C^{1,1/2}$ regularity of the
minimizers. Since then there has been a significant effort in the
literature to understand the structure and regularity properties of
the free boundary in many different settings including equations with
variable coefficients, problems for the fractional Laplacian, as well
as the time-dependent problems, see e.g.\
\cites{Sil07,AthCafSal08,CafSalSil08,GarPet09,GarSVG14,KocPetShi15,PetPop15,DeSSav16,GarPetSVG16,KocRueShi16,BanSVGZel17,CafRosSer17,DanGarPetTo17,GarPetPopSVG17,KocRueShi17a,KocRueShi17b,RueShi17,AthCafMil18,DanPetPop18,FocSpa18,GarPetSVG18,PetZel19,ColSpoVel20},
and many others.

\subsection{Almost minimizers}

The approach we take in this paper is by considering the so-called
almost minimizers of the functional $\mathcal{J}_{A,D}$ in the sense
of Anzellotti \cite{Anz83}. For this we need a \emph{gauge} function
$\omega:(0,r_0)\to [0,\infty)$, $r_0>0$, which is a nondecreasing
function with $\omega(0+)=0$, as well as a family
$\{E_r(x_0)\}_{0<r<r_0}$ of open sets for any $x_0\in D$, comparable
to balls centered at $x_0$ (in what comes next, we will take it to be a family of ellipsoids).

\begin{definition}[Almost minimizers]\label{def:almost-min-A}
We say $U$ is an \emph{almost minimizer for the $A$-Signorini problem}
in $D$ if $U\in W^{1,2}_{\loc}(D)$, $U\geq \psi$ on $D\cap\Pi$, and for any
$E_r(x_0)\Subset D$ with $0<r<r_0$, we have 
\begin{equation}\label{eq:almost-min-A}
\int_{E_r(x_0)}\langle A\nabla U,\nabla U \rangle\leq (1+\omega(r)) \int_{E_r(x_0)}\langle A\nabla V,\nabla V \rangle,
\end{equation}
for any competitor function $V\in\mfK_{\psi,U}(E_{r}(x_0),\Pi)$, i.e.,
$V$ satisfying
$$
V=U\quad\text{on }\partial E_r(x_0),\quad V\geq \psi\quad\text{on
}E_r(x_0)\cap \Pi.
$$
\end{definition}
In fact, observing that for $x,x_0\in D$, and $\xi\in\R^n$, $\xi\neq 0$
$$
(1-C|x-x_0|^\alpha)\leq\frac{\langle
  A(x_0)\xi,\xi\rangle}{\langle A(x)\xi,\xi\rangle}\leq
(1+C|x-x_0|^\alpha),
$$
with $C$ depending on the ellipticity of $A$ and
$\|A\|_{C^{0,\alpha}(D)}$,
we can rewrite \eqref{eq:almost-min-A} in
the form  \emph{with frozen coefficients}
\begin{equation}\label{eq:almost-min-A-frozen}
\int_{E_r(x_0)}\langle A(x_0)\nabla U,\nabla U \rangle\leq (1+\omega(r)) \int_{E_r(x_0)}\langle A(x_0)\nabla V,\nabla V \rangle,
\end{equation}
by replacing the gauge
$\omega(r)$ with $C(\omega(r)+r^\alpha)$ if necessary.

An example of an almost minimizer is given in Appendix. Generally, we view almost minimizers as perturbations of minimizers in
a certain sense, but in the case of variable coefficients there are even some
advantages of treating minimizers themselves as almost minimizers,
particularly in the sense of frozen coefficients \eqref{eq:almost-min-A-frozen}.

Almost minimizers for the Signorini problem have already been studied
in
\cite{JeoPet19a} in the case $A(x)\equiv I$, where their $C^{1,\beta}$-regularity (on the either side of the thin space) has been established and a number of
technical tools such as Weiss- and Almgren-type monotonicity formulas
were proved. In combination with the epiperimetric and
$\log$-epiperimetric inequalities these tools allowed to establish the optimal
growth and prove the $C^{1,\gamma}$-regularity of the regular set and a
structural theorem on the singular set. 
The aim of this paper is to extend these results to the variable
coefficient case.
It is noteworthy that the results that we obtain (see
Theorems~\ref{mthm:I}--\ref{mthm:IV} below) for almost minimizers improve even on some of the results
available for the minimizers. For example, we only need the coefficients $A(x)$ to be
$C^{0,\alpha}$ with arbitrary $0<\alpha<1$ in order to study the free
boundary, compared to $W^{1,p}$, $p>n$, in \cite{KocRueShi17b} or
$C^{0,\alpha}$, $1/2<\alpha<1$, in \cite{RueShi17} for the regular
part of the free boundary and $C^{0,1}$ in \cite{GarPetSVG18} for the
singular set. 

A related notion of almost minimizers has been considered recently in \cite{JeoPet19b} for the obstacle problem for the fractional Laplacian, with the help of the Caffarelli-Silvestre extension. While the $C^{1,\beta}$ regularity of almost minimizers holds for the fractional orders $1/2\leq s<1$, the study of the free boundary still remains open.

Almost minimizers have been studied also for other free boundary
problems,  particularly Alt-Caffarelli-type (or Bernoulli-type) problems \cites{DavTor15,DavEngTor19,DeSSav19}, their thin
counterpart \cite{DeSSav18}, as well as the variable coefficient
versions \cites{deQTav18,DavEngSVGTor19}. We have to mention that the
Signorini problem is quite different from Alt-Caffarelli-type
problems, as the solutions may grow at different rates near the free
boundary (such as $3/2$, $2$, $7/2$, $4$, \dots,  powers of the distance), as opposed to a specific
rate in Alt-Caffarelli-type problems (linear in the classical case and
the square root of the distance in the thin counterpart). Therefore, it is quite
important that the almost minimizing property that we impose for the
Signorini problem is \emph{multiplicative}, to allow the capture of all
possible rates, while the almost minimizing property in the Alt-Caffarelli-type
problems can be also imposed in an \emph{additive} way, see \cite{DavEngSVGTor19}.

\subsection{Main results} Since we are interested in local regularity
results, we will assume that $D=B_1$, the unit ball in $\R^n$, and that
$$
\Pi=\R^{n-1}\times\{0\}
$$
after a local diffeomorphism. In this paper, we will consider only the
case when the thin obstacle is identically zero: $\psi\equiv 0$. 

Further, we will assume $r_0=1$ in Definition~\ref{def:almost-min-A}
and take $\{E_r(x_0)\}$ to be the family of ellipsoids associated
with the positive symmetric matrix $A(x_0)$:
$$
E_r(x_0):=A^{1/2}(x_0)(B_r)+x_0.
$$
By the ellipticity of $A(x_0)$, we have
$$
B_{\lambda^{1/2}r}(x_0)\subset E_r(x_0)\subset B_{\Lambda^{1/2}r}(x_0).
$$
To simplify the tracking of the constants, we will assume that there
is $M>0$ such that
\begin{equation}\label{eq:M}
\|A\|_{C^{0,\alpha}(B_1)}\leq M,\quad\lambda^{-1},\Lambda\leq M,\quad
\omega(r)\leq M r^\alpha,\quad 0<\alpha<1.
\end{equation}
Then we can go between almost minimizing properties
\eqref{eq:almost-min-A} and \eqref{eq:almost-min-A-frozen} by changing
$M$ if necessary.

\medskip
Then our first result is as follows.

\begin{maintheorem}[$C^{1,\beta}$-regularity of almost minimizers]
  \label{mthm:I}
  Let $U\in W^{1,2}(B_1)$ be an almost minimizer for the $A$-Signorini problem in $B_1$,
  under the assumptions above. Then,
  $U\in C^{1,\beta}_\loc(B_1^\pm\cup B_1')$ for
  $\beta=\beta(\alpha,n)\in (0,1)$ and
  $$
  \|U\|_{C^{1,\beta}(K)}\leq C\|U\|_{W^{1,2}(B_1)},
$$
for any $K\Subset B_1^\pm\cup B_1'$ and $C=C(n,\alpha,M,K)$.
\end{maintheorem}
The proof is obtained by using Morrey and Campanato space estimates,
following the original idea of Anzellotti \cite{Anz83} that was
successfully used in the constant coefficient case of our problem in
\cite{JeoPet19a}. We explicitly mention, however, that in the above
theorem we do not require the even symmetry of the almost minimizer in
the $x_n$-variable, so Theorem~\ref{mthm:I} extends the corresponding
result in \cite{JeoPet19a} also in that respect.

To state our results related to the free boundary, we need to assume
the following quasisymmetry condition. For $x_0\in B_1'=B_1\cap \Pi$,
let
$$
P_{x_0}=I-2\frac{A(x_0)e_n\otimes e_n }{a_{nn}(x_0)}
$$
be a matrix corresponding to the reflection with respect to $\Pi$ in
the conormal direction $A(x_0)e_n$ at $x_0$. Note that $P_{x_0}x=x$ for
any $x\in\Pi$ and $P_{x_0}E_r(x_0)=E_r(x_0)$. Then, for a function
$U$ in $B_1$ define
$$
U^*_{x_0}(x):=\frac{U(x)+U(P_{x_0}x)}{2}.
$$
Note that $U^*_{x_0}$ may not be defined in all of $B_1$, but is
defined in any ellipsoid $E_r(x_0)$ as long as it is contained in $B_1$. Note also that $U=U^*_{x_0}$ on $\Pi$.

\begin{definition}[Quasisymmetry]\label{def:A-quasisym}

  We say that $U\in W^{1,2}(B_1)$ is
  $A$-\emph{quasisymmetric} with respect to $\Pi$, if there is a
  constant $Q$ such that
$$
\int_{E_r(x_0)}\langle A(x_0)\nabla U,\nabla U\rangle\leq Q
\int_{E_r(x_0)}\langle A(x_0)\nabla U^*_{x_0},\nabla U^*_{x_0}\rangle,
$$
for any ellipsoid $E_r(x_0)\Subset B_1$ centered at any $x_0\in B_1'$.

We will assume
$Q\leq M$ throughout the paper, in addition to \eqref{eq:M}. 
\end{definition}

Note that when $A(x)\equiv I$ and $U$ is even in $x_n$, then it
is automatically quasisymmetric in the sense of the above definition. The quasisymmetry condition will also hold for even minimizers if $e_n$ is an eigenvector of $A(x_0)$ for any $x_0\in B_1'$, i.e., when
$$
a_{in}(x_0)=0,\quad\text{for }i=1,\dots,n-1,\ x_0\in B_1'.
$$
This condition is typically imposed in the existing literature and can be satisfied with an application of a local $C^{1,\alpha}$-diffeomorphism that preserves $\Pi$, see \cites{Ura86,GarSVG14,RueShi17}.
The reason for a quasisymmetry condition is that the growth rate of the symmetrization $U_{x_0}^*$ over the ellipsoids $E_r(x_0)$ captures that of $U=U^*_{x_0}$ on the thin space $\Pi$ at $x_0\in\Gamma(U)$,
while in the nonsymmetric case there could be a mismatch in those rates
caused by the odd component of $U$, vanishing on $\Pi$. 

More specifically, the growth rate of $U$ on $\Pi$ at $x_0\in\Gamma(U)$
is determined by the following quantity
$$
N^A(r,U^*_{x_0},x_0):=\frac{r\int_{E_r(x_0)}\langle A(x_0)\nabla U^*_{x_0},\nabla U^*_{x_0}\rangle}{\int_{\partial E_r(x_0)}(U_{x_0}^*)^2\mu_{x_0}(x-x_0)},
$$
which is a version of \emph{Almgren's frequency functional}
\cite{Alm00} written in the geometric terms determined by $A(x_0)$,
where
$\mu_{x_0}(z)=\frac{|A^{-1/2}(x_0)z|}{|A^{-1}(x_0)z|}$
is the \emph{conformal factor}. As in the constant coefficient case,
this quantity is of paramount importance for the classification of free
boundary points. 

\begin{maintheorem}[Monotonicity of the truncated frequency]
  \label{mthm:II}
  Let $U$ be as in Theorem~\ref{mthm:I} and assume additionally that
  $U$ is $A$-quasisymmetric with respect to $\Pi$. Then for any
  $\kappa_0\geq 2$, there is $b=b(n,\alpha,M,\kappa_0)$ such that the
  \emph{truncated frequency}
$$
r\mapsto\widehat{N}^A_{\kappa_0}(r,U_{x_0}^*,x_0) :=
\min\left\{\frac{1}{1-br^\alpha}N^A(r,U^*_{x_0},x_0),\kappa_0\right\}
$$ 
is monotone increasing for $x_0\in B_{1/2}'\cap\Gamma(u)$, and
$0<r<r_0(n,\alpha,M,\kappa_0)$. Moreover, if we define 
$$
\kappa(x_0):=\widehat{N}^A_{\kappa_0}(0+,U_{x_0}^*,x_0),
$$
the \emph{frequency of $U$ at $x_0$}, then we have that either
$$
\kappa(x_0)=3/2\quad\text{or}\quad
\kappa(x_0)\geq 2.
$$
\end{maintheorem}
The monotonicity of the truncated frequency follows from that of an
one-parametric family of so-called Weiss-type energy functionals
$\{W_{\kappa}^A\}_{0<\kappa<\kappa_0}$, see
Section~\ref{sec:weiss-almgren-type}, which also play a fundamental role in the analysis of the free boundary.

The theorem above gives the following
decomposition of the free boundary 
$$
\Gamma(U)=\Gamma_{3/2}(U)\cup\bigcup_{\kappa\geq 2}
\Gamma_\kappa(U),
$$
where
$$
\Gamma_\kappa(U):=\{x_0\in\Gamma(U):\kappa(x_0)=\kappa\}. 
$$
The set $\Gamma_{3/2}(U)$, where the frequency is minimal is known as
the \emph{regular set} and is also denoted $\mathcal{R}(U)$.

\begin{maintheorem}[Regularity of the regular set]
  \label{mthm:III}
  Let $U$ be as in Theorem~\ref{mthm:II}. Then $\mathcal{R}(U)$ is a relatively
  open subset of the free boundary $\Gamma(U)$ and is an
  $(n-2)$-dimensional manifold of class $C^{1,\gamma}$.
\end{maintheorem}

Finally, we state our main result for the so-called \emph{singular
  set}. A free boundary point $x_0\in \Gamma(U)$ is called
\emph{singular} if the \emph{coincidence set}
$\Lambda(U):=\{x\in B_1': U(x)=0\}$ has $H^{n-1}$-density zero at $x_0$,
i.e.,
$$
\lim_{r\to 0+} \frac{H^{n-1}(\Lambda(U)\cap
  B_r'(x_0))}{H^{n-1}(B_r')}=0.
$$
We denote the set of all singular points by $\Sigma(U)$ and call it
the \emph{singular set}. It can be shown that if
$\kappa(x_0)<\kappa_0$, then $x_0\in\Sigma(U)$ if and only if $\kappa(x_0)=2m$, $m\in\mathbb{N}$ (see
Proposition~\ref{prop:var-sing-char}). For such values of $\kappa$, we then define
$$
\Sigma_{\kappa}(U):=\Gamma_\kappa(U).
$$

\begin{maintheorem}[Structure of the singular set]
  \label{mthm:IV}
  Let $U$ be as in Theorem~\ref{mthm:II}. Then, for any
  $\kappa=2m<\kappa_0$, $m\in\mathbb{N}$, $\Sigma_\kappa(U)$ is
  contained in a countable union of $(n-2)$-dimensional manifolds of
  class $C^{1,\log}$.
\end{maintheorem}
A more refined version of this result is given in
Theorem~\ref{thm:var-sing}.

Theorems~\ref{mthm:III} and \ref{mthm:IV} follow by establishing the
uniqueness and continuous dependence of \emph{almost homogeneous blowups} with H\"older modulus of continuity in the case of
regular free boundary points and a logarithmic one in the case of the
singular points. These follow from optimal growth and rotation
estimates which are based on the use of Weiss-type monotonicity
formulas in conjunction with so-called \emph{epiperimetric}
\cite{GarPetSVG16} and \emph{$\log$-epiperimetric} \cite{ColSpoVel20}
inequalities for the solutions of the Signorini problem.

\subsubsection{Proofs of Theorems~\ref{mthm:I}--\ref{mthm:IV}} While we
don't give formal proofs of the theorems above in the main body of the paper, they are contained in the following results proved there:
\begin{enumerate}[label=$\circ$,leftmargin=*,labelindent=\parindent]
\item Theorem~\ref{mthm:I} is essentially the same as Theorem~\ref{thm:var-grad-holder}.
\item Theorem~\ref{mthm:II} follows by combining
  Theorem~\ref{thm:var-Almgren} and Corollary~\ref{cor:var-gap}.
\item The statement of Theorem~\ref{mthm:III} is contained in that of
  Theorem~\ref{thm:var-C1g-regset}.
\item The statement of Theorem~\ref{mthm:IV} is contained in that of
  Theorem~\ref{thm:var-sing}.
\end{enumerate}

\subsection{Notation}
We use the following notation throughout the paper.  

$\R^n$ stands for
the $n$-dimensional Euclidean space. The points of $\R^n$ are denoted by
$x=(x', x_n)$, where $x'=(x_1, \ldots, x_{n-1})\in \R^{n-1}$. We
often identify $x'\in \R^{n-1}$ with
$(x', 0)\in \R^{n-1}\times \{0\}$. $\R^{n}_\pm$ stand for open
halfspaces $\{x\in\R^n: \pm x_n>0\}$.

For $\xi,\eta\in\R^n$, the standard inner product is denoted by $\langle\xi,\eta\rangle$. Thus, $|\xi|^2=\langle\xi,\xi\rangle$, where $|\xi|$ is the Euclidean norm of $\xi$.

For $x\in\R^n$, $r>0$, we denote
\begin{alignat*}{2}
  B_r(x)&:=\{y\in \R^n:|x-y|<r\},&\quad&\text{ball in $\R^n$,}\\
  B^{\pm}_r(x')&:=B_r(x',0)\cap \{\pm x_n>0\},&& \text{half-ball in $\R^n$,}\\
  B'_r(x')&:=B_r(x',0)\cap \{ x_n=0\}, &&\text{ball in $\R^{n-1}$, or
    thin ball.}
\end{alignat*}
We typically drop the center from the notation if it is the
origin. Thus, $B_r:=B_r(0)$, $B'_r:=B'_r(0)$, etc.

For a function $f$ in $\R^n$, $\nabla f$ denotes its gradient (in the classical or weak sense)
$$
\nabla f:=(\partial_{x_1}f,\partial_{x_2}f,\dots,\partial_{x_n}f),
$$
where $\partial_{x_i}f$ are the partial derivatives in the variables $x_i$, $i=1,2,\dots,n$.

In integrals, we often drop the variable and the measure of
integration if it is with respect to the Lebesgue measure or the
surface measure. Thus,
$$
\int_{B_r} f=\int_{B_r} f(x)dx,\quad \int_{\partial B_r}
f=\int_{\partial B_r} f(x)dS_x,
$$
where $S_x$ stands for the surface measure.

If $E$ is a set of positive and finite Lebesgue measure, we indicate by $\mean{f}_{E}$ the integral mean value of a function
$u$ over $E$. That is,
$$
\mean{f}_{E}:=\dashint_{E}f=\frac{1}{|E|}\int_{E}f.
$$

\section{Coordinate transformations}
In order to use the results available for almost minimizers in the
case of $A\equiv I$, proved in \cite{JeoPet19a}, in this section we
describe a ``deskewing procedure'' or coordinate
transformations to straighten $A(x_0)$, $x_0\in B_1$.

For the notational convenience, we will denote
$$
\mfa_{x_0}=A^{1/2}(x_0),\quad x_0\in B_1
$$
so that
$$
\langle{A(x_0)\xi,\xi}\rangle=|\mfa_{x_0}\xi|^2,\quad \xi\in\R^n.
$$
Then $\mfa_{x_0}$ is a symmetric positive definite matrix, with
eigenvalues between $\lambda^{1/2}$ and $\Lambda^{1/2}$ and the
mapping $x_0\mapsto
\mfa_{x_0}$ is $\alpha$-H\"older continuous for $x_0\in B_1$. 
For every $x_0\in B_1$, we
define an affine transformation $T_{x_0}$ by
$$
T_{x_0}(x)=\mfa_{x_0}^{-1}(x-x_0).
$$
Note that $T_{x_0}^{-1}(y)=\mfa_{x_0}y+x_0$.
Then for the ellipsoids $E_r(x_0)$, we have
$$
E_r(x_0)=T_{x_0}^{-1}(B_r)=\mfa_{x_0} B_r+x_0,\quad T_{x_0}(E_r(x_0))=B_r.
$$
Further, we let
$$
\Pi_{x_0}:=T_{x_0}(\Pi).
$$
Then $\Pi_{x_0}$ is a hyperplane parallel to a linear subspace
$\mfa_{x_0}^{-1}\Pi$ spanned by the vectors $\mfa_{x_0}^{-1}e_1$, $\mfa_{x_0}^{-1}e_2$, \dots,
$\mfa_{x_0}^{-1}e_{n-1}$ and with a normal $\mfa_{x_0} e_n$. Generally, this
hyperplane will be tilted with respect to $\Pi$, unless $\mfa_{x_0} e_n$
is a multiple of $e_n$, or equivalently that $e_n$ is an eigenvector
of the matrix $A(x_0)$, or that $a_{in}(x_0)=0$ for
$i=1,\dots,n-1$ for its entries. To rectify
that, we construct a family of 
orthogonal transformations $O_{x_0}$, $x_0\in B_1$, by applying the
Gram-Schmidt process to the ordered basis $\{\mfa_{x_0}^{-1}
e_1, \mfa_{x_0}^{-1} e_2,\dots, \mfa_{x_0}^{-1} e_{n-1}\}$ of
$\mfa_{x_0}^{-1}\Pi$.
Namely, let
\begin{align*}
    e_1^{x_0}&:=\frac{\mfa_{x_0}^{-1} e_1}{|\mfa_{x_0}^{-1} e_1|},\\
    e_2^{x_0}&:=\frac{\mfa_{x_0}^{-1} e_2-\langle \mfa_{x_0}^{-1} e_2,e_1^{x_0}\rangle
      e_1^{x_0}}{|\mfa_{x_0}^{-1} e_2-\langle \mfa_{x_0}^{-1} e_2,e_1^{x_0}\rangle e_1^{x_0}|},\\
    e_3^{x_0}&:=\frac{\mfa_{x_0}^{-1} e_3-\langle{\mfa_{x_0}^{-1} e_3,e_1^{x_0}}\rangle
      e_1^{x_0}-\langle{\mfa_{x_0}^{-1} e_3,e_2^{x_0}}\rangle e_2^{x_0}}{|\mfa_{x_0}^{-1}
      e_3- \langle{\mfa_{x_0}^{-1} e_3,e_1^{x_0}}\rangle e_1^{x_0}
      - \langle{\mfa_{x_0}^{-1} e_3,e_2^{x_0}}\rangle e_2^{x_0}|}\\
    &\vdots\\
\intertext{Moreover, letting}
e_n^{x_0}&:=\frac{\mfa_{x_0}e_n}{|\mfa_{x_0} e_n|},
\end{align*}
we obtain an ordered orthonormal basis $\{e_1^{x_0},\dots,e_{n-1}^{x_0},e_n^{x_0}\}$ of $\R^n$. Then consider the rotation $O_{x_0}$ of $\R^n$ that takes the standard basis $\{e_1,e_2,\dots,e_n\}$ to the one above, i.e.,
$$
O_{x_0}:\R^n\to\R^n,\quad O_{x_0}(e_i)=e_i^{x_0},\ i=1,2,\dots,n.
$$
Note that the Gram-Schmidt process above guarantees that $x_0\mapsto
O_{x_0}$ is $\alpha$-H\"older continuous. We also have that by construction
$$
O_{x_0}^{-1}\mfa_{x_0}^{-1}\Pi=\Pi.
$$
In particular, when $x_0\in\Pi$, we have $\Pi_{x_0}=\mfa_{x_0}^{-1}\Pi$ and
therefore
$$
O_{x_0}^{-1}(\Pi_{x_0})=\Pi.
$$
Because of this property, we also define the modifications of the
matrices $\mfa_{x_0}$ and the 
transformations $T_{x_0}$ as follows:
$$
\bar{\mfa}_{x_0}=\mfa_{x_0}O_{x_0},\quad \bar{T}_{x_0}=
O_{x_0}^{-1}\circ T_{x_0},
$$
so that $\bar{T}_{x_0}(x)=\bar{\mfa}_{x_0}^{-1}(x-x_0)$.
Since $O_{x_0}$ is a rotation, we still have
$$
E_r(x_0)=\bar{T}_{x_0}^{-1}(B_r),\quad \bar{T}_{x_0} (E_r(x_0))=B_r,
$$
see Fig.~\ref{fig:deskew}.
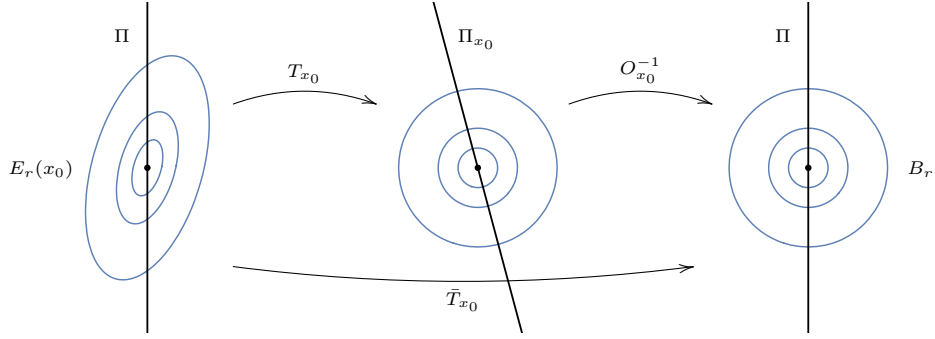
\begin{figure}[t]
\centering
\begin{picture}(375,125)(0,0)
  \put(0,0){\includegraphics[height=125pt]{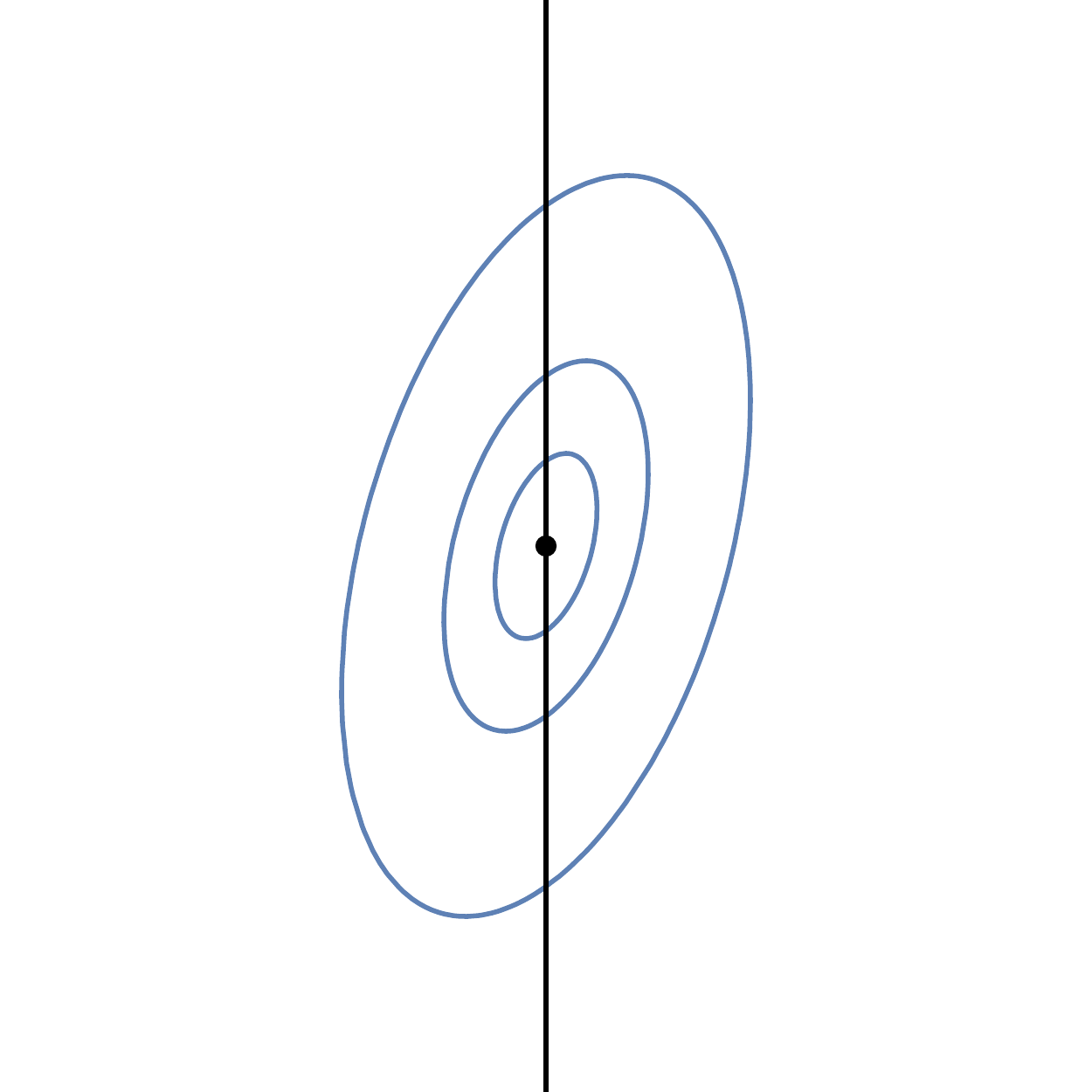}}
  \put(125,0){\includegraphics[height=125pt]{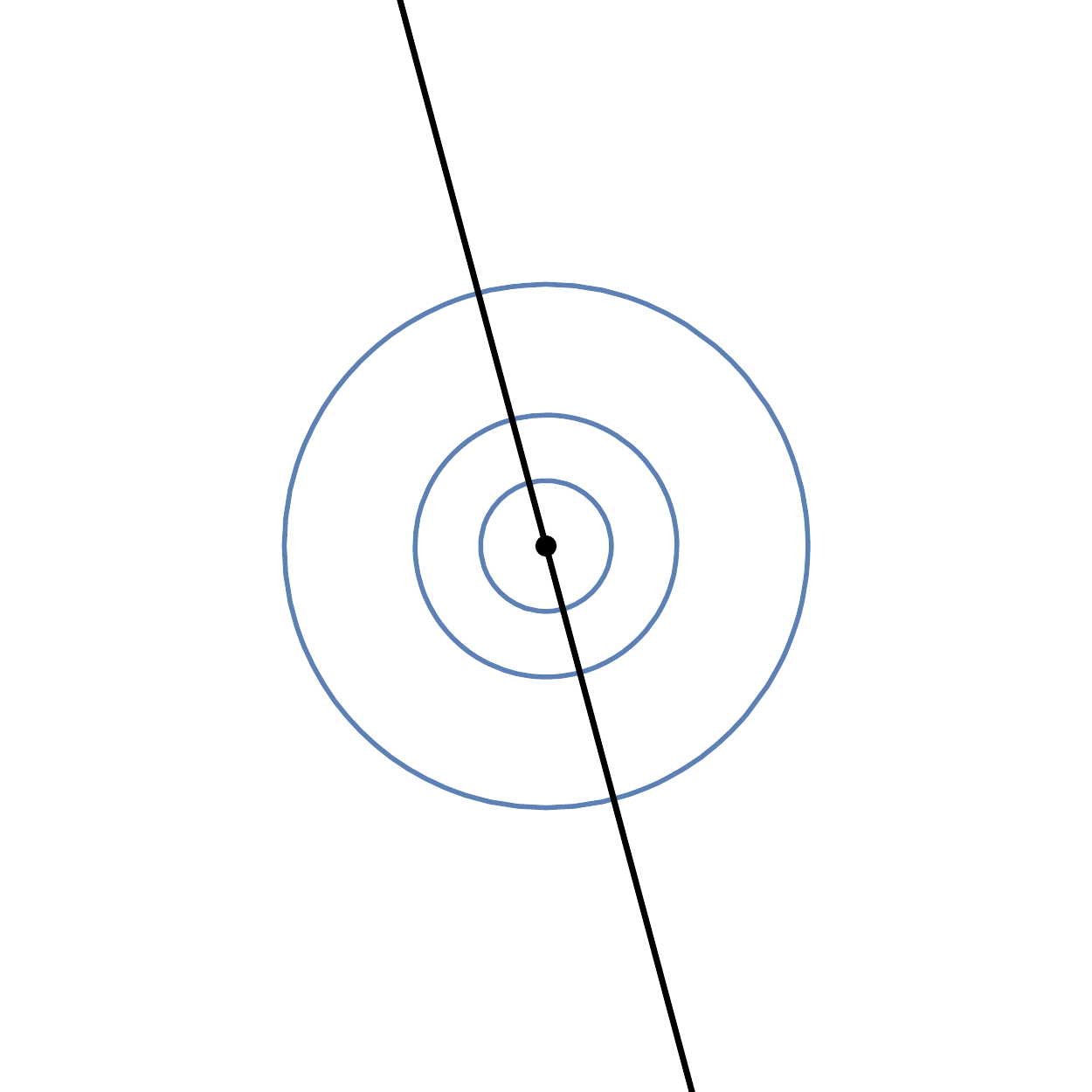}}
  \put(250,0){\includegraphics[height=125pt]{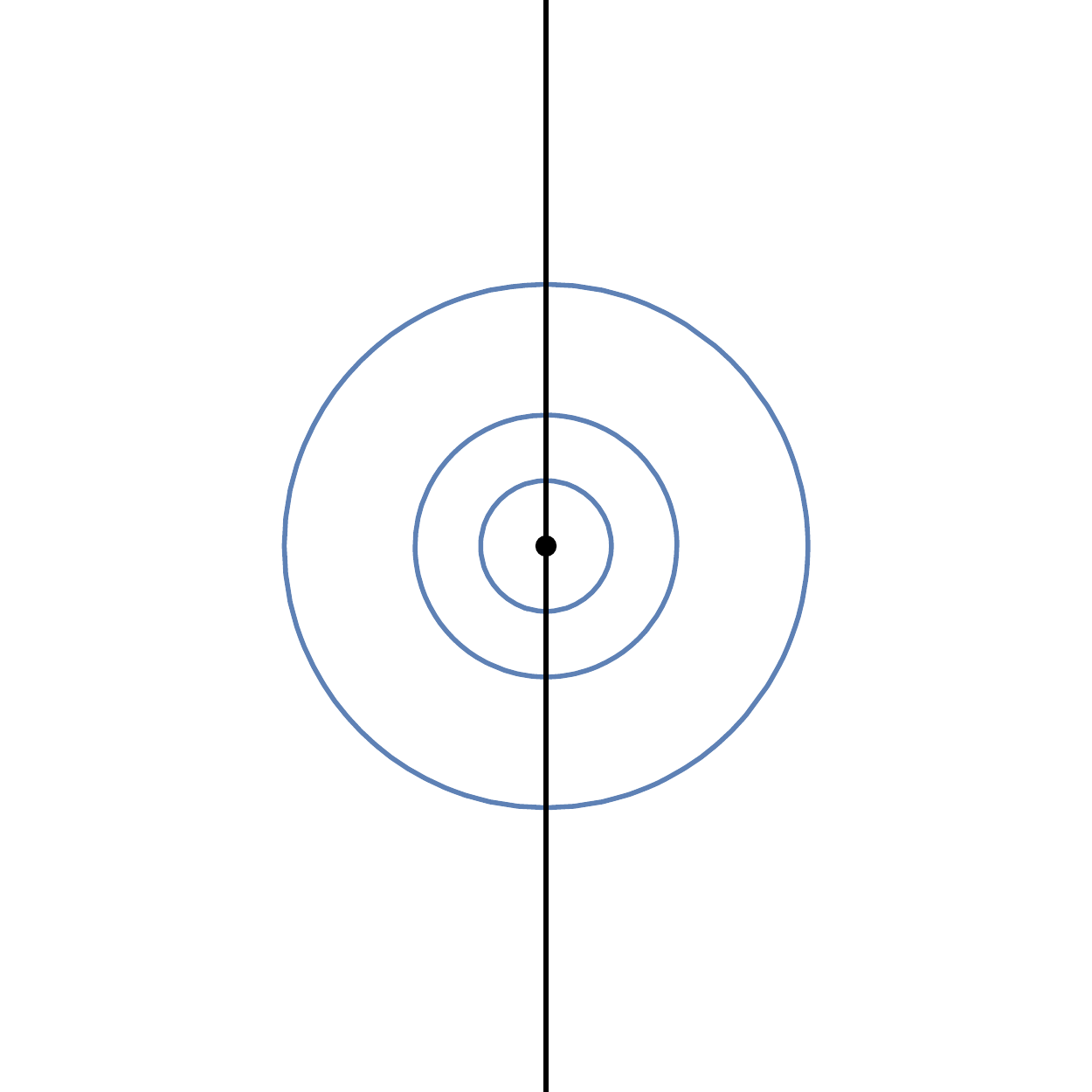}}
  \put(89,83){$\xymatrix{\ar@/^/[rr]^{T_{x_0}}_{} & & } $}
  \put(216,83){$\xymatrix{\ar@/^/[rr]^{O_{x_0}^{-1}}_{} & & } $}
  \put(89,23){$\xymatrix{\ar@/_/[rrrrrr]^{}_{\bar{T}_{x_0}}
      & & & & & &} $}
  \put(10,60){\scriptsize $E_r(x_0)$}
  \put(350,60){\scriptsize $B_r$}
  \put(50,110){\scriptsize $\Pi$}
  \put(180,110){\scriptsize $\Pi_{x_0}$}
  \put(300,110){\scriptsize $\Pi$}
\end{picture}
\caption{Deskewing: coordinate transformations $T_{x_0}$,
  $O_{x_0}^{-1}$, $\bar T_{x_0}$.}
\label{fig:deskew}
\end{figure}%

Next, for a function $U:B_1\to \R$ and a point $x_0\in B_1$, we define
its ``deskewed'' version at $x_0$ by
$$
u_{x_0}=U\circ \bar{T}_{x_0}^{-1}.
$$
As we will see, if $U$ is an almost minimizer, the transformed
function $u_{x_0}$ will satisfy an almost minimizing property with the
identity matrix $I$ at the origin. Before we state and prove that fact, 
we need the following basic change of variable formulas:
\begin{align}
\label{eq:u-u_x-0}
  \int_{E_r(x_0)} U^2
  &=\det \mfa_{x_0}\int_{B_r}u_{x_0}^2\\
  \label{eq:u-u_x-1}
  \int_{E_r(x_0)}\mean{A(x_0)\D U,\D U}
  &=\det \mfa_{x_0}\int_{B_r}|\D u_{x_0}|^2\\
\label{eq:u-u_x-bdry}
  \int_{\pa E_r(x_0)}U^2\mu_{x_0}(x-x_0)
  &=\det \mfa_{x_0}\int_{\pa B_r}u_{x_0}^2,                     
\end{align}
with the conformal factor
\begin{equation}\label{eq:mu}
\mu_{x_0}(z):=\frac{|\mfa_{x_0}^{-1}z|}{|A^{-1}(x_0)z|}.
\end{equation}
We also have the following modified version of \eqref{eq:u-u_x-1}. 
\begin{align}
\label{eq:u-u_x-2}
  \int_{E_r(x_0)}|\mfa_{x_0}\D U-\mean{\mfa_{x_0}\D
    U}_{E_r(x_0)}|^2
  =\det \mfa_{x_0}\int_{B_r}|\D u_{x_0}-\mean{\D u_{x_0}}_{B_r}|^2.
\end{align}
While \eqref{eq:u-u_x-0}--\eqref{eq:u-u_x-1} and \eqref{eq:u-u_x-2} are clear, let us give more details
on \eqref{eq:u-u_x-bdry}.
If we let $f(x):=|\mfa_{x_0}^{-1}(x-x_0)|$, then $\{f=t\}=\pa
E_t(x_0)$, $t>0$, and by the coarea formula
$$
\int_{E_r(x_0)}U^2dx=\int_0^r\int_{\pa E_t(x_0)}\frac{U^2}{|\D
  f(x)|}dS_xdt.
$$
Using now that $1/|\nabla
f(x)|=\frac{|\mfa_{x_0}^{-1}(x-x_0)|}{|A^{-1}(x_0)(x-x_0)|}=\mu_{x_0}(x-x_0)$
and then differentiating \eqref{eq:u-u_x-0}, we obtain \eqref{eq:u-u_x-bdry}.

We will also need the following estimate for the conformal factor $\mu_{x_0}$:
\begin{equation}\label{eq:mu-bounds}
  \lambda^{1/2}\leq\mu_{x_0}(z)\leq \Lambda^{1/2}.
\end{equation}
Indeed, if $y=A^{-1}(x_0)z$, then
$$
\mu_{x_0}(z)=\frac{|\mfa_{x_0}y|}{|y|}\in[\lambda^{1/2},\Lambda^{1/2}].
$$

\begin{definition}[Almost Signorini property at a point] We say that a function $u\in
W^{1, 2}(B_R)$ satisfies the \emph{almost Signorini property at $0$}
in $B_R$ if
$$
\int_{B_r}|\D u|^2\le (1+\omega(r))\int_{B_r}|\D v|^2,
$$
for all $0<r<R$ and $v\in \mathfrak{K}_{0,u}(B_r, \Pi)$.
\end{definition}

\begin{lemma}\label{lem:u_x-alm-min}
Suppose $U$ is an almost minimizer of the $A$-Signorini problem in
$B_1$. Let $x_0\in B'_{1}$ be such that $E_R(x_0)\subset B_1$. Then $u_{x_0}=U\circ \bar{T}^{-1}_{x_0}$ satisfies the almost Signorini property at $0$ in $B_R$.
\end{lemma}

\begin{proof}
Let $V$ be the energy minimizer of $\int_{E_r(x_0)}\langle A(x_0)\D V,\D
V\rangle$ on $\mathfrak{K}_{0,U}(E_r(x_0),\Pi)$,
$0<r<R$. Then $v_{x_0}=V\circ \bar{T}_{x_0}^{-1}$ is the
energy minimizer of $\int_{B_r}|\D v_{x_0}|^2$ on
$\mathfrak{K}_{0,u_{x_0}}(B_r,\Pi)$. Moreover, by
\eqref{eq:u-u_x-1},
\begin{align*}
\int_{B_r}|\D u_{x_0}|^2 &= \det \mfa_{x_0}^{-1}\int_{E_r(x_0)}\langle A(x_0)\D U, \D U\rangle\\
&\le (1+\omega(r))\det \mfa_{x_0}^{-1}\int_{E_r(x_0)}\langle A(x_0)\D V, \D V\rangle\\
&= (1+\omega(r))\int_{B_r}|\D v_{x_0}|^2.
\end{align*}
This completes the proof.
\end{proof}

\section{Almost $A$-harmonic functions}

We start our analysis of almost minimizers in the absence of the thin obstacle. We call such functions almost $A$-harmonic functions. In this section, we establish their $C^{1,\alpha/2}$ regularity (Theorem~\ref{thm:var-alm-har-reg}). A similar result has already been proved by Anzellotti \cite{Anz83}, but for almost minimizers over balls
$\{B_r(x_0)\}$ rather than ellipsoids $\{E_r(x_0)\}$; nevertheless,
the proofs are similar. The proofs in this section also illustrate how
we are going to use the results available for ``deskewed''
functions $u_{x_0}=U\circ \bar{T}_{x_0}^{-1}$ to infer the corresponding
results for almost minimizers $U$.

\begin{definition}[Almost $A$-harmonic functions] We say that $U$ is an \emph{almost $A$-harmonic function} in $B_1$ if
  $U\in W^{1, 2}(B_1)$ and
$$
\int_{E_r(x_0)}\mean{A\D U, \D U}\le (1+\omega(r))\int_{E_r(x_0)}\mean{A\D V, \D V},
$$
whenever $E_r(x_0)\Subset B_1$ and $V\in \mathfrak{K}_{U}(E_r(x_0)):= U+W^{1,2}_0(E_r(x_0))$.
\end{definition}
Note that similarly to the case of $A$-Signorini
problem, we can write the almost minimizing property above in the form with frozen coefficients 
$$
\int_{E_r(x_0)}\mean{A(x_0)\D U, \D U}\le (1+\omega(r))\int_{E_r(x_0)}\mean{A(x_0)\D V, \D V}.
$$

\begin{definition}[Almost harmonic property at a point]
We say that a function $u\in W^{1, 2}(B_R)$ satisfies \emph{almost harmonic property at $0$} in $B_R$ if 
$$
\int_{B_r}|\D u|^2\le (1+\omega(r))\int_{B_r}|\D v|^2,
$$
for all $0<r<R$ and $v\in \mathfrak{K}_u(B_r)$.
\end{definition}

\begin{lemma}
If $U$ is an almost $A$-harmonic function in $B_1$ and $x_0\in B_1$
with $E_R(x_0)\subset B_1$, then $u_{x_0}$ satisfies the almost harmonic
property at $0$ in $B_R$.
\end{lemma}

\begin{proof}
The proof is similar to that of Lemma~\ref{lem:u_x-alm-min}.
\end{proof}

\begin{proposition}[cf.~\cite{JeoPet19a}*{Proposition~2.3}]\label{prop:var-Anz-Mor-Camp-est}
Let $U$ be an almost $A$-harmonic function in
$B_1$. Then for any $B_r(x_0)\Subset B_1$ and $0<\rho<r$, we have
\begin{align}\label{eq:var-Anz-Mor-Camp-est1}
  \int_{B_{\rho}(x_0)}|\D U|^2
  &\le C\left[\left(\frac\rho
    r\right)^n+r^\al\right]\int_{B_r(x_0)}|\D U|^2,\\
\label{eq:var-Anz-Mor-Camp-est2}
\int_{B_\rho(x_0)}\left|\D U-\mean{\D U}_{B_\rho(x_0)}\right|^2
  &
\begin{multlined}[t]
    \le C\left(\frac\rho r\right)^{n+2}\int_{B_r(x_0)}\left|\D
      U-\mean{\D U}_{B_r(x_0)}\right|^2\\+Cr^\al\int_{B_r(x_0)}|\D
    U|^2,
  \end{multlined}
\end{align}
with $C=C(n,\alpha,M)$.
\end{proposition}

\begin{proof}
Since $u_{x_0}$ satisfies the almost harmonic property at $0$, if $h$ is the harmonic replacement of $u_{x_0}$ in $B_r$ (i.e.,\ $h$ is harmonic in $B_r$ with $h=u_{x_0}$ on $\pa B_r$), then $$
\int_{B_r}|\D u_{x_0}|^2\le (1+Mr^\al)\int_{B_r}|\D h|^2.
$$
This is enough to repeat the arguments in
\cite{JeoPet19a}*{Proposition~2.3}, to obtain
\begin{align*}
  \int_{B_\rho}|\D u_{x_0}|^2
  &\le 2\left[\left(\frac{\rho}r\right)^n+Mr^\al\right]\int_{B_r}|\D u_{x_0}|^2,\\
    \int_{B_\rho}|\D u_{x_0}-\langle \D u_{x_0}\rangle_{B_\rho}|^2
  &
\begin{multlined}[t]
    \le 9\left(\frac{\rho}r\right)^{n+2}\int_{B_r}|\D u_{x_0}-\mean{\D u_{x_0}}_{B_r}|^2\\+ 24Mr^\al\int_{B_r}|\D u_{x_0}|^2.
  \end{multlined}
\end{align*}
Then, by the change of variables formulas \eqref{eq:u-u_x-1} and
\eqref{eq:u-u_x-2}, we have
\begin{align}\label{eq:var-Anz-Mor-Camp-est3}
  &\int_{E_\rho(x_0)}\langle A(x_0)\D U, \D U\rangle
  \le 2\left[\left(\frac{\rho}r\right)^n+Mr^\al\right]\int_{E_r(x_0)}\langle A(x_0)\D U, \D U\rangle,\\
\label{eq:var-Anz-Mor-Camp-est4}
&\begin{multlined}[t]
  \int_{E_\rho(x_0)}\left|\mfa_{x_0}\D U-\mean{\mfa_{x_0}\D U}_{E_\rho(x_0)}\right|^2\\
    \qquad\qquad\le 9\left(\frac\rho r\right)^{n+2}\int_{E_r(x_0)}\left|\mfa_{x_0}\D U-\mean{\mfa_{x_0}\D U}_{E_r(x_0)}\right|^2\\+24Mr^\al\int_{E_r(x_0)}\langle A(x_0)\D U,\D U\rangle.
  \end{multlined}
\end{align}
To show now that
\eqref{eq:var-Anz-Mor-Camp-est3}--\eqref{eq:var-Anz-Mor-Camp-est4}
imply
\eqref{eq:var-Anz-Mor-Camp-est1}--\eqref{eq:var-Anz-Mor-Camp-est2}, we
first consider the case
$$
0<\rho<(\lambda/\Lambda)^{1/2}r.
$$
Then, using the inclusions
$$
B_\rho(x_0)\subset E_{\lambda^{-
    1/2}\rho}(x_0) \subset E_{\Lambda^{-
    1/2}r}(x_0)\subset B_r(x_0),
$$
applying
\eqref{eq:var-Anz-Mor-Camp-est3}--\eqref{eq:var-Anz-Mor-Camp-est4}
with $\lambda^{-1/2}\rho$ and $\Lambda^{-1/2}r$ in  place of $\rho$
and $r$, and using the ellipticity of $A(x_0)$, we obtain
\eqref{eq:var-Anz-Mor-Camp-est1}--\eqref{eq:var-Anz-Mor-Camp-est2} in
this case.

In the remaining case
$$
(\lambda/\Lambda)^{1/2}r\leq \rho\leq r,
$$
the inequalities
\eqref{eq:var-Anz-Mor-Camp-est1}--\eqref{eq:var-Anz-Mor-Camp-est2}
hold readily, as
\begin{align*}
\int_{B_\rho(x_0)}|\D U|^2&\le\left(\frac\Ld\la\right)^{n/2}\left(\frac\rho r\right)^n\int_{B_r(x_0)}|\D U|^2,\\
    \int_{B_\rho(x_0)}|\D U-\mean{\D U}_{B_\rho(x_0)}|^2 & \le\int_{B_\rho(x_0)}|\D U-\mean{\D U}_{B_r(x_0)}|^2\\
    &\le\left(\frac\Ld\la\right)^{\frac{n+2}2}\left(\frac\rho r\right)^{n+2}\int_{B_r(x_0)}|\D U-\mean{\D U}_{B_r(x_0)}|^2.\qedhere
\end{align*}
\end{proof}

We now recall a useful lemma, the proof of which can be found e.g.\ in
\cite{HanLin97}.
\begin{lemma}\label{lem:var-HL}
  Let $r_0>0$ be a positive number and
  $\vp:(0,r_0)\to (0, \infty)$ a nondecreasing function. Let $a$,
  $\beta$, and $\gamma$ be such that $a>0$, $\gamma >\beta >0$. There
  exist two positive numbers $\e=\e(a,\gamma,\beta)$,
  $c=c(a,\gamma,\beta)$ such that, if
$$
\vp(\rho)\le
a\Bigl[\Bigl(\frac{\rho}{r}\Bigr)^{\gamma}+\e\Bigr]\vp(r)+b\, r^{\be},
$$
for all $\rho$, $r$ with $0<\rho\leq r<r_0$, where $b\ge 0$,
then one also has, still for $0<\rho<r<r_0$,
$$
\vp(\rho)\le
c\Bigl[\Bigl(\frac{\rho}{r}\Bigr)^{\be}\vp(r)+b\rho^{\be}\Bigr].
$$
\end{lemma}

\begin{theorem}\label{thm:var-alm-har-reg}
  Let $U$ be an almost $A$-harmonic function  in $B_1$. Then $U\in
  C^{1,\al/2}(B_1)$ with
  $$
\|U\|_{C^{1,\alpha/2}(K)}\leq C\|U\|_{W^{1,2}(B_1)},
$$
for any $K\Subset B_1$, with $C=C(n,\alpha,M, K)$.
\end{theorem}

\begin{proof}
Let $K\Subset B_1$ and $x_0\in \widetilde{K}:=\{y\in B_1:\dist(y,\pa B_1)\ge r_0\}$, where $r_0=\frac12\dist(K,\pa B_1)$. For $\si\in (0,1)$, a direct application of Lemma~\ref{lem:var-HL} to \eqref{eq:var-Anz-Mor-Camp-est1} gives $$
\int_{B_r(x_0)}|\D U|^2\le C\|\D U\|_{L^2(B_1)}^2r^{n-2+2\si},
$$
for any $0<r<r_0$, with $C$ depending on $n$,
$\al$, $\si$, $M$, $K$. Combining this with
\eqref{eq:var-Anz-Mor-Camp-est2} also gives,
\begin{multline}\label{eq:var-alm-har-reg}
\int_{B_\rho(x_0)}\left|\D U-\mean{\D U}_{B_\rho(x_0)}\right|^2\le C\left(\frac\rho r\right)^{n+2}\int_{B_r(x_0)}\left|\D U-\mean{\D U}_{B_r(x_0)}\right|^2\\+C\|\D U\|_{L^2(B_1)}^2r^{n-2+2\si+\al},
\end{multline}
for any $0<\rho<r<r_0$.
If we take $\si\in (0,1)$ such that $\al':=\frac{-2+2\si+\al}2>0$, then Lemma~\ref{lem:var-HL} produces $$
\int_{B_\rho(x_0)}\left|\D U-\mean{\D U}_{B_\rho(x_0)}\right|^2 \le C\|\D U\|_{L^2(B_1)}^2\rho^{n+2\al'}
$$
and this readily implies $\D U\in C^{0, \al'}(\widetilde{K})$. Now we know that $\D U$ is bounded in $\widetilde{K}$, and thus $\int_{B_r(x_0)}|\D U|^2\le C\|\D U\|_{L^2(B_1)}^2r^n$. Plugging this in the last term of \eqref{eq:var-Anz-Mor-Camp-est2} and repeating the arguments above, we conclude that $U\in C^{1,\al/2}$.
\end{proof}


\section{Almost Lipschitz regularity of almost minimizers}
\label{sec:var-almost-lipsch-regul}

In this section, we make the first step towards the regularity of almost
minimizers for the $A$-Signorini problem and show that they are almost Lipschitz, i.e., $C^{0,\sigma}$ for every
$0<\sigma<1$ (Theorem~\ref{thm:var-holder}). The proof is based on the Morrey space embedding,
similar to the case of almost $A$-harmonic functions, as well as the case
of almost minimizers with $A=I$, treated in \cite{JeoPet19a}. We want to
emphasize, however, that the results on almost Lipschitz and
$C^{1,\beta}$ regularity of almost minimizers (in the next section) do
not require any symmetry condition that was imposed in \cite{JeoPet19a}. 

\medskip
We start with an auxiliary result on the solutions of the Signorini problem.
\begin{proposition}\label{prop:nonsym-sol-Sig-est}
Let $h$ be a solution of the Signorini problem in $B_1$. Then
\begin{equation}\label{eq:nonsym-sol-Sig-est}
\int_{B_\rho}|\D h|^2\le \left(\frac\rho R\right)^{n}\int_{B_R}|\D h|^2,\quad 0<\rho<R<1.
\end{equation}
\end{proposition}
\begin{proof} The difference of this proposition from
  \cite{JeoPet19a}*{Proposition~3.2} is that $h(y)$ is not assumed to be even
  symmetric in $y_n$-variable. To circumvent that, we decompose $h$
  into the sum of even
  and odd functions in $y_n$, i.e.,
  \begin{align}\label{eq:h-split-he-ho}
    h(y',y_n)&=\frac{h(y',y_n)+h(y',-y_n)}2+\frac{h(y',y_n)-h(y',-y_n)}2\\
    &=:h^*(y',y_n)+h^\sharp(y',y_n).\nonumber
\end{align}
It is easy to see that $h^*$ is a solution of the Signorini problem, even in $y_n$-variable, and $h^\sharp$ is a harmonic function, odd in $y_n$-variable.

Then both $|\nabla h^*|^2$ and $|\nabla h^\sharp|^2$ are subharmonic
functions in $B_1$ (see \cite{JeoPet19a}*{Proposition~3.2} for $h^*$), which implies
that for $0<\rho<R<1$
\begin{align*}
  \int_{B_\rho}|\D h^*|^2
  &\le \left(\frac\rho R\right)^{n}\int_{B_R}|\D h^*|^2,\\
   \int_{B_\rho}|\D h^\sharp|^2
  &\le \left(\frac\rho R\right)^{n}\int_{B_R}|\D h^\sharp|^2.
\end{align*}
Now observing that
$\int_{B_t}|\D h|^2=\int_{B_t}\left(|\D h^*|^2+|\D h^\sharp|^2\right)$, for $0<t\leq R$,
we obtain \eqref{eq:nonsym-sol-Sig-est}.
\end{proof}

\begin{proposition}[cf.~\cite{JeoPet19a}*{Proposition~3.3}]\label{prop:var-alm-min-Sig-Mor-est}
Let $U$ be an almost minimizer for the $A$-Signorini problem in $B_1$, and $B_R(x_0)\Subset B_1$. Then, there is
$C_1=C_1(n,M)>1$ such that
\begin{equation}\label{eq:var-alm-min-Sig-Mor-est}
\int_{B_\rho(x_0)}|\D U|^2\le C_1\left[\left(\frac\rho R\right)^{n}+R^\al\right]\int_{B_R(x_0)}|\D U|^2,\quad 0<\rho<R.
\end{equation}
\end{proposition}

\begin{proof}
 \emph{Case 1.} Suppose $x_0\in B'_1$. Note that $u_{x_0}$ satisfies
 the Signorini property at $0$ in $B_r$ with $r=\Lambda^{-1/2}R$. 
 If $h$ is the Signorini replacement of $u_{x_0}$ in $B_r$ (that is, $h$ solves the Signorini problem in $B_r$ with thin obstacle 0 on $\Pi$ and boundary values $h = u_{x_0}$ on $\partial B_r$), then $h$ satisfies 
 $$
 \int_{B_r}\langle{\D h,\D(v-h)}\rangle\ge 0,
 $$
 for any $v\in \mathfrak{K}_{0,u_{x_0}}(B_r,\Pi)$, which easily
 follows from the standard first variation argument. Plugging in
 $v=u_{x_0}$, we obtain
 $$
 \int_{B_r}\langle{\D h,\D u_{x_0}}\rangle\ge\int_{B_r}|\D h|^2.
$$
Then it follows that \begin{align*}
\int_{B_r}|\D(u_{x_0}-h)|^2&=\int_{B_r}\left(|\D u_{x_0}|^2+|\D h|^2-2\langle{\D u_{x_0},\D h}\rangle\right)\\
&\le \int_{B_r}|\D u_{x_0}|^2-\int_{B_r}|\D h|^2\\
&\le (1+Mr^\al)\int_{B_r}|\D h|^2-\int_{B_r}|\D h|^2\\
&\le Mr^\al\int_{B_r}|\D u_{x_0}|^2,
\end{align*}
where in the last inequality we have used that $h$ is the energy
minimizer of the Dirichlet integral in
$\mathfrak{K}_{0,u_{x_0}}(B_r,\Pi)$.
Then, for $\rho\leq r$, we have
\begin{align*}
  \int_{B_\rho}|\D u_{x_0}|^2
  &\le 2\int_{B_\rho}|\D h|^2+2\int_{B_\rho}|\D(u_{x_0}-h)|^2\\
  &\le 2\left(\frac\rho r\right)^n
    \int_{B_r}|\D h|^2+2Mr^\al\int_{B_r}|\D u_{x_0}|^2\\
&\le C\left[\left(\frac\rho r\right)^{n}+r^\al\right]\int_{B_r}|\D u_{x_0}|^2.
\end{align*}
Now, we transform back from $u_{x_0}$ to $U$ as we did in Proposition~\ref{prop:var-Anz-Mor-Camp-est} to obtain
\eqref{eq:var-alm-min-Sig-Mor-est} in this case.

\medskip\noindent \emph{Case 2.} Now consider the case
  $x_0\in B_1^+$. If $\rho\ge r/4$, then we simply have
$$
\int_{B_{\rho}(x_0)}|\D U|^2 \le
4^{n}\left(\frac{\rho}{r}\right)^{n}\int_{B_r(x_0)}|\D U|^2.$$ Thus, we
may assume $\rho<r/4$. Then, let $d:=\dist(x_0, B'_1)>0$ and choose
$x_1\in \pa B_d(x_0)\cap B'_1$.

\medskip\noindent \emph{Case 2.1.} If $\rho\ge d$, then we use
$B_{\rho}(x_0)\subset B_{2\rho}(x_1)\subset B_{r/2}(x_1)\subset
B_r(x_0)$ and the result of Case 1 to write
\begin{align*}
  \int_{B_{\rho}(x_0)}|\D U|^2
  &\le \int_{B_{2\rho}(x_1)}|\D U|^2
    \le C\left[\left(\frac{2\rho}{r/2}\right)^{n} +
    (r/2)^{\al}\right]\int_{B_{r/2}(x_1)}|\D U|^2\\
  &\le
    C\left[\left(\frac{\rho}{r}\right)^{n}+r^{\al}\right]\int_{B_r(x_0)}|\D U|^2.
\end{align*}

\medskip\noindent \emph{Case 2.2.}  Suppose now $d>\rho$. If $d>r$,
then $B_r(x_0)\Subset B_1^+$. Since $U$ is almost harmonic  in $B_1^+$, 
we can apply Proposition~\ref{prop:var-Anz-Mor-Camp-est} to obtain
$$
\int_{B_{\rho}(x_0)}|\D U|^2\le C\left[\left(\frac \rho
    r\right)^n+r^{\al}\right]\int_{B_r(x_0)}|\D U|^2.
$$
Thus, we may assume $d\le r$. Then we note that
$B_d(x_0)\subset B_1^+$ and by a limiting argument from the previous
estimate, we obtain
$$
\int_{B_{\rho}(x_0)}|\D U|^2 \le C\left[\left(\frac \rho
    d\right)^n+r^{\al}\right]\int_{B_d(x_0)}|\D U|^2.
$$
To estimate $\int_{B_d(x_0)}|\D U|^2$ in the right-hand side of the above inequality, we further consider the two subcases.

\medskip\noindent \emph{Case 2.2.1.} If $r/4\le d$, then
$$
\int_{B_d(x_0)}|\D U|^2 \le
4^{n}\left(\frac{d}{r}\right)^{n}\int_{B_r(x_0)}|\D U|^2,
$$ 
which immediately implies \eqref{eq:var-alm-min-Sig-Mor-est}.

\medskip\noindent \emph{Case 2.2.2.}  It remains to consider the case
$\rho<d<r/4$. Using Case 1 again, we have
\begin{align*}
  \int_{B_d(x_0)}|\D U|^2
  &\le \int_{B_{2d}(x_1)}|\D U|^2 \le C\left[\left(\frac{2d}{r/2}\right)^{n}+(r/2)^{\al}\right]\int_{B_{r/2}(x_1)}|\D U|^2 \\
  &\le C\left[\left(\frac{d}{r}\right)^{n}+r^{\al}\right]\int_{B_r(x_0)}|\D U|^2,
\end{align*}
which also implies \eqref{eq:var-alm-min-Sig-Mor-est}.  This concludes the proof of
the proposition.
\end{proof}

As we have seen in \cite{JeoPet19a},
Proposition~\ref{prop:var-alm-min-Sig-Mor-est} implies the almost Lipschitz regularity of almost minimizers. 

\begin{theorem}\label{thm:var-holder}
Let $U$ be an almost minimizer for the $A$-Signorini problem  in $B_1$. Then $U\in C^{0, \si}(B_1)$ for all $0<\si<1$. Moreover, for any $K\Subset B_1$, $$
\|U\|_{C^{0, \si}(K)}\le C\|U\|_{W^{1, 2}(B_1)},
$$
with $C=C(n, \al, M,\si, K)$.
\end{theorem}

\begin{proof}
The proof is essentially identical to that of \cite{JeoPet19a}*{Theorem~3.1}. Let $K\Subset B_1$ and $x_0\in K$. Take $r_0=r_0(n,\al,M,\si,K)>0$ such that $r_0<\dist(K, \pa B_1)$ and $r_0^\al\le\e(C_1, n, n+2\si-2)$, where $\e=\e(C_1, n, n+2\si-2)$ is as in Lemma~\ref{lem:var-HL} and $C_1=C_1(n,M)$ is as in Proposition~\ref{prop:var-alm-min-Sig-Mor-est}. Then for all $0<\rho<r<r_0$, by Proposition~\ref{prop:var-alm-min-Sig-Mor-est}, $$
\int_{B_\rho(x_0)}|\D U|^2\le C_1\left[\left(\frac{\rho}r\right)^{n}+r^\al\right]\int_{B_r(x_0)}|\D U|^2.
$$
By Lemma~\ref{lem:var-HL}, we get $$
\int_{B_\rho(x_0)}|\D U|^2\le C(n,M,\si)\left(\frac\rho r\right)^{n+2\si-2}\int_{B_r(x_0)}|\D U|^2.
$$
Taking $r\nearrow r_0$, we conclude that \begin{equation}\label{eq:var-holder-est}
\int_{B_\rho(x_0)}|\D U|^2 \le C(n,\al,M,\si,K)\|\D U\|^2_{L^2(B_1)}\rho^{n+2\si-2}.
\end{equation}
By the Morrey space embedding
\cite{HanLin97}*{Corollary~3.2}, we obtain $U\in C^{0,\si}(K)$ with \begin{equation}\label{eq:var-holder} \|U\|_{C^{0,\si}(K)}\le C(n,\al,M,\si,K)\|U\|_{W^{1,2}(B_1)}.\qedhere
\end{equation}
\end{proof}


\section{$C^{1,\be}$ regularity of almost minimizers}
\label{sec:var-c1-beta-regularity}

In this section we prove $C^{1,\beta}$ regularity of the almost
minimizers for the $A$-Signorini problem (Theorem~\ref{thm:var-grad-holder}). While we take advantage of
the results available for the even symmetric almost minimizers with
$A=I$ in \cite{JeoPet19a}, removing the symmetry condition requires
new additional steps, combined with ``deskewing'' arguments to
generalize to the variable coefficient case.

\medskip
We start again with an auxiliary result for the solutions of the
Signorini problem.

\begin{proposition}\label{prop:nonsym-sig-est} Let $h$ be a solution of the
  Signorini problem in $B_r$, $0<r<1$.  Define
$$
\wDh:=
\begin{cases}\D h(y', y_n),& y_n\ge 0 \\
  \D h(y', -y_n), &y_n<0,
\end{cases}
$$
the even extension of $\D h$ from $B_r^+$ to $B_r$.  Then
for $0<\al <1$, there are $C_1=C_1(n, \al)$, $C_2=C_2(n, \al)$ such
that for all $0<\rho\leq s\le (3/4)r$,
\begin{multline}\label{eq:nonsym-sig-est}\int_{B_{\rho}}|\wDh-\mean{\wDh}_{B_\rho}|^2
  \le C_1
  \Bigl(\frac{\rho}{s}\Bigr)^{n+\al}\int_{B_s}|\wDh-\mean{\wDh}_{B_
    s}|^2\\+C_2\left(\int_{B_r}h^2\right)\frac{s^{n+1}}{r^{n+3}}.
\end{multline}
\end{proposition}

\begin{proof} This proposition differs from \cite{JeoPet19a}*{Proposition~4.4} only by not requiring $h(y)$ to be even in the
  $y_n$-variable. As in the proof of
  Proposition~\ref{prop:nonsym-sol-Sig-est} we split $h$ into its even
  and odd parts
  $$
h(y)=h^*(y)+h^\sharp(y),\quad y\in B_r.
$$
Recall that $h^*$ is still a solution of the Signorini problem in
$B_r$, but now even in $y_n$ and $h^\sharp$ is a harmonic function in
$B_r$, odd in $y_n$.
Then, by \cite{JeoPet19a}*{Proposition~4.4} we have
\begin{multline}\label{eq:nonsym-sig-est-even}\int_{B_{\rho}}|\widehat{\D
  h^*}-\mean{\widehat{\D h^*}}_{B_\rho}|^2
  \le C_1
  \Bigl(\frac{\rho}{s}\Bigr)^{n+\al}\int_{B_s}|\widehat{\D h^*}
  -\mean{\widehat{\D h^*}}_{B_
    s}|^2\\+C_2\left(\int_{B_r}(h^*)^2\right)\frac{s^{n+1}}{r^{n+3}}.
\end{multline}
Now we need a similar estimate for $h^\sharp$. Since $h^\sharp$ is harmonic, by
the standard interior estimates, we have
$$
\sup_{B_{(3/4)r}}|D^2h^\sharp|\leq \frac{C(n)}{r^2}\left(\frac1{r^{n}}\int_{B_r}(h^\sharp)^2\right)^{1/2}.
$$
Thus, taking the averages on $B_\rho^+$, we will therefore have
\begin{align*}
  \int_{B_\rho^+}|\nabla h^\sharp-\mean{\nabla h^\sharp}_{B_\rho^+}|^2
  &\leq C(n)\Bigl(\sup_{B_{\rho}}|D^2h^\sharp|\Bigr)^2\rho^{n+2}\leq
    C(n)\left(\int_{B_r} (h^\sharp)^2\right)\frac{\rho^{n+2}}{r^{n+4}}\\
  & \nonumber\leq C(n)\left(\int_{B_r}
    (h^\sharp)^2\right)\frac{s^{n+1}}{r^{n+3}},\quad 0<\rho<s\leq (3/4)r,
\end{align*}
which can be rewritten as
\begin{equation}\label{eq:nonsym-sig-est-odd}\int_{B_{\rho}}|\widehat{\D h^\sharp}-\mean{\widehat{\D h^\sharp}}_{B_\rho}|^2
  \le C(n)\left(\int_{B_r}(h^\sharp)^2\right)\frac{s^{n+1}}{r^{n+3}}.
\end{equation}
Now using that $\widehat{\D h}-\mean{\widehat{\D h}}_{B_\rho}=[\widehat{\D h^*}-\mean{\widehat{\D h^*}}_{B_\rho}]
+[\widehat{\D h^\sharp}-\mean{\widehat{\D h^\sharp}}_{B_\rho}]$ in $B_\rho$, we
deduce from \eqref{eq:nonsym-sig-est-odd} that
\begin{align}\label{eq:nonsym-sig-est-odd-2}\int_{B_{\rho}}|\widehat{\D
    h}-\mean{\widehat{\D h}}_{B_\rho}|^2
  &\le 2\int_{B_{\rho}}|\widehat{\D h^*}-\mean{\widehat{\D h^*}}_{B_\rho}|^2+2\int_{B_{\rho}}|\widehat{\D h^\sharp}-\mean{\widehat{\D h^\sharp}}_{B_\rho}|^2\\
  &\le 2\int_{B_{\rho}}|\widehat{\D h^*}-\mean{\widehat{\D h^*}}_{B_\rho}|^2+C(n)\left(\int_{B_r}(h^\sharp)^2\right)\frac{s^{n+1}}{r^{n+3}}.\nonumber
\end{align}
Similarly, representing  $\widehat{\D h^*}-\mean{\widehat{\D
    h^*}}_{B_s}=[\widehat{\D h}-\mean{\widehat{\D
    h}}_{B_s}]-[\widehat{\D h^\sharp}-\mean{\widehat{\D h^\sharp}}_{B_s}]$ in $B_s$,
we deduce from \eqref{eq:nonsym-sig-est-odd} (by taking $\rho=s$) that
\begin{align}\label{eq:nonsym-sig-est-odd-3}\int_{B_{s}}|\widehat{\D
    h^*}-\mean{\widehat{\D h^*}}_{B_s}|^2
  &\le 2\int_{B_{s}}|\widehat{\D h}-\mean{\widehat{\D h}}_{B_s}|^2+C(n)\left(\int_{B_r}(h^\sharp)^2\right)\frac{s^{n+1}}{r^{n+3}}.
\end{align}
Hence, combining
\eqref{eq:nonsym-sig-est-even}--\eqref{eq:nonsym-sig-est-odd-3}, and
using that both $\int_{B_r} (h^*)^2$ and $\int_{B_r} (h^\sharp)^2$ cannot
exceed $\int_{B_r} h^2$, we obtain the claimed estimate \eqref{eq:nonsym-sig-est}. 
\end{proof}

\begin{theorem}\label{thm:var-grad-holder}
  Let $U$ be an almost minimizer of the $A$-Signorini problem  in $B_1$. Then $$
U\in C^{1,\be}(B_1^\pm\cup B_1')\quad\text{with }\be=\frac\al{4(2n+\al)}.
$$
Moreover, for any $K\Subset B_1^\pm\cup B'_1$, we have \begin{equation}\label{eq:var-grad-holder}
\|U\|_{C^{1,\be}(K)}\le C(n,\al,M,K)\|U\|_{W^{1,2}(B_1)}.
\end{equation}
\end{theorem}

\begin{proof}
Let $K$ be a ball centered at $0$. Fix a small $r_0=r_0(n,\al,M,K)>0$ to be determined later. In particular, we will ask $r_1:=r_0^{\frac{2n}{2n+\al}}\Lambda^{1/2}\le (1/2)\dist(K,\pa B_1)$, which implies that $$
\widetilde{K}:=\{y\in B_1:\dist(y,K)\le r_1\}\Subset B_1.
$$
Define $$
\wDU(y',y_n):=\begin{cases}\D U(y',y_n), &y_n\ge 0\\ \D U(y',-y_n), &y_n<0.\end{cases}
$$
Our goal is to show that for $x_0\in K$, $0<\rho<r<r_0$,
\begin{multline}\label{eq:var-grad-holder-estimate}
 \int_{B_{\rho}(x_0)}|\wDU-\mean{\wDU}_{B_\rho(x_0)}|^2 
\\\le
C(n,\al,M)\left(\frac{\rho}{r}\right)^{n+\al}\int_{B_r(x_0)}|\wDU-\mean{\wDU}_{B_r(x_0)}|^2\\+C(n,
\al, M, K)\|U\|^2_{W^{1, 2}(B_1)}r^{n+2\be}.
\end{multline}

\medskip\noindent\emph{Case 1.} Suppose $x_0\in K\cap B_1'$. For given $0<r<r_0$, we denote $\al':=1-\frac\al{8n}\in (0, 1)$, $R:=r^{\frac{2n}{2n+\al}}$. We then consider two cases: $$
\sup_{\pa E_R(x_0)}|U|\le C_3(\Ld^{1/2}R)^{\al'}\quad\text{and}\quad \sup_{\pa E_R(x_0)}|U|>C_3(\Ld^{1/2}R)^{\al'},
$$
where $C_3=2[U]_{0,\al',\widetilde{K}}=2\sup_{\substack{y, z\in \K\\ y\neq z}}\frac{|U(y)-U(z)|}{|y-z|^{\al'}}$.

\medskip\noindent \emph{Case 1.1.} Assume that $\sup_{\pa
  E_R(x_0)}|U|\le C_3(\Ld^{1/2}R)^{\al'}$. Then $u_{x_0}$ satisfies
almost Signorini property at $0$ in $B_R$ with
$$
\sup_{\pa B_R}|u_{x_0}|\le C_3(\Ld^{1/2}R)^{\al'}.
$$
Let $h$ be the Signorini replacement of $u_{x_0}$ in $B_R$. If we define
\begin{align*}
\widehat{\D u_{x_0}}(y',y_n)&:=\begin{cases}\D u_{x_0}(y',y_n),& y_n\ge 0\\ \D
  u_{x_0}(y',-y_n),& y_n<0\end{cases}\\
  \intertext{and}
  \wDh(y',y_n)&:=\begin{cases}\D h(y',y_n),& y_n\ge 0\\ \D h(y',-y_n),& y_n<0,\end{cases}
\end{align*}
then we have
\begin{align}\label{u}
\int_{B_\rho}|\widehat{\D u_{x_0}}-\mean{\widehat{\D u_{x_0}}}_{B_\rho}|^2&\le 3\int_{B_\rho}|\wDh-\mean{\wDh}_{B_\rho}|^2+6\int_{B_\rho}|\widehat{\D u_{x_0}}-\wDh|^2,\\
\label{h}
\int_{B_r}|\wDh-\mean{\wDh}_{B_r}|^2&\le 3\int_{B_r}|\widehat{\D u_{x_0}}-\mean{\widehat{\D u_{x_0}}}_{B_r}|^2+6\int_{B_r}|\widehat{\D u_{x_0}}-\wDh|^2.
\end{align}
Note that if $r_0\le (3/4)^{\frac{2n+\al}\al}$, then $r< (3/4)R$, thus by Proposition~\ref{prop:nonsym-sig-est}, the Signorini replacement $h$ satisfies, for $0<\rho<r$,
\begin{multline*}\int_{B_{\rho}}|\wDh-\mean{\wDh}_{B_\rho}|^2\le C(n,\al)\Bigl(\frac{\rho}{r}\Bigr)^{n+\al}\int_{B_r}|\wDh-\mean{\wDh}_{B_r}|^2\\+C(n,\al)\left(\sup_{\pa B_R}h^2\right)\frac{r^{n+1}}{R^3}.
\end{multline*}
Combining the above three inequalities, we obtain
\begin{multline}\label{eq:var-grad-holder-est1}
\int_{B_\rho}|\widehat{\D u_{x_0}}-\mean{\widehat{\D u_{x_0}}}_{B_\rho}|^2
\le  C(n,\al)\left(\frac\rho r\right)^{n+\al}\int_{B_r}|\widehat{\D u_{x_0}}-\mean{\widehat{\D u_{x_0}}}_{B_r}|^2\\
+C(n,\al)\left(\sup_{\pa B_R}h^2\right)\frac{r^{n+1}}{R^3}+C(n,\al)\int_{B_r}|\widehat{\D u_{x_0}}-\wDh|^2.
\end{multline}
Let us estimate the last term in the right-hand side of \eqref{eq:var-grad-holder-est1}.
Take $\de=\de(n,\al,M,K)>0$ such that $\de<\dist(K,\pa B_1)$ and $\de^\al\le\e=\e(C_1,n,n+2\al'-2)$, where $C_1=C_1(n,M)$ is as in Proposition~\ref{prop:var-alm-min-Sig-Mor-est} and $\e$ is as in Lemma~\ref{lem:var-HL}. If $r_0\le\left(\Ld^{-1/2}\de\right)^{\frac{2n+\al}{2n}}$, then $\Ld^{1/2}R<\de$, thus, by following the proof of Theorem~\ref{thm:var-holder} up to \eqref{eq:var-holder-est}, we have $$
\int_{B_{\Ld^{1/2}R}(x_0)}|\D U|^2\le C(n,\al, M, K)\|\D U\|_{L^2(B_1)}^2\left(\Ld^{1/2}R\right)^{n+2\al'-2}.
$$
It follows that \begin{align*}
\int_{E_R(x_0)}\langle A(x_0)\D U,\D U\rangle &\le
                                                \Ld\int_{B_{\Ld^{1/2}R}(x_0
                                                )}|\D U|^2\\
&\le C\|\D U\|_{L^2(B_1)}^2R^{n+2\al'-2}.
\end{align*}
Then by the change of variables \eqref{eq:u-u_x-1}, we have 
\begin{align}\label{eq:var-grad-holder-u_x-u-est}
\int_{B_R}|\D u_{x_0}|^2\le C\|\D U\|^2_{L^2(B_1)}R^{n+2\al'-2}.
\end{align}
Now we can estimate the third term in the right-hand side of
\eqref{eq:var-grad-holder-est1}:
\begin{multline}\label{eq:var-u_x-h-est}
\int_{B_r}|\widehat{\D u_{x_0}}-\wDh|^2  = 2\int_{B^+_r}|\D u_{x_0}-\D h|^2\\
\begin{aligned}
&\le 2\int_{B_{R}}|\D u_{x_0}-\D h|^2\le 2\left(\int_{B_R}|\D u_{x_0}|^2-\int_{B_R}|\D h|^2\right)\\
&\le 2MR^{\al}\int_{B_{R}}|\D h|^2\le 2MR^{\al}\int_{B_{R}}|\D u_{x_0}|^2\\
&\le  C\|\D U\|^2_{L^2(B_1)}R^{n+\al+2\al'-2} \\
&= C\|\D U\|^2_{L^2(B_1)}r^{n+\frac{\al}{2n+\al}(n-\frac{1}{2})}.
\end{aligned}
\end{multline}
To estimate the second term in the right-hand side of \eqref{eq:var-grad-holder-est1}, we observe that $$
\sup_{\pa B_R}h^2=\sup_{\pa B_R}u_{x_0}^2=\sup_{\pa E_R(x_0)}U^2\le C_3^2(\Ld^{1/2}R)^{2\al'}.
$$
Note that by \eqref{eq:var-holder}, $C_3\le C(n,\al,M,K)\|U\|_{W^{1,2}(B_1)}$. Thus, 
\begin{align*}
\left(\sup_{\pa B_R}h^2\right)\frac{r^{n+1}}{R^{3}}&\le C\|U\|_{W^{1,2}(B_1)}^2r^{n+\frac{\al}{2(2n+\al)}}.
\end{align*}
Now \eqref{eq:var-grad-holder-est1} becomes
\begin{multline}\label{eq:var-grad-holder-est2}
\int_{B_\rho}|\widehat{\D u_{x_0}}-\mean{\widehat{\D u_{x_0}}}_{B_\rho}|^2
\le C(n,\al)\left(\frac\rho
  r\right)^{n+\al}\int_{B_r}|\widehat{\D u_{x_0}}-\mean{\widehat{\D u_{x_0}}}_{B_r}|^2\\+C\|U\|^2_{W^{1,
    2}(B_1)}r^{n+\frac{\al}{2(2n+\al)}}.
\end{multline}
We now want to deduce \eqref{eq:var-grad-holder-estimate} from
\eqref{eq:var-grad-holder-est2}. The complication here is that the mapping $\bar{T}_{x_0}^{-1}$ does not preserve the even
symmetry with respect to the thin plane, since the conormal direction
$A(x_0)e_n$ might be different from the normal direction $e_n$ to
$\Pi$ at $x_0$.
To address this issue, by using the even symmetry of $\widehat{\D u_{x_0}}$, we
rewrite \eqref{eq:var-grad-holder-est2} in terms of halfballs $B_r^+=B_r\cap \R^n_+$
\begin{multline}\label{eq:var-grad-holder-est3}
        \int_{B_\rho^+}|\D u_{x_0}-\mean{\D u_{x_0}}_{B_\rho^+}|^2
\le C(n,\al)\left(\frac\rho r\right)^{n+\al}\int_{B_r^+}|\D u_{x_0}-\mean{\D u_{x_0}}_{B_r^+}|^2\\+C\|U\|^2_{W^{1, 2}(B_1)}r^{n+\frac{\al}{2(2n+\al)}}.
\end{multline}
Similarly, if we denote $E_r^+(x_0)=E_r(x_0)\cap \R^n_+$, then using
that $\bar{T}_{x_0}(E_t^+(x_0))=B_t^+$, $t>0$, 
\eqref{eq:var-grad-holder-est3}
becomes
\begin{multline*} 
        \int_{E_\rho^+(x_0)}|\mfa_{x_0}\D U-\mean{\mfa_{x_0}\D U}_{E_\rho^+(x_0)}|^2
\\\le C(n,\al)\left(\frac\rho r\right)^{n+\al}\int_{E_r^+(x_0)}|\mfa_{x_0}\D U-\mean{\mfa_{x_0}\D U}_{E_r^+(x_0)}|^2\\
+C\det \mfa_{x_0}\|U\|^2_{W^{1, 2}(B_1)}r^{n+\frac{\al}{2(2n+\al)}}.
\end{multline*}
Repeating the argument that \eqref{eq:var-Anz-Mor-Camp-est4} implies \eqref{eq:var-Anz-Mor-Camp-est2} in the proof of Proposition~\ref{prop:var-Anz-Mor-Camp-est}, we have
\begin{multline}\label{eq:var-grad-holder-est4}
  \int_{B_\rho^+(x_0)}|\D U-\mean{\D U}_{B_\rho^+(x_0)}|^2\le C\left(\frac\rho r\right)^{n+\al}\int_{B_r^+(x_0)}|\D U-\mean{\D U}_{B_r^+(x_0)}|^2\\+C\|U\|^2_{W^{1, 2}(B_1)}r^{n+\frac{\al}{2(2n+\al)}}.
\end{multline}
      Then by the even symmetry of $\wDU$, \eqref{eq:var-grad-holder-est4} implies \eqref{eq:var-grad-holder-estimate}. 

\medskip\noindent \emph{Case 1.2.} Now we assume that $\sup_{\pa
  E_R(x_0)}|U|> C_3(\Ld^{1/2}R)^{\al'}$. By the choice of
$C_3=2[U]_{0,\al',\widetilde{K}}$, we have either
\begin{align*}
&U\ge \left(C_3/2\right)(\Ld^{1/2}R)^{\al'}\quad\text{in }E_R(x_0),\
  \text{or}\\
&U\le -\left(C_3/2\right)(\Ld^{1/2}R)^{\al'}\quad\text{in } E_R(x_0).
\end{align*}
However, from $U\ge 0$ on $B_1'$, the only possibility is $$
U\ge \left(C_3/2\right)(\Ld^{1/2}R)^{\al'}\quad\text{in } E_R(x_0).
$$
Consequently,
$$
u_{x_0}\ge \left(C_3/2\right)(\Ld^{1/2}R)^{\al'}\quad\text{in } B_R.
$$
If we let $h$ again be the Signorini replacement of $u_{x_0}$ in $B_R$, then the positivity of $h=u_{x_0}>0$ on $\pa B_R$ and superharmonicity of $h$ in $B_R$ give that $h>0$ in $B_R$, and hence $h$ is harmonic in $B_R$. Thus, $$
\int_{B_\rho}|\D h-\mean{\D h}_{B_\rho}|^2\le\left(\frac\rho r\right)^{n+2}\int_{B_r}|\D h-\mean{\D h}_{B_r}|^2,\quad 0<\rho<r.
$$
We next decompose $h=h^*+h^\sharp$ in $B_R$ as in \eqref{eq:h-split-he-ho}. Note
that since both $h$ and $h^\sharp$ are harmonic, $h^*$ must be harmonic as well.
Then we have
\begin{align*}
 \int_{B_\rho}|\wDh-\mean{\wDh}_{B_\rho}|^2&\le 3\int_{B_\rho}|\D h-\mean{\D h}_{B_\rho}|^2+6\int_{B_\rho}|\wDh-\D h|^2\\
 &= 3\int_{B_\rho}|\D h-\mean{\D h}_{B_\rho}|^2+6\int_{B^-_\rho}\left(|2\D_{y'}h^\sharp|^2+|2\pa_{y_n}h^*|^2\right)\\
 &= 3\int_{B_\rho}|\D h-\mean{\D h}_{B_\rho}|^2+12\int_{B_\rho}\left(|\D_{y'}h^\sharp|^2+|\pa_{y_n}h^*|^2\right),
\end{align*}
and similarly,
\begin{align*}
 \int_{B_r}|\D h-\mean{\D h}_{B_r}|^2 \le 3\int_{B_r}|\wDh-\mean{\wDh}_{B_r}|^2+12\int_{B_r}\left(|\D_{y'}h^\sharp|^2+|\pa_{y_n}h^*|^2\right).
 \end{align*}
Combining the above three inequalities, we have that for all
$0<\rho<r$
\begin{multline}\label{eq:var-grad-holder-h-est1}\int_{B_\rho}|\wDh-\mean{\wDh}_{B_\rho}|^2 \le 3\left(\frac\rho r\right)^{n+2}\int_{B_r}|\wDh-\mean{\wDh}_{B_r}|^2\\ +48\int_{B_r}\left(|\D_{y'}h^\sharp|^2+|\pa_{y_n}h^*|^2\right).
\end{multline}
Now, note that if $r_0\le \left(1/2\right)^{\frac{2n+\al}\al}$, then
$r\le R/2$. By the harmonicity
of both $h^*$ and $h^\sharp$ in $B_R$, we have
\begin{align*}
 \sup_{B_{R/2}}|D^2h^*|+\sup_{B_{R/2}}|D^2h^\sharp|
&\le \frac{C(n)}R\left(\sup_{B_{(3/4)R}}|\D h^*|+\sup_{B_{(3/4)R}}|\D h^\sharp|\right)\\
&\le \frac{C(n)}{R^{1+\frac n 2}}\left(\int_{B_R}|\D h^*|^2+\int_{B_R}|\D h^\sharp|^2\right)^{1/2}\\
&= \frac{C(n)}{R^{1+\frac n 2}}\left(\int_{B_R}|\D h|^2\right)^{1/2}\le \frac{C(n)}{R^{1+\frac n 2}}\left(\int_{B_R}|\D u_{x_0}|^2\right)^{1/2}\\
&\le C(n,\al,M,K)\|\D U\|_{L^2(B_1)}R^{\al'-2},
\end{align*}
where the last inequality follows from \eqref{eq:var-grad-holder-u_x-u-est}. Also, note that $\D_{y'}h^\sharp=\pa_{y_n}h^*=0$ on $B'_{R/2}$. Thus, for $y=(y',y_n)\in B_r$, we have
\begin{align*}
|\D_{y'}h^\sharp|+|\pa_{y_n}h^*|&\le |y_n|\left(\sup_{B_{R/2}}|D^2h^*|+\sup_{B_{R/2}}|D^2h^\sharp|\right)\\
&\le C\|\D U\|_{L^2(B_1)}rR^{\al'-2}\\
&=C\|\D U\|_{L^2(B_1)}r^{1+\frac{2n}{2n+\al}(\al'-2)},
\end{align*}
with $C=(n,\al,M,K)$.
Hence, it follows that
\begin{align}\label{eq:var-grad-holder-h-est2}
\int_{B_r}|\D_{y'}h^\sharp|^2+|\pa_{y_n}h^*|^2 &\le C\|\D U\|_{L^2(B_1)}^2r^{n+2+\frac{4n}{2n+\al}(\al'-2)}\\
&\le C\|\D U\|_{L^2(B_1)}^2r^{n+\frac\al{2(2n+\al)}}.\nonumber
\end{align}
Combining \eqref{eq:var-grad-holder-h-est1} and
\eqref{eq:var-grad-holder-h-est2}, we obtain
\begin{multline}\label{eq:var-grad-holder-h-est3}
\int_{B_\rho}|\wDh-\mean{\wDh}_{B_\rho}|^2 \le 3\left(\frac\rho r\right)^{n+2}\int_{B_r}|\wDh-\mean{\wDh}_{B_r}|^2\\
+C\|\D U\|_{L^2(B_1)}^2r^{n+\frac\al{2(2n+\al)}}.
\end{multline}
Note that \eqref{eq:var-u_x-h-est} was induced in Case 1.1 without the use of the assumption $\sup_{\pa E_r(x_0)}|U|\le C_3\left(\Ld^{1/2}R\right)^{\al'}$, so it is also valid in this case. Finally, \eqref{u}, \eqref{h}, \eqref{eq:var-u_x-h-est} and \eqref{eq:var-grad-holder-h-est3} give
\begin{multline*}
  \int_{B_\rho}|\widehat{\D u_{x_0}}-\mean{\widehat{\D u_{x_0}}}_{B_\rho}|^2\\*
  \begin{aligned}
    &\le 3\int_{B_\rho}|\wDh-\mean{\wDh}_{B_\rho}|^2+6\int_{B_\rho}|\widehat{\D u_{x_0}}-\wDh|^2    \\
    &\le 9\left(\frac{\rho}r\right)^{n+2}\int_{B_r}|\wDh-\mean{\wDh}_{B_r}|^2+C\|\D U\|^2_{L^2(B_1)}r^{n+\frac{\al}{2(2n+\al)}}\\
    &\qquad +6\int_{B_\rho}|\widehat{\D u_{x_0}}-\wDh|^2    \\
      &\le 27\left(\frac{\rho}r\right)^{n+2}\int_{B_r}|\widehat{\D u_{x_0}}-\mean{\widehat{\D u_{x_0}}}_{B_r}|^2+C\|\D U\|^2_{L^2(B_1)}r^{n+\frac{\al}{2(2n+\al)}}\\
    &\qquad +60\int_{B_r}|\widehat{\D u_{x_0}}-\wDh|^2    \\
 &\le 27\left(\frac{\rho}r\right)^{n+2}\int_{B_r}|\widehat{\D u_{x_0}}-\mean{\widehat{\D u_{x_0}}}_{B_r}|^2+C\|\D U\|^2_{L^2(B_1)}r^{n+\frac{\al}{2(2n+\al)}}\\
    &\qquad +C \|\D U\|^2_{L^2(B_1)}r^{n+\frac\al{2n+\al}\left(n-1/2\right)}  \\    
      &\le
      27\left(\frac{\rho}r\right)^{n+2}\int_{B_r}|\widehat{\D u_{x_0}}-\mean{\widehat{\D u_{x_0}}}_{B_r}|^2+C\|\D
      U\|^2_{L^2(B_1)}r^{n+\frac{\al}{2(2n+\al)}}.
    \end{aligned}
  \end{multline*}
    As we have seen in Case 1.1, this implies \eqref{eq:var-grad-holder-estimate}. This completes the proof of \eqref{eq:var-grad-holder-estimate} when $x_0\in K\cap B_1'$.

\medskip\noindent \emph{Case 2.} The extension of \eqref{eq:var-grad-holder-estimate} to general $x_0\in K$ follows from the combination of Case 1 and \eqref{eq:var-alm-har-reg}. The argument is the same as Case 2 in the proof of Theorem~4.6 in \cite{JeoPet19a}.

\medskip
Thus, the estimate \eqref{eq:var-grad-holder-estimate} holds in all possible cases.

\medskip

To complete the proof of the theorem, we now apply Lemma~\ref{lem:var-HL}
to the estimate \eqref{eq:var-grad-holder-estimate} to obtain for $0<\rho<r<r_0$
\begin{multline*}
\int_{B_{\rho}(x_0)}|\wDU-\mean{\wDU}_{B_\rho(x_0)}|^2 
\le C\biggl[\Bigl(\frac{\rho}{r}\Bigr)^{n+2\be}\int_{B_{r}(x_0)}|\wDU-\mean{\wDU}_{B_r(x_0)}|^2\\+\|U\|_{W^{1, 2}(B_1)}^2\rho^{n+2\be}\biggr].
\end{multline*}
Taking $r\nearrow r_0=r_0(n,\al,M,K)$, we have
\begin{align*}
  \int_{B_{\rho}(x_0)}|\wDU-\mean{\wDU}_{B_\rho(x_0)}|^2 &\le C\|U\|_{W^{1, 2}(B_1)}^2\rho^{n+2\be},
\end{align*}
with $C=C(n, \al,M, K)$.
Then by the Campanato space embedding this readily implies that
$\wDU\in C^{0,\beta}(K)$
with 
\begin{equation*}
  \|\wDU\|_{C^{0,\be}(K)}\le C\|U\|_{W^{1, 2}(B_1)}.
\end{equation*}
Since $\wDU=\D U$ in $B_1^+\cup B_1'$, we therefore conclude that
$$
U\in C^{1, \be}(K\cap(B_1^+\cup B'_1)),
$$
and combining with the bound in Theorem~\ref{thm:var-holder}, we also
deduce that
\begin{equation*}
  \|U\|_{C^{1,\be}(K\cap(B_1^+\cup B'_1))}\le C(n, \al,M, K)\|U\|_{W^{1, 2}(B_1)}.
\end{equation*}
To see the $C^{1,\be}$ regularity of $U$ in $B_1^-\cup B_1'$, we
simply observe that the function $U(y',-y_n)$ is also an almost minimizer of the
Signorini problem with the appropriately modified coefficient matrix
$A$.
\end{proof}


\section{Quasisymmetric almost minimizers}

In the study of the free boundary in the Signorini problem, the even
symmetry of the minimizer with respect to the thin space plays a
crucial role. The even symmetry guarantees that the growth rate of the
minimizer $u$ over ``thick'' balls $B_r(x_0)\subset\R^n$ matches the growth rate
over thin balls $B_r'(x_0)\subset\Pi$. This allows to use tools such as
Almgren's monotonicity formula (see the next section) to classify the
free boundary 
points. Without even symmetry, minimizers may have an odd
component, vanishing on the thin space $\Pi$ that may create a mismatch of growth
rates on the thick and thin spaces.

In the case of minimizers of the Signorini problem (with $A=I$) or
harmonic functions, it is easy to see that the even symmetrization 
$$
u^*(x)=\frac{u(x',x_n)+u(x',-x_n)}2
$$
is still a minimizer. Unfortunately, the even symmetrization
may destroy the almost minimizing property, as well as the minimizing property with variable coefficients, as can be seen from the
following simple example.

\begin{example}\label{ex:counter-ex-alm-min}
Let $u:(-1,1)\to \R$ be defined by $u(x)=x+x^2/4$. Then $u$ is an
almost harmonic function in $(-1,1)$ with a gauge function
$\omega(r)=C(\alpha)r^\alpha$ for $0<\alpha<1$. In fact, $u$ is a
minimizer of the energy functional
$$
\int (1+x/2)^{-1} (v')^2
$$
with a Lipschitz function $A(x)=(1+x/2)^{-1}$ in $(-1,1)$. On the other hand, the even symmetrization
$$
u^*(x)=\frac{u(x)+u(-x)}2=\frac{x^2}4
$$
is not almost harmonic for any gauge function $\omega(r)$. Indeed, for
any small $\delta>0$, if we take a competitor $v=\delta^2/4$ in
$(-\delta, \delta)$, then it satisfies
$
\int_{-\delta}^{\delta}|v'|^2=0
$ and if $u^*$ were almost harmonic, we would
have that $\int_{-\delta}^{\delta}|(u^*)'|^2=0$ as well, implying that
$u^*$ is constant in $(-\delta,\delta)$, a contradiction.
\end{example}

To overcome this difficulty, we need to impose the $A$-\emph{quasisymmetry}
condition on almost minimizers $U$, that we have already stated in
Definition~\ref{def:A-quasisym}. In this section, we give more details
on quasisymmetric almost minimizers.

Recall that for each $x_0\in B_1'$, we defined a reflection matrix
$P_{x_0}$ by
$$
P_{x_0}=I-2\frac{A(x_0)e_n \otimes e_n}{a_{nn}(x_0)}.
$$
From the ellipticity of $A$, we have $a_{nn}(x_0)\ge\la$, thus
$P_{x_0}$ is well-defined. Note that $P_{x_0}^2=I$. Besides,
$P_{x_0}\big|_\Pi = I\big|_\Pi$ and $P_{x_0}E_r(x_0)=E_r(x_0)$.
We then define the ``skewed'' even/odd symmetrizations of the almost
minimizer $U$ in $B_1$ by
\begin{align*}
    U_{x_0}^*(x)&:=\frac{U(x)+U(P_{x_0}x)}2,\\
    U_{x_0}^\sharp(x)&:=\frac{U(x)-U(P_{x_0}x)}2.
\end{align*}
\begin{figure}[t]
\centering
\begin{picture}(250,125)(0,0)
  \put(0,0){\includegraphics[height=125pt]{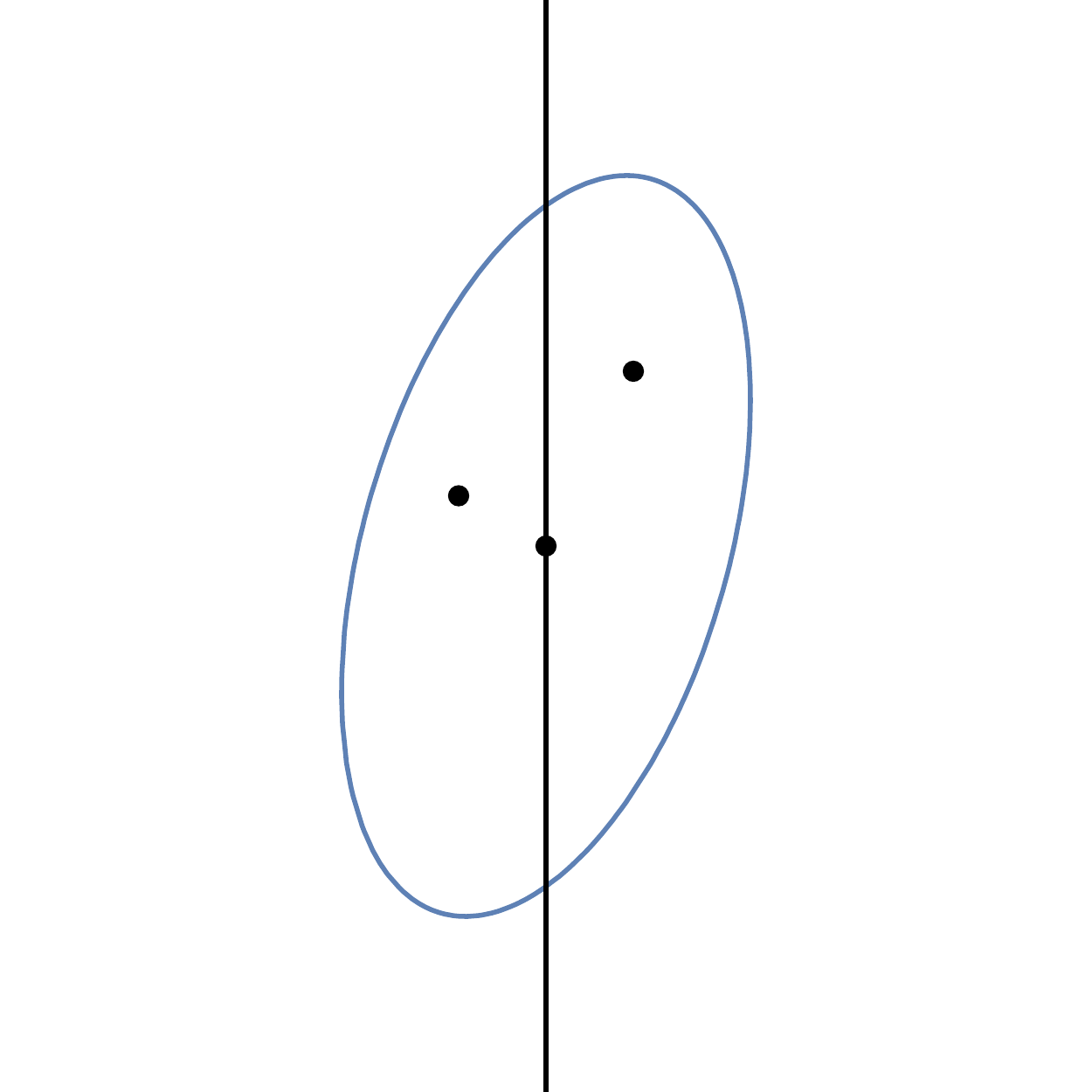}}
  \put(125,0){\includegraphics[height=125pt]{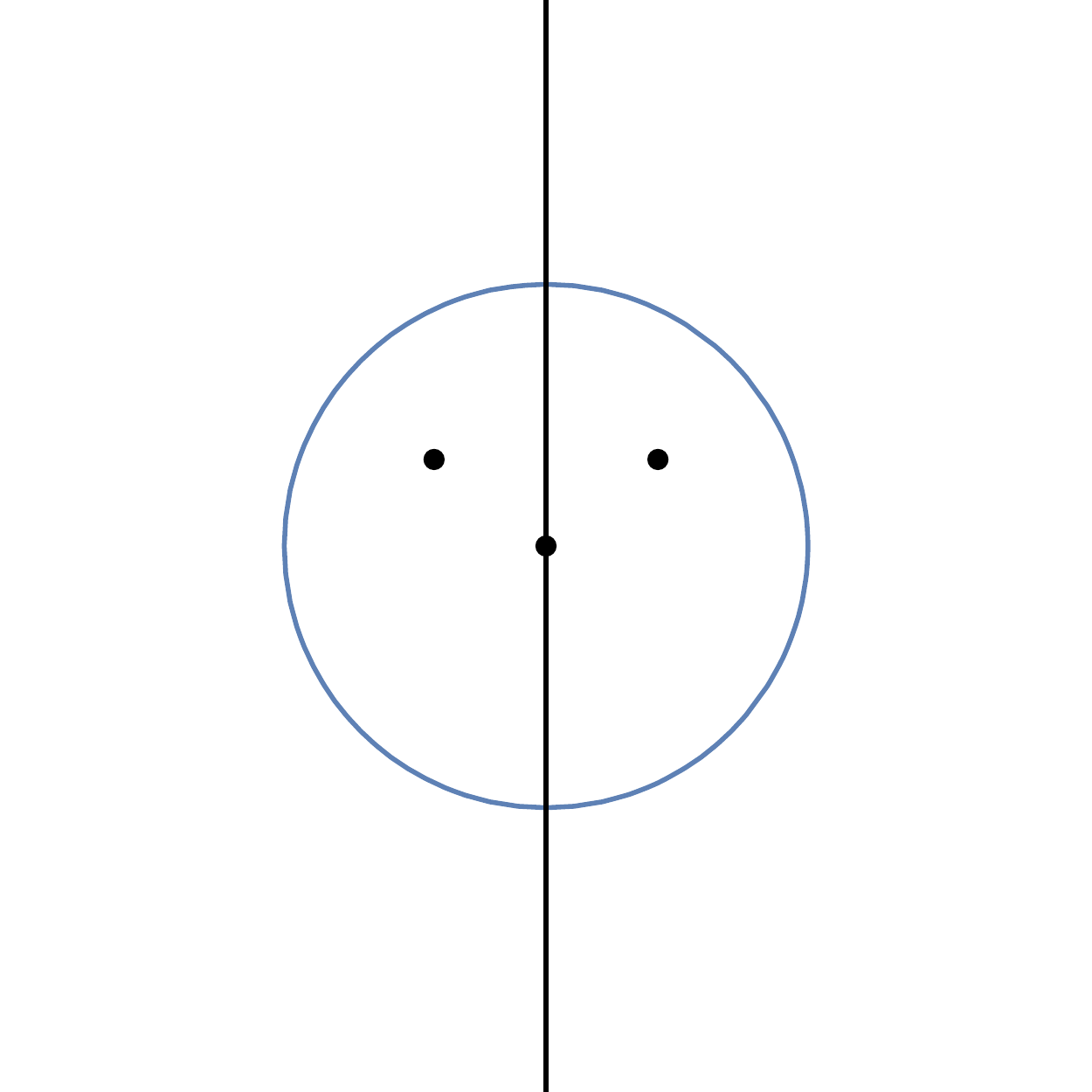}}
  \put(91,83){$\xymatrix{\ar@/^/[rr]^{\bar{T}_{x_0}}_{} & & } $}
  \put(10,60){\scriptsize $E_r(x_0)$}
  \put(225,60){\scriptsize $B_r$}
  \put(50,110){\scriptsize $\Pi$}
  \put(175,110){\scriptsize $\Pi$}
  \put(65,57){\scriptsize $x_0$}
  \put(74,76){\scriptsize $x$}
  \put(47,60){\scriptsize $\bar x$}
  \put(190,57){\scriptsize $0$}
  \put(203,68){\scriptsize $y$}
  \put(166,68){\scriptsize $\bar y$}
\end{picture}
\caption{Reflection $P_{x_0}$: here $\bar x=P_{x_0}x$, $y=\bar{T}_{x_0}(x)$,
  and $\bar y=(y',-y_n)=\bar{T}_{x_0}(\bar{x})$}
\label{fig:refl}
\end{figure}
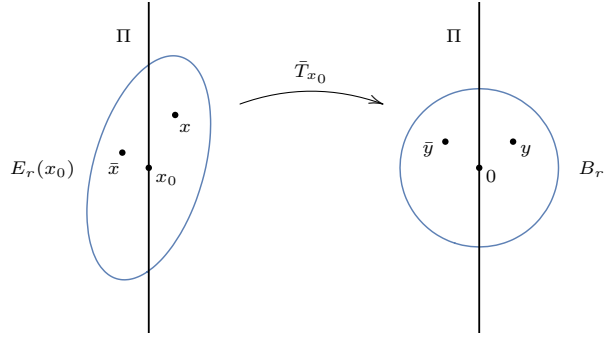%
Note that $U^*_{x_0}$ and $U^\sharp_{x_0}$ may not be defined in all
of $B_1$, but are defined in any ellipsoid $E_r(x_0)$ as long as it is
contained in $B_1$. Note also that $U=U^*_{x_0}$ and
$U^\sharp_{x_0}=0$ on $\Pi$. 
Further, we note that transformed with $\bar{T}_{x_0}$, $P_{x_0}$
becomes an even reflection with respect to $\Pi$, i.e.,
$$
\bar{T}_{x_0}\circ P_{x_0}\circ \bar{T}_{x_0}^{-1}(y)=(y',-y_n),
$$
see Fig~\ref{fig:refl}.
Therefore, denoting
\begin{align*}
  u_{x_0}^*(y)&:=\frac{u_{x_0}(y',y_n)+u_{x_0}(y',-y_n)}2,\\
  u_{x_0}^\sharp(y)&:=\frac{u_{x_0}(y',y_n)-u_{x_0}(y',-y_n)}2,
\end{align*}
the even/odd symmetrizations of $u_{x_0}$ about $\Pi$, we will have
\begin{align*}
    U_{x_0}^*\circ \bar{T}_{x_0}^{-1}=u_{x_0}^*,\qquad U_{x_0}^\sharp\circ \bar{T}_{x_0}^{-1}=u_{x_0}^\sharp.
\end{align*}
We also observe that the symmetries of $u_{x_0}^*$ and
$u_{x_0}^\sharp$ imply the following decompositions
\begin{align}
  \int_{B_r} u_{x_0}^2
  &=\int_{B_r} (u_{x_0}^*)^2+\int_{B_r}(u_{x_0}^\sharp)^2,\\
  \int_{B_r} |\nabla u_{x_0}|^2
  &=\int_{B_r}|\nabla u_{x_0}^*|^2+\int_{B_r} |\nabla u_{x_0}^\sharp|^2,
\end{align}
which after a change of variables, can also be written as
\begin{align}\label{eq:U2-sym-dec}
  \int_{E_r(x_0)} U^2
  &=\int_{E_r(x_0)} (U_{x_0}^*)^2+\int_{E_r(x_0)}(U_{x_0}^\sharp)^2,\\
\label{eq:DU2-sym-dec}
  \int_{E_r(x_0)} \langle A(x_0)\nabla U,\nabla U\rangle
  &=
\begin{multlined}[t]
    \int_{E_r(x_0)}\langle{A(x_0)\nabla U_{x_0}^*},\nabla
    U_{x_0}^*\rangle\\+\int_{E_r(x_0)}\langle{A(x_0)\nabla
      U_{x_0}^\sharp},\nabla U_{x_0}^\sharp\rangle.
  \end{multlined}
\end{align}
We now recall that by Definition~\ref{def:A-quasisym}, $U\in
W^{1,2}(B_1)$ is called
$A$-\emph{quasisymmetric} if there is a constant $Q>0$ such
that
\begin{equation}\label{eq:quasisymmetry}
\int_{E_r(x_0)}\langle A(x_0)\D U,\D U\rangle \le Q\int_{E_r(x_0)}\langle A(x_0)\D U_{x_0}^*,\D U_{x_0}^*\rangle,
\end{equation}
whenever $E_r(x_0)\Subset B_1$ and $x_0\in B_1'$.
By the uniform ellipticity of $A$, \eqref{eq:quasisymmetry} is equivalent to $$
  \int_{E_r(x_0)}|\D U|^2\le Q\int_{E_r(x_0)}|\D U_{x_0}^*|^2,
  $$
  by changing $Q$ to $Q(\Lambda/\lambda)$, if necessary. Besides, using \eqref{eq:DU2-sym-dec}, 
  \eqref{eq:quasisymmetry} is also equivalent to
 \begin{equation}\label{eq:Q-sym-alt}
  \int_{E_r(x_0)}\langle A(x_0)\D U_{x_0}^\sharp,\D
  U_{x_0}^\sharp\rangle \le C\int_{E_r(x_0)}\langle A(x_0)\D
  U_{x_0}^*,\D U_{x_0}^*\rangle,
\end{equation}
with some $C=C(Q)$.

\begin{lemma}\label{lem:quasisym-alm-min}
Let $U$ be an $A$-quasisymmetric almost minimizer for the
$A$-Signorini problem  in $B_1$, with constant $Q>0$. Then there are
$r_1=r_1(n,\al,M,Q)>0$ and $M_1=M_1(n,M,Q)>0$ such that
\begin{align}\label{eq:e_u-alm-min-property}
\int_{E_r(x_0)}\langle  A(x_0)\D U_{x_0}^*,\D U_{x_0}^*\rangle\le(1+M_1r^\al)\int_{E_r(x_0)}\langle A(x_0)\D W,\D W\rangle,
\end{align}
whenever $E_r(x_0)\Subset B_1$, $x_0\in B_1'$, $0<r<r_1$, and $W\in \mathfrak{K}_{0,U_{x_0}^*}(E_r(x_0),\Pi)$ .
\end{lemma}
\begin{remark} Since we are interested in local results, in what
  follows, we will assume without loss of generality that $r_1=1$ and $M_1=M$.
\end{remark}

\begin{proof}
Let $V$ be the energy minimizer of $$
\int_{E_r(x_0)}\langle A(x_0)\D V,\D V\rangle\quad\text{on } \mathfrak{K}_{0,U}(E_r(x_0),\Pi).
$$
Then $v_{x_0}=V\circ \bar{T}_{x_0}^{-1}$ is the energy minimizer of $$
\int_{B_r}|\D v_{x_0}|^2\quad\text{on } \mathfrak{K}_{0,u_{x_0}}(B_r,\Pi).
$$
Note that $v_{x_0}^*$ is a solution of the Signorini problem, even in $y_n$, with $v_{x_0}^*=u_{x_0}^*$ on $\pa B_r$. Similarly, $v_{x_0}^\sharp$ is a harmonic function, odd in $y_n$, with $v_{x_0}^\sharp=u_{x_0}^\sharp$ on $\pa B_r$. Thus, $v_{x_0}^*$ is the energy minimizer of $$
\int_{B_r}|\D v_{x_0}^*|^2\quad\text{on } \mathfrak{K}_{0,u_{x_0}^*}(B_r,\Pi), 
$$
and so $V_{x_0}^*$ is the energy minimizer of  $$
\int_{E_r(x_0)}\langle A(x_0)\D V_{x_0}^*,\D V_{x_0}^*\rangle\quad\text{on } \mathfrak{K}_{0,U_{x_0}^*}(E_r(x_0),\Pi).
$$
Thus, to show \eqref{eq:e_u-alm-min-property}, it is enough to show $$
\int_{B_r}|\D u_{x_0}^*|^2\le (1+M_1r^\al)\int_{B_r}|\D v_{x_0}^*|^2.
$$
To this end, we first observe that the quasisymmetry of $U$ implies the quasisymmetry of $u_{x_0}$: $$
\int_{B_r}|\D u_{x_0}^\sharp|^2\le C\int_{B_r}|\D u_{x_0}^*|^2.
$$
Using this, together with the symmetry of $u_{x_0}^*$, $u_{x_0}^\sharp$, $v_{x_0}^*$ and $v_{x_0}^\sharp$, we have
\begin{align*}
    \int_{B_r}|\D u_{x_0}^*|^2
    &=\int_{B_r}|\D u_{x_0}|^2-\int_{B_r}|\D u_{x_0}^\sharp|^2\\
    &\qquad\le (1+Mr^\al)\int_{B_r}|\D v_{x_0}|^2-\int_{B_r}|\D u_{x_0}^\sharp|^2\\
    &\qquad=(1+Mr^\al)\int_{B_r}|\D v_{x_0}^*|^2+(1+Mr^\al)\int_{B_r}|\D v_{x_0}^\sharp|^2-\int_{B_r}|\D u_{x_0}^\sharp|^2\\
    &\qquad\le(1+M r^\al)\int_{B_r}|\D v_{x_0}^*|^2+Mr^\al\int_{B_r}|\D u_{x_0}^\sharp|^2\\
    &\qquad\le (1+Mr^\al)\int_{B_r}|\D v_{x_0}^*|^2+CMr^\al\int_{B_r}|\D u_{x_0}^*|^2.
\end{align*}
Therefore, \begin{align*}
    \int_{B_r}|\D u_{x_0}^*|^2\le \frac{1+Mr^\al}{1-CMr^\al}\int_{B_r}|\D v_{x_0}^*|^2\le (1+M_1r^\al)\int_{B_r}|\D v_{x_0}^*|^2,
\end{align*}
for $0<r<r_1=(2CM)^{-1/\al}$, as desired.
\end{proof}

\begin{remark}
If $U$ satisfies the following \emph{weak quasisymmetry} with order $-\g$: $$
\int_{E_r(x_0)}|\D U|^2\le Q\,r^{-\g}\int_{E_r(x_0)}|\D U_{x_0}^*|^2,
$$
whenever $E_r(x_0)\Subset B_1$, $x_0\in B'_1$ for some $0<\g<\al$, then it is easy to see from the proof of Lemma~\ref{lem:quasisym-alm-min} that $U_{x_0}^*$ satisfies \eqref{eq:e_u-alm-min-property}, but with $\al-\g>0$ instead of $\al$.
\end{remark}

\begin{theorem}\label{thm:u_e-grad-holder}
Let $U$ be an $A$-quasisymmetric almost minimizer for the $A$-Signorini problem  in $B_{1}$. Then
for $x_0\in B_{1/2}'$ and $0<r\leq (1/2)\Lambda^{-1/2}$, we have $U_{x_0}^*\in C^{1,\be}(E^\pm_r(x_0)\cup
E'_{r}(x_0))$ with $\be=\frac\al{4(2n+\al)}$. Moreover,
$$
\| U_{x_0}^*\|_{C^{1,\be}(K)}\le C(n,\al,M,K,r)\|U_{x_0}^*\|_{W^{1,2}(E_r(x_0))},
$$
for any $K\Subset E_r^\pm(x_0)\cup E'_r(x_0)$.
Similarly,
$u_{x_0}^*\in
C^{1,\be}(B_r^\pm\cup B_r')$ with 
\[
\|u_{x_0}^*\|_{C^{1,\be}(K)}\le C(n,\al,M,K,r)\|u_{x_0}^*\|_{W^{1,2}(B_r)},
\]
for any $K\Subset B_r^\pm\cup B_r'$.
\end{theorem}

\begin{proof}
From Theorem~\ref{thm:var-grad-holder}, we have $ U\in
C^{1,\be}(B_{1}^\pm\cup B_{1}')$, which immediately gives $U_{x_0}^*\in C^{1,\be}(E^\pm_{r}(x_0)\cup
E'_{r}(x_0))$, by using the inclusion $E_{r}(x_0)\subset
B_{\Lambda^{1/2}r}(x_0)\subset B_{1}$. Thus, for $$
\widehat{\D U_{x_0}^*}(x',x_n):=\begin{cases} \D U_{x_0}^*(x',x_n),&x_n\ge 0\\ \D U_{x_0}^*(x',-x_n),& x_n<0,
\end{cases} 
$$
we have $\widehat{\D U_{x_0}^*}\in C^{0,\be}(E_r(x_0))$ with
$$
\|\widehat{\D U_{x_0}^*}\|_{C^{0,\be}(K)}\le C(n,\al,M,K,r)\|U\|_{W^{1,2}(E_r(x_0))},
$$
for any $K\Subset E_r(x_0)$. Hence, it is enough to show that
$$
\|U\|_{W^{1,2}(E_r(x_0))}\le C\|U_{x_0}^*\|_{W^{1,2}(E_r(x_0))}.
$$
Now, note that by \eqref{eq:U2-sym-dec}--\eqref{eq:DU2-sym-dec}, we
readily have
$$
\|U\|_{W^{1,2}(E_r(x_0))}\le C\left(\|U_{x_0}^*\|_{W^{1,2}(E_r(x_0))}+\|U_{x_0}^\sharp\|_{W^{1,2}(E_r(x_0))}\right),
$$
and thus, it will suffice to show that
$$
\|U_{x_0}^\sharp\|_{W^{1,2}(E_r(x_0))}\le C\|U_{x_0}^*\|_{W^{1,2}(E_r(x_0))}.
$$
By the symmetry again, $$
\mean{U_{x_0}^\sharp}_{E_r(x_0)}=\mean{u_{x_0}^\sharp}_{B_r}=0,
$$
thus by Poincare's inequality, \begin{align}\label{eq:poin-ineq}
\|U_{x_0}^\sharp\|_{L^2(E_r(x_0))}\le C(n,M)r\|\D U_{x_0}^\sharp\|_{L^2(E_r(x_0))}.
\end{align}
Finally, by the quasisymmetry of $U$, we have
\begin{align*}
    \|\D U_{x_0}^\sharp\|_{L^2(E_r(x_0))}\le C\|\D U_{x_0}^*\|_{L^2(E_r(x_0))},
\end{align*}
see \eqref{eq:Q-sym-alt}.
This completes the proof of the theorem for $U_{x_0}^*$.

Applying now
the affine transformation $\bar{T}_{x_0}$, we obtain the part of the
theorem for $u_{x_0}^*$.
\end{proof}

We complete this section with a version of Signorini's complementarity condition that will play an important role in the analysis of the free boundary.

\begin{lemma}[Complementarity condition]\label{u_ex0-comp-cond}
Let $U$ be an $A$-quasisymmetric almost minimizer for the $A$-Signorini problem  in $B_{1}$, and $x_0\in B'_{1/2}$. Then $u_{x_0}^*$ satisfies the following \emph{complementarity condition}
$$
u_{x_0}^*(\partial_{y_n}^+u_{x_0}^*)
=0 \quad\text{on } B_{R_0}',\quad R_0=(1/2)\Ld^{-1/2},
$$
where $\partial_{y_n}^+u_{x_0}^*$ on $B_{R_0}'$ is computed as the limit from inside $B_{R_0}^+$.
Moreover, if $x_0\in \Gamma(U)$, then
$$
u_{x_0}^*(0)=0\quad\text{and}\quad |\widehat{\D u_{x_0}^*}(0) |=0.
$$
\end{lemma}

\begin{proof} 
  Let $y_0\in B_{R_0}'$ be such
  that $u_{x_0}^*(y_0)>0$. Then we need to show that $\partial_{y_n}^+u_{x_0}^*(y_0)=0$. Since $u_{x_0}=u_{x_0}^*$ on $\Pi$,
  we have $u_{x_0}(y_0)>0$ and by continuity $u_{x_0}>0$ in a small
  ball $B_{\delta}(y_0)$. Then $U>0$ in
  $\Omega=\bar{T}_{x_0}^{-1}(B_{\delta}(y_0))$. We claim now that $U$ is
  almost $A$-harmonic  in $\Omega$. Indeed,  if $E_r(y)\Subset \Om$ (not necessarily with $y\in B_1'$) and $V$ is $A(y)$-harmonic replacement of $U$ on $E_r(y)$ (i.e.\ $\dv(A(y)\D V) =0$ in $E_r(y)$ with $V=U$ on $\pa E_r(y)$), then since $V=U>0$ on $\pa E_r(y)$, by the minimum principle $V>0$ on $\overline{E_r(y)}$. This means that $V\in \mathfrak{K}_{0,U}(E_r(y),\Pi)$ and therefore we must have $$
  \int_{E_r(y)}\langle A(y)\D U, \D U\rangle \le (1+\om({r}))\int_{E_r(y)}\langle A(y)\D V,\D V\rangle,
  $$
which also implies that $U$ is an almost $A$-harmonic function in $\Om$. Hence, $U\in C^{1, \alpha/2}(\Om)$ by
Theorem~\ref{thm:var-alm-har-reg}, implying also that $u_{x_0}\in
C^{1,\alpha/2}(B_\delta(y_0))$. Consequently, also $u_{x_0}^*\in
C^{1,\alpha/2}(B_\delta(y_0))$ and by even symmetry in the $y_n$-variable,
we therefore conclude that $\partial_{y_n}^+u_{x_0}^*(y_0)=0$.

The second part of the lemma now follows by the $C^{1,\be}$ regularity and the complementarity condition.
\end{proof}


\section{Weiss- and Almgren-type monotonicity formulas}
\label{sec:weiss-almgren-type}

In this section we introduce two technical tools: Weiss- and
Almgren-type monotonicity formulas, that will play a fundamental role
in the analysis of the free boundary. In fact, the proofs of these
formulas follow immediately from the case $A\equiv I$,
following the deskewing procedure. 

To proceed, we fix a constant $\ka_0>0$. We can take it as large as we
want, however, some constants in what follows, will depend on
$\ka_0$. Then for $0<\ka<\ka_0$, we consider the Weiss-type energy functional introduced in \cite{JeoPet19a}: \begin{align*}
    W_{\ka}(t, v, x_0):=
    &\frac{e^{a t^{\al}}}{t^{n+2\ka-2}}\left[\int_{B_{t}(x_0)}|\D v|^2-\kappa\frac{1-b t^\alpha}{t}\int_{\pa B_{t}(x_0)}v^2\right],
  \end{align*}
  with
$$
a=a_\kappa=\frac{M(n+2\ka-2)}{\al},\quad b=\frac{M(n+2\ka_0)}\alpha.
$$
(The formula in \cite{JeoPet19a} corresponds to the case $M=1$.)
Based on that, we define an appropriate version of
Weiss's functional for our problem.  For a function $V$ in $E_r(x_0)$, let
\begin{multline}\label{eq:Weiss-def2}
    W_{\ka}^A(t, V, x_0):=
    \frac{e^{a t^{\al}}}{t^{n+2\ka-2}}\left[\int_{E_{t}(x_0)}\langle
      A(x_0)\D V, \D V\rangle\right.\\ -\kappa\left.\frac{1-b t^\alpha}{t}\int_{\pa
        E_{t}(x_0)}V^2\mu_{x_0}(x-x_0)\right],
  \end{multline}
for $0<t<r$, with $a$, $b$ same as above, where the weight $\mu_{x_0}$ is as in
\eqref{eq:mu}. Note that by the change of variables formulas
\eqref{eq:u-u_x-0}--\eqref{eq:u-u_x-bdry}, we have
\begin{equation}\label{eq:Weiss-def1}
W_\kappa^A(t,V,x_0):=\det \mfa_{x_0} W_\kappa(t,v_{x_0},0),\quad v_{x_0}=V\circ
\bar{T}_{x_0}^{-1}.
\end{equation}

Let now $U$ be an $A$-quasisymmetric almost minimizer for the
$A$-Signorini problem  in $B_{1}$ and $x_0\in B'_{1/2}$. By
Lemma~\ref{lem:quasisym-alm-min}, $U_{x_0}^*$ satisfies the almost
$A$-Signorini property at $x_0$ in $E_{(1/2)\Ld^{-1/2}}(x_0)$. Thus
$u_{x_0}^*$ also satisfies the almost Signorini property at $0$ in
$B_{(1/2)\Ld^{-1/2}}$. By using this observation, we then have the following Weiss-type monotonicity formulas for $U_{x_0}^*$ and $u_{x_0}^*$.

\begin{theorem}[Weiss-type monotonicity formula]\label{thm:var-weiss}
Let $U$ be an $A$-quasisymmetric almost minimizer for the $A$-Signorini
problem  in $B_{1}$. Suppose $x_0\in
B'_{1/2}$ and $U(x_0)=0$. Let $0<\ka<\ka_0$ with a fixed
$\ka_0>0$. Then, for $0<t<t_0=t_0(n,\al,\ka_0,M)$,
\begin{align*}
  &\frac d{dt}W_{\ka}(t, u_{x_0}^*, 0)
  \ge \frac{e^{a t^{\al}}}{t^{n+2\ka-2}}\int_{\partial
  B_t} \left(\partial_\nu u_{x_0}^* -\frac{\kappa(1-b
  t^\alpha)}{t}u_{x_0}^*\right)^2,\\
  &\begin{multlined}
    \frac{d}{dt}W_\kappa^A(t, U_{x_0}^*, x_0)\\
\geq \frac{e^{a t^{\al}}}{t^{n+2\ka-2}}\int_{\partial
  E_t(x_0)} \left(\langle \mfa_{x_0}\D U_{x_0}^*, \nu\rangle -\frac{\kappa(1-b
    t^\alpha)}{t}U_{x_0}^*\right)^2\mu_{x_0}(x-x_0).
\end{multlined}
\end{align*}
In particular, $W_\ka(t,u_{x_0}^*, 0)$ and $W_\ka^A(t,U_{x_0}^*,x_0)$ are nondecreasing in $t$ for $0<t<t_0$.
\end{theorem}
\begin{proof} We note that the proof of
  \cite{JeoPet19a}*{Theorem~5.1}
for the monotonicity of $W_\kappa(t,v,x_0)$  
requires the function $v$ to be an almost minimizer for the Signorini
problem with $v(x_0)=0$ for the monotonicity of its energy. However,
it is not hard to see that the almost minimizing property of $v$ is
used only when it is compared with the $\ka$-homogeneous replacement
$w$ of $v$ on balls centered at the given point $x_0$ to obtain
$$ 
\int_{B_t(x_0)}|\D w|^2\ge \frac 1{1+t^\al}\int_{B_t(x_0)}|\D v|^2,
$$
see (5.2) in \cite{JeoPet19a}.
This means that the argument in the proof of \cite{JeoPet19a}*{Theorem~5.1} also
works in our case as long as $u_{x_0}^*(0)=U(x_0)=0$ and implies the
part of the theorem for $u_{x_0}^*$. We note that
the constants $a_\kappa$ and $b$ in our case will have an additional
factor of $M$, as we 
work with $\omega(r)=Mr^\alpha$ rather than $\omega(r)=r^\alpha$ in
our case, but this change of the constants can be easily traced.

The part of the theorem for $U_{x_0}^*$ follows by a change of variables.
\end{proof}

The families of monotonicity formulas
$\{W_{\kappa}\}_{0<\kappa<\kappa_0}$ and
$\{W_{\kappa}^A\}_{0<\kappa<\kappa_0}$ have an important feature that
their intervals of monotonicity and the constant $b$ can be taken the
same for all $0<\kappa<\kappa_0$. Because of that, their monotonicity
indirectly implies that of another important quantity that we describe
below. Namely,
recall that for a function $v$ in $B_r(x_0)$, \emph{Almgren's frequency} of $v$ at
$x_0$ is defined as
$$
N(t,v,x_0):=\frac{t\int_{B_t(x_0)}|\D v|^2}{\int_{\pa
    B_t(x_0)}v^2},\quad 0<t<r.
$$
Note that this quantity is well-defined when $v$ has an
almost Signorini property at $x_0$ and $x_0\in\Gamma(v)$, since
vanishing of $\int_{\pa B_t(x_0)}v^2$ for any $t>0$, would imply
vanishing of $v$ in $B_t(x_0)$ by taking $0$ as a competitor and
consequently that $x_0\notin\Gamma(v)$.

Next consider a modification of $N$, which we call the \emph{truncated frequency}:
$$
\widehat N_{\ka_0}(t,v,x_0):=\min\left\{\frac 1{1-bt^\al}N(t,v,x_0),\ka_0\right\},
$$
where $b$ is as in Weiss-type monotonicity formulas for $\kappa<\kappa_0$. We next define the
appropriate version of $N$, $\widehat{N}_{\ka_0}$ in our
setting.
For a function $V$ in $E_r(x_0)$, we define
\begin{align*}
N^A(t,V,x_0)&:=N(t,v_{x_0},0),\\
\widehat{N}^A_{\ka_0}(t,V,x_0)&:=\widehat{N}_{\ka_0}(t,v_{x_0},0),
\end{align*}
for $0<t<r$, where $v_{x_0}=V\circ \bar{T}_{x_0}^{-1}$.
More explicitly, we have
\begin{align*}
N^A(t,V,x_0)&:=\frac{t\int_{E_t(x_0)}\langle A(x_0)\D V,\D V\rangle}{\int_{\pa
    E_t(x_0)}V^2\mu_{x_0}(x-x_0)},\\
\widehat N_{\ka_0}^A(t,V,x_0)&:=\min\left\{\frac 1{1-bt^\al}N^A(t,V,x_0),\ka_0\right\}.
\end{align*}
As observed in \cite{JeoPet19a}*{Theorem~5.4}, the  Weiss-type
monotonicity formula implies the following monotonicity of $\widehat
N_{\ka_0}^A$. 

\begin{theorem}[Almgren-type monotonicity formula]\label{thm:var-Almgren}
  Let $U$, $\kappa_0$, and $t_0$ be as in Theorem~\ref{thm:var-weiss}, and
  $x_0\in B'_{1/2}$ a free boundary point. Then
  $$t\mapsto\widehat
  N_{\ka_0}^A(t,U_{x_0}^*,x_0)=\widehat
  N_{\ka_0}(t,u_{x_0}^*,0)
  $$ is nondecreasing for $0<t<t_0$.
\end{theorem}

\begin{definition}[Almgren's frequency at free boundary point] For an  $A$-quasi\-sym\-met\-ric almost
  minimizer $U$ of the $A$-Signorini problem in $B_1$ and $x_0\in\Gamma(U)$ let
  $$
\kappa(x_0):=\widehat{N}^A_{\ka_0}(0+,U_{x_0}^*,x_0)=\widehat{N}_{\ka_0}(0+,u_{x_0}^*,0).
  $$
 We call $\kappa(x_0)$ \emph{Almgren's frequency} at $x_0$. 
\end{definition}
\begin{remark}\label{rem:kax_0}
Note that even though the monotonicity of the truncated frequency is
stated in Theorem~\ref{thm:var-Almgren} only for $x_0\in B_{1/2}'\cap
\Gamma(U)$, by a simple recentering and a scaling argument, it will be
monotone also at all $x_0\in\Gamma(U)$, but for a possibly shorter
interval of values $0<t<t_0(x_0)$ depending on $x_0$. Thus,
$\kappa(x_0)$ exists at all $x_0\in\Gamma(U)$.

Further note that when $\kappa(x_0)<\kappa_0$, then
$\widehat{N}_{\ka_0}^A(t,U^*_{x_0},x_0)=\frac{1}{1-bt^\alpha}N^A(t,U^*_{x_0},x_0)$
for small $t$ and therefore
$$
\kappa(x_0)=N^A(0+,U_{x_0}^*,x_0),
$$
which means that it will not change if we replace $\kappa_0$ with a
larger value.
\end{remark}


\section{Almgren rescalings and blowups}
\label{sec:var-almgr-resc-blow}

Our analysis of the free boundary is based on the analysis of blowups,
which are the limits of rescalings of the solutions at free boundary
points. In Signorini problem, there are a few types of rescalings that
use different normalizations. In this section, we look at so-called
Almgren rescalings and blowups that play well with the Almgren
frequency formula.

Let $V\in W^{1,2}(B_{1})$ and $x_0\in B'_{1/2}$ be a free
boundary point. For small $r>0$ define the \emph{Almgren rescaling} of $V$
at $x_0$ by
$$
V_{x_0,r}^A(x):=\frac{V(rx+x_0)}{\left(\frac 1{r^{n-1}}\int_{\pa E_r(x_0)}V^2\mu_{x_0}(x-x_0)\right)^{1/2}}.
$$
The Almgren rescalings have the following normalization and
scaling properties
\begin{align*} 
  &\|V^A_{x_0,r}\|_{L^2(\mfa_{x_0}\pa B_1)}=1\\
  & N^{A(x_0)}(\rho,V_{x_0,r}^A,0)=N^A(\rho r,V,x_0).
\end{align*}
Here $N^{A(x_0)}$ denotes Almgren's frequency for a constant matrix
$A(x_0)$. Thus, we also have $N^{A}(r,V,x_0)=N^{A(x_0)}(r,V,x_0)$.
                                                                                Note that when $A=I$, then
$$V_{x_0,r}^I=\frac{V(rx+x_0)}{\left(\frac1{r^{n-1}}\int_{\pa
      B_r(x_0)}V^2\right)^{1/2}}$$ is same as the Almgren rescaling in
\cite{JeoPet19a}, and satisfies
\begin{align*}
&\|V^I_{x_0,r}\|_{L^2(\pa B_1)}=1\\
& N(\rho,V_{x_0,r}^I,0)=N(\rho r,V,x_0).
\end{align*}
We will call the limits of $V_{x_0,r}^A$ over any subsequence
$r=r_j\to 0+$ \emph{Almgren blowups} of $V$ at $x_0$ and denote them
by $V_{x_0,0}^A$.

By using a change of variables, we can express Almgren
rescalings of $V$ in terms of those of $v_{x_0}=V\circ
\bar{T}^{-1}_{x_0}$ and vice versa. Namely, we have
$$
(v_{x_0})_r^I(y)=(\det \mfa_{x_0})^{1/2}V_{x_0,r}^A(\bar{\mfa}_{x_0}y),
$$
wherever they are defined.
Applied to the particular case $V=U_{x_0}^*$, we have
$$
(u_{x_0}^*)_{r}^I(y)=(\det \mfa_{x_0})^{1/2}(U_{x_0}^*)_{x_0,r}^A(\bar{\mfa}_{x_0}y).
$$

\begin{proposition}[Existence of Almgren
  blowups]\label{prop:var-exist-Alm-blowup} Let $U$ be an
  $A$-quasisymmetric almost minimizer for the $A$-Signorini problem in $B_{1}$, and $x_0\in
  B'_{1/2}\cap \Ga(U)$ be such that $\kappa(x_0)<\ka_0$. Then, every sequence of Almgren rescalings $(U_{x_0}^*)_{x_0,t_j}^A$, with $t_j\to 0+$, contains a subsequence, sill denoted $t_j$ such that for a function $(U_{x_0}^*)_{x_0,0}^A\in C^1_{\loc}(\mfa_{x_0}(B_1^\pm\cup B_1'))$
 \begin{align*}
   (U_{x_0}^*)_{x_0,t_j}^A\ra (U_{x_0}^*)_{x_0,0}^A
   &\quad\text{in }C^1_{\loc}(\mfa_{x_0}(B_1^\pm\cup B_1')).
 \end{align*}
Moreover, $(U_{x_0}^*)_{x_0,0}^A$ extends to a nonzero solution of the
$A(x_0)$-Signorini problem in $\R^n$, $(U_{x_0}^*)^A_{x_0,0}(x)=(U_{x_0}^*)^A_{x_0,0}(P_{x_0}x)$, and it is homogeneous
of degree $\ka(x_0)$ in $\R^n$.

Similarly, every sequence of Almgren
rescalings $(u_{x_0}^*)_{t_j}^I$, with $t_j\to 0+$ contains a
subsequence, sill denoted $t_j$ such that for a function
$(u_{x_0}^*)_{0}^I\in  C^1_{\loc}(B_1^\pm\cup B_1')$
 \begin{align*}
   (u_{x_0}^*)_{t_j}^I\ra (u_{x_0}^*)_{0}^I
   &\quad\text{in }C^1_{\loc}(B_1^\pm\cup B_1').
  \end{align*}
Moreover, $(u_{x_0}^*)_{0}^I$ extends to a nonzero solution of the Signorini problem in $\R^n$, even in $y_n$, and it is homogeneous of degree $\ka(x_0)$ in $\R^n$. 
\end{proposition}

\begin{proof}
\medskip\noindent \emph{Step 1.}
Since $\ka(x_0)<\ka_0$, we must have $N(t,u_{x_0}^*,0)<\ka_0$ for
small $t>0$. Then, for such $t$
\begin{align*}
\int_{B_1}|\D(u_{x_0}^*)_{t}^I|^2=N(1,(u_{x_0}^*)_{t}^I,0)=N(t,u_{x_0}^*,0)\le \ka_0,
\end{align*}
and combined with the normalization $\int_{\pa
  B_1}\left((u_{x_0}^*)_{t}^I\right)^2=1$, we see that the family
$(u_{x_0}^*)_{t}^I$ is bounded in $W^{1,2}(B_1)$, for small
$t>0$. Hence, for any sequence $t_j\to 0+$, there is a function
$(u_{x_0}^*)_{0}^I\in W^{1,2}(B_1)$ such that, over a subsequence,
\begin{align*}
    (u_{x_0}^*)_{t_j}^I\ra (u_{x_0}^*)_{0}^I &\quad\quad\text{weakly in }W^{1, 2}(B_1),\\
    (u_{x_0}^*)_{t_j}^I\ra (u_{x_0}^*)_{0}^I &\quad\quad\text{strongly in }L^2(\pa B_1).
  \end{align*}
In particular, $\int_{\pa B_1}\left( (u_{x_0}^*)_{0}^I\right)^2=1$, implying that $(u_{x_0}^*)_{0}^I\not\equiv 0$ in $B_1$.

\medskip\noindent \emph{Step 2.} For $0<t<1$ and  $x\in
B_{1/(2t)}(x_0)$, let
$$
U_{x_0,t}(x)=U(x_0+t(x-x_0)),\quad A_{x_0,t}(x)=A(x_0+t (x-x_0)).
$$
Then by a simple scaling argument, we have that $U_{x_0,t}$ is an almost minimizer of the
$A_{x_0,t}$-Signorini problem in $B_{1/(2t)}(x_0)$ with a gauge function
  $\mu_t(r)=(tr)^\alpha\leq r^\alpha$. 
In particular, for any $R>0$, we will have that $U_{x_0,t}\in C^{1,\beta}(E_R^\pm(x_0)\cup
E_R'(x_0))$ for $0<t<t(R,M)$ with
$$
\|U_{x_0,t}\|_{C^{1,\be}(K)}\le C\|U_{x_0,t}\|_{W^{1,2}(E_{R}(x_0))},
$$
with $C=C(n,\al,M,R,K)$, for any $K\Subset E^\pm_{R}(x_0)\cup E'_{R}(x_0)$.
Then, arguing as in the proof of Theorem~\ref{thm:u_e-grad-holder}, by
using the quasisymmetry of $U$, we obtain that
$$
\|(U_{x_0,t})^*_{x_0}\|_{C^{1,\be}(K)}\le C\|(U_{x_0,t})^*_{x_0}\|_{W^{1,2}(E_{R}(x_0))},
$$
where
$$
(U_{x_0,t})^*_{x_0}(x)=\frac{U_{x_0,t}(x)+U_{x_0,t}(P_{x_0}x)}2.
$$
Next, observing that $(u_{x_0}^*)_t^I$ is a positive constant multiple
of $(U_{x_0,t})^*_{x_0}\circ{\bar{T}_{x_0}^{-1}}$, we obtain that
$$
\|(u_{x_0}^*)_t^I\|_{C^{1,\be}(K)}\le C\|(u_{x_0}^*)_t^I\|_{W^{1,2}(B_R)},
$$
for any $K\Subset B_R^\pm\cup B_R'$.
Taking $R=1$, combined with the boundedness of $(u_{x_0}^*)_{t}^I$ in
$W^{1,2}(B_1)$ for small $t>0$, it follows that up to a subsequence, $$
  (u_{x_0}^*)_{t_j}^I\ra (u_{x_0}^*)_0^I \quad\text{in }C^1_{\loc}(B_1^\pm\cup B'_1).
$$

\medskip\noindent \emph{Step 3.}
Next, we claim that the blowup $(u_{x_0}^*)_0^I$ is a solution of the Signorini problem in $B_1$. Indeed, fix $0<R<1$, and for each $t_j$ let $h_{t_j}$ be the Signorini replacement of $(u_{x_0}^*)_{t_j}^I$ in $B_R$. Then a first variation argument gives (see \cite{JeoPet19a}*{(3.2)}) $$
\int_{B_R}\langle{\D h_{t_j},\D((u_{x_0}^*)_{t_j}^I-h_{t_j})}\rangle\ge 0.
$$
Since $(u_{x_0}^*)_{t_j}^I$ has an almost Signorini property at $0$
with a gauge function $r\mapsto C (t_jr)^\al$, it follows that
$$
\int_{B_R}|\D((u_{x_0}^*)_{t_j}^I-h_{t_j})|^2\le C (Rt_j)^\al\int_{B_R}|\D(u_{x_0}^*)_{t_j}^I|^2.
$$
This implies that $h_{t_j}\to (u_{x_0}^*)_{0}^I$ weakly in
$W^{1,2}(B_R)$. On the other hand, by the boundedness of the sequence
$h_{t_j}$ in $W^{1,2}(B_R)$, we have also boundedness in $C^{1,1/2}$ norm
locally in $(B_R^\pm\cup B_R')$ and hence, over a subsequence, $h_{t_j}\to (u_{x_0}^*)_{0}^I$ in $C_{\loc}^1(B_R^\pm\cup B_R')$. By this convergence, we then conclude that $(u_{x_0}^*)_{0}^I$ satisfies  
\begin{align*}
\La (u_{x_0}^*)_{0}^I =0&\quad\text{in }B_R\setminus B_R'\\
  (u_{x_0}^*)_{0}^I\geq 0,\quad -\partial_{y_n}^+(u_{x_0}^*)_{0}^I\geq 0,\quad
  (u_{x_0}^*)_{0}^I\partial_{y_n}^+(u_{x_0}^*)_{0}^I=0&\quad\text{on } B_R',
\end{align*}
and hence, by letting $R\to 1$, $(u_{x_0}^*)_{0}^I$ itself solves the Signorini problem in $B_1$. 

\medskip\noindent \emph{Step 4.} Recall now that the blowup
$(u_{x_0}^*)_{0}^I$ is nonzero in $B_1$. In particular, $\int_{\partial
B_r} ((u_{x_0}^*)_{0}^I)^2>0$ for any $0<r<1$, otherwise we would have
that $(u_{x_0}^*)_{0}^I$ is identically zero on $\partial B_r$ and
consequently also on $B_r$. Using this fact, combined with $C^1_{\loc}$ convergence in
$B_1^\pm\cup B_1'$, we have that for any $0<r<1$
\begin{align*}
N(r,(u_{x_0}^*)_{0}^I,0) &= \lim_{t_j\to 0}N(r,(u_{x_0}^*)_{t_j}^I,0)\\
&=\lim_{t_j\ra 0}N(r t_j,u_{x_0}^*,0)\\
&=N(0+,u_{x_0}^*,0)\\
&=\ka(x_0).
\end{align*}
Thus, Almgren's frequency of $(u_{x_0}^*)_{0}^I$ is constant
$\kappa(x_0)$ on $0<r<1$ which is possible only if $(u_{x_0}^*)_{0}^I$ is a
$\ka(x_0)$-homogeneous solution of the Signorini problem in $B_1$, see
\cite{PetShaUra12}*{Theorem~9.4}. Finally, by using the homogeneity, we
readily extend $(u_{x_0}^*)_{0}^I$ to a solution of the Signorini
problem in all of $\R^n$.
This completes the proof for
$(u_{x_0}^*)_{0}^I$.

The corresponding result for $(U_{x_0}^*)_{x_0,t_j}^A$
follows now by a change of variables.
\end{proof}

With Proposition~\ref{prop:var-exist-Alm-blowup} at hand, we can repeat the argument in the proof of Lemma~6.2 in \cite{JeoPet19a} with $u_{x_0}^*$ to obtain the following, which is possible since $u_{x_0}^*$ satisfies the complementarity condition and an Almgren-type monotonicity formula with a blowup as a nonzero solution of the Signorini problem.

\begin{lemma}[Minimal frequency]
  \label{lem:var-min-freq} Let $U$ be an $A$-quasisymmetric almost minimizer for the $A$-Signorini problem  in $B_{1}$. If $x_0\in B_{1/2}'\cap \Gamma(U)$, then
$$
\kappa(x_0)\geq\frac{3}{2}.
$$
Consequently, we also have
$$
\widehat{N}^A_{\ka_0}(t,U_{x_0}^*,x_0)=\widehat{N}_{\ka_0}(t,u_{x_0}^*,0)\geq 3/2\quad\text{for }0<t<t_0.
$$
\end{lemma}

Lemma~\ref{lem:var-min-freq} readily gives the following. (see \cite{JeoPet19a}*{Corollary~6.3})

\begin{corollary}\label{cor:W32-nonneg} Let $U$ be an $A$-quasisymmetric almost minimizer for the $A$-Signorini problem  in $B_{1}$ and $x_0$ a free boundary
  point. Then
  $$
  W_{3/2}^A(t,U_{x_0}^*, x_0)=\det \mfa_{x_0}W_{3/2}(t,u_{x_0}^*,0)\geq 0,\quad\text{for }0<t<t_0.
  $$
\end{corollary}


\section{Growth estimates}

The first result in this section (Lemma~\ref{lem:var-almost-opt-growth}) provides growth estimates for the quasi\-sym\-met\-ric almost minimizers
near free boundary points $x_0$ with $\kappa(x_0)\geq\kappa$. Such estimates were obtained in
\cite{JeoPet19a}*{Lemmas~7.1} in the case $A\equiv I$ as
a consequence of Weiss-type monotonicity formulas. However, they contain an unwanted logarithmic term that creates difficulties in the blowup analysis of the problem.

The next two results (Lemmas~\ref{lem:var-W-gr-est} and \ref{lem:var-opt-growth}) remove the logarithmic term from these estimates for $\kappa=3/2$, by establishing first a growth rate for $W_{3/2}$. (Recall that $\kappa(x_0)\geq 3/2$ at every free boundary point $x_0$, by Lemma~\ref{lem:var-min-freq}.) These are analogous to \cite{JeoPet19a}*{Lemmas~7.3, 7.4} in the case $A\equiv I$ and follow from the so-called epiperimetric inequality for $\kappa=3/2$ (see e.g.\ \cite{JeoPet19a}*{Theorem~7.2}). Later, in Section~\ref{sec:var-singular-points}, we remove the logarithmic term also in the case $\kappa=2m<\kappa_0$, $m\in\mathbb{N}$, see Lemma~\ref{lem:var-opt-est}.

The results in this section are stated in terms of both $u_{x_0}^*$ and $U_{x_0}^*$, as we need both forms in the subsequent arguments. We note that the estimates for $u_{x_0}^*$ follow directly from
\cite{JeoPet19a}*{Lemmas~7.1, 7.3, 7.4} and the ones for $U_{x_0}^*$ are obtained by using the deskewing procedure and therefore we skip all proofs in this section.

In the estimates below, as well in the rest of the paper, we use the notation
$$
R_0:=(1/2)\Lambda^{-1/2},
$$
which is the radius of the largest ball $B_{R_0}$, where $u_{x_0}^*$ is guaranteed to exists for any $x_0\in B_{1/2}'$ for an almost minimizer $U$ in $B_1$.

\begin{lemma}[Weak growth estimate]
  \label{lem:var-almost-opt-growth}
  Let $U$ be an $A$-quasisymmetric almost minimizer for the
  $A$-Signorini problem in $B_1$ and $x_0\in B_{1/2}'\cap \Gamma(U)$. If 
  \[
  \kappa(x_0)\geq \ka\] for some $\kappa\leq\kappa_0$, then
  \begin{align*}
    \int_{\pa B_t}(u_{x_0}^*)^2&\le C\|u_{x_0}^*\|^2_{W^{1, 2}(B_{R_0})}\left(\log\frac{1}{t}\right)t^{n+2\ka-1},\\
    \int_{B_t}|\D u_{x_0}^*|^2&\le C\|u_{x_0}^*\|^2_{W^{1, 2}(B_{R_0})}\left(\log\frac{1}{t}\right)t^{n+2\ka-2},\\
      \int_{\pa E_t(x_0)}(U_{x_0}^*)^2&\le C\|U\|^2_{W^{1, 2}(B_1)}\left(\log\frac{1}{t}\right)t^{n+2\ka-1},\\
    \int_{E_t(x_0)}|\D U_{x_0}^*|^2&\le C\|U\|^2_{W^{1, 2}(B_1)}\left(\log\frac{1}{t}\right)t^{n+2\ka-2},
  \end{align*}
  for $0<t<t_0=t_0(n, \al, M,\kappa_0)$ and $C=C(n, \al,M,\kappa_0)$.
\end{lemma}

\begin{lemma}\label{lem:var-W-gr-est}  Let $U$ and $x_0$ be as above. Then, there exists $\de=\de(n,\alpha)>0$ such that
\begin{align*}
0\leq W_{3/2}(t,u_{x_0}^*,0)\leq C\|u_{x_0}^*\|^2_{W^{1, 2}(B_{R_0})} t^\de,\\
0\leq W_{3/2}^A(t,U_{x_0}^*,x_0)\leq C\|U\|^2_{W^{1, 2}(B_1)} t^\de,
\end{align*}
for $0<t<t_0=t_0(n,\alpha,M)$ and $C=C(n, \al,M)$.
\end{lemma}

\begin{lemma}[Optimal growth estimate]
  \label{lem:var-opt-growth} Let $U$ and $x_0$ be as above. Then, 
  \begin{align*}
    \int_{\partial B_t} (u_{x_0}^*)^2
    &\leq C\|u_{x_0}^*\|^2_{W^{1, 2}(B_{R_0})}t^{n+2},\\
    \int_{B_t} |\nabla u_{x_0}^*|^2
    &\leq  C\|u_{x_0}^*\|^2_{W^{1, 2}(B_{R_0})}t^{n+1},\\
    \int_{\partial E_t(x_0)} (U_{x_0}^*)^2
    &\leq C\|U\|^2_{W^{1, 2}(B_1)}t^{n+2},\\
    \int_{E_t(x_0)} |\nabla U_{x_0}^*|^2
    &\leq  C\|U\|^2_{W^{1, 2}(B_1)}t^{n+1},
  \end{align*}
  for $0<t<t_0=t_0(n,\alpha,M)$ and $C=C(n, \al,M)$.
\end{lemma}


\section{$3/2$-almost homogeneous rescalings and blowups}
\label{sec:var-32-homog-blow}

In this section we study another kind of rescalings and blowups that
will play a fundamental role in the analysis of regular free
boundary points where $\kappa(x_0)=3/2$ (see the next section), namely
$3/2$-almost homogeneous blowups. The
main result that we prove in this section is the uniqueness and
H\"older continuous dependence of such blowups at a free boundary
point $x_0$ (Lemma~\ref{lem:var-blowup-est}).

For a function $v$ in $B_1$ and $x_0\in B'_{1/2}$, we
define the
$3/2$-\emph{almost homogeneous rescalings} of $v$ at $x_0$ by
$$ v^{\phi}_{x_0, t}(x)=\frac{v(tx+x_0)}{\phi(t)}, \quad \phi(t)=e^{-\left(\frac{3b}{2\al}\right)t^\al}t^{3/2},
$$
with $b$ as in the Weiss-type monotonicity formulas $W^A_{3/2}$ and $W_{3/2}$.
When $x_0=0$, we simply write $v_{0,t}^\phi=v_t^\phi$.

The name is
explained by the fact that
$$
\lim_{t\to 0}\frac{\phi(t)}{t^{3/2}}=1,
$$
and the reason to look at such rescalings instead of $3/2$-homogeneous
rescalings (that would correspond to $\phi(t)=t^{3/2}$) is how they play
well with the Weiss-type monotonicity formulas $W^A_{3/2}$ and $W_{3/2}$.

Now, if $U$ is an $A$-quasisymmetric almost minimizer and $x_0\in
B_{1/2}'\cap \Gamma(U)$, then for any fixed $R>1$, if $t=t_j>0$ is small, then by Lemma~\ref{lem:var-opt-growth}, 
\begin{align}\label{eq:hom-grad-est}
  \int_{B_R}|\D (u_{x_0}^*)^{\phi}_{t}|^2
  &=\frac{e^{{\frac{3b}{\al}t^\al}}}{t^{n+1}}\int_{B_{Rt}}|\D u_{x_0}^*|^2\le C\|u_{x_0}^*\|^2_{W^{1, 2}(B_{R_0})}R^{n+1},\\
  \int_{\pa B_R}((u_{x_0}^*)^{\phi}_{t})^2
  &= \frac{e^{{\frac{3b}{\al}t^\al}}}{t^{n+2}}\int_{\pa B_{Rt}}(u_{x_0}^*)^2\le C\|u_{x_0}^*\|^2_{W^{1, 2}(B_{R_0})}R^{n+2},\label{eq:hom-bdry-est}
\end{align}
with $C=C(n, \al,M)$, $R_0=(1/2)\Lambda^{-1/2}$.
Hence, $(u_{x_0}^*)_{t_j}^\phi$ is a bounded sequence in
$W^{1,2}(B_R)$. Next, arguing as in the proof of
Proposition~\ref{prop:var-exist-Alm-blowup}, we will have that 
\begin{align}\label{eq:hom-grad-holder-est}
\|\widehat{\D (u_{x_0}^*)_{t}^\phi}\|_{C^{0,\be}(K)}\le C\|(u_{x_0}^*)_{t}^\phi\|_{W^{1,2}(B_R)},
\end{align}
with $C=C(n,\al,M,R,K)$ for $K\Subset B_R$. Thus, by letting $R\to \infty$ and using Cantor's
diagonal argument, we can conclude that over a subsequence $t=t_j\to 0+$,
$$
(u_{x_0}^*)^{\phi}_{t_j}\ra (u_{x_0}^*)^{\phi}_{0}\quad\text{in}\quad
C^1_{\loc}(\R^{n}_\pm\cup\R^{n-1}).
$$
We call such $(u_{x_0}^*)^{\phi}_{0}$ a \emph{$3/2$-homogeneous blowup} of
$u_{x_0}^*$ at $0$. (We may skip the ``almost'' modifier here as the limit is the same
as for $3/2$-homogeneous rescalings.)
Furthermore, from the relation
$$
(u_{x_0}^*)^\phi_{t}(y)
=(U_{x_0}^*)_{x_0,t}^\phi(\bar{\mfa}_{x_0}y),
$$
we also conclude that for any sequence $t_j\to 0+$, there is a subsequence, still denoted by $t_j$, such that $$
(U_{x_0}^*)^{\phi}_{x_0, t_j}\ra (U_{x_0}^*)^{\phi}_{x_0, 0}\quad\text{in}\quad
C^1_{\loc}(\R^{n}_\pm\cup\R^{n-1}).
$$

Apriori, the blowups $(u_{x_0}^*)^\phi_0$ and
$(U_{x_0}^*)^{\phi}_{x_0,0}$ may depend on the sequence $t_j\to
0+$. However, this does not happen in the case of $3/2$-homogeneous
blowups. We start with what we call a rotation estimate for rescalings.

\begin{lemma}[Rotation estimate]
  \label{lem:var-rotation-est} Let $U$ be an $A$-quasisymmetric almost minimizer for the $A$-Signorini problem  in $B_{1}$, $x_0\in B_{1/2}'$ a free boundary point,
  and $\de$ as in Lemma~\ref{lem:var-W-gr-est}. Then, 
\begin{align*}
  \int_{\partial B_1} |(u_{x_0}^*)^\phi_{t}-(u_{x_0}^*)^\phi_{s}|&\leq C\|u_{x_0}^*\|_{W^{1,2}(B_{R_0})}
  t^{\de/2},\\
  \int_{\mfa_{x_0}\partial B_1} |(U_{x_0}^*)^\phi_{x_0, t}-(U_{x_0}^*)^\phi_{x_0, s}|&\leq C\|U\|_{W^{1,2}(B_1)}
  t^{\de/2}, 
  \end{align*}
for $s<t<t_0=t_0(n,\alpha,M)$ and $C=C(n,\alpha,M)$.
\end{lemma}
\begin{proof} 
This is an analogue of Lemma~8.2 in \cite{JeoPet19a}, which follows
from the computation done in the proof of \cite{JeoPet19a}*{Lemma
  7.1}, the growth estimate for $W_{3/2}$ in \cite{JeoPet19a}*{Lemma
  7.3} and a dyadic argument. The analogues of those results in our
case are stated in Lemma~\ref{lem:var-almost-opt-growth} and
\ref{lem:var-W-gr-est}. This proves the lemma for $u_{x_0}^*$.
The estimate for $(U_{x_0}^*)_{x_0,t}^\phi$ then follows from the equality
\[
(u_{x_0}^*)_{t}^\phi(y)=(U_{x_0}^*)_{x_0,t}^\phi(\bar{\mfa}_{x_0}y),\quad y\in B_{R_0/t}.\qedhere
\]
\end{proof}

The uniqueness of $3/2$-homogeneous blowup now follows.

\begin{lemma}\label{lem:var-blowup-rot-est} Let $(U_{x_0}^*)_{x_0,0}^\phi$ and $(u_{x_0}^*)_{0}^\phi$ be blowups of $(U_{x_0}^*)_{x_0,t}^\phi$ and $(u_{x_0}^*)_{t}^\phi$, respectively, at a free boundary point $x_0\in B_{1/2}'$. Then, 
\begin{align*}
  \int_{\partial B_1} |(u_{x_0}^*)_{t}^\phi - (u_{x_0}^*)_{0}^\phi|&\leq C\|u_{x_0}^*\|_{W^{1,2}(B_{R_0})}
  t^{\de/2},\\
  \int_{\mfa_{x_0}\partial B_1} |(U_{x_0}^*)_{x_0,t}^\phi - (U_{x_0}^*)_{x_0,0}^\phi|&\leq C\|U\|_{W^{1,2}(B_1)}
  t^{\de/2},  
\end{align*}
for $0<t<t_0(n,\al,M)$ and $C=C(n,\al,M)$, where $\de=\de(n,\alpha)>0$ is as in Lemma~\ref{lem:var-rotation-est}. In
particular, the blowups $(u_{x_0}^*)_{0}^\phi$ and $(U_{x_0}^*)_{x_0,0}^\phi$ are unique.
\end{lemma}
\begin{proof} If $(u_{x_0}^*)_{0}^\phi$ is the limit of $(u_{x_0}^*)^\phi_{t_j}$ for
  $t_j\to 0$, then the first part of the lemma follows immediately from
  Lemma~\ref{lem:var-rotation-est}, by taking $s=t_j\to 0$ and passing to
  the limit.

  To see the uniqueness of blowups, we observe that
  $(u_{x_0}^*)_{0}^\phi$ is
  a solution of the Signorini problem in $B_1$, by arguing as in the
  proof of Proposition~\ref{prop:var-exist-Alm-blowup} for Almgren
  blowups. Now, if
  $v_0$ is another blowup, over a possibly different sequence $t_j'\to
  0$, then passing to the limit in the first part
  of the lemma we will have
$$
\int_{\partial B_1} |v_0 - (u_{x_0}^*)_{0}^\phi|^2=0,
$$
implying that both $v_0$ and $(u_{x_0}^*)_{0}^\phi$ are
solutions of the Signorini problem in $B_1$ with the same boundary
values on $\partial B_1$. By the uniqueness of such solutions, we have
$v_0=(u_{x_0}^*)_{0}^\phi$ in $B_1$. The equality
propagates to all of $\R^n$ by the unique continuation of harmonic
functions in $\R^n_\pm$. This completes the proof for $u_{x_0}^*$. An
analogous argument holds for $U_{x_0}^*$ using the equalities
\begin{align*}
  (u_{x_0}^*)_{t}^\phi(y)=(U_{x_0}^*)_{x_0,t}^\phi(\bar{\mfa}_{x_0}y),
  &\quad y\in B_{R_0/t},\\
(u_{x_0}^*)_{0}^\phi(y)=(U_{x_0}^*)_{x_0,0}^\phi(\bar{\mfa}_{x_0}y),&\quad y\in \R^n.\qedhere
\end{align*}
\end{proof}

The rotation estimate for rescalings implies not only the uniqueness
of blowups and the convergence rate to blowups, but also the
continuous dependence of blowups on a free boundary point.  

\begin{lemma}[Continuous dependence of blowups]
  \label{lem:var-blowup-est} There exists $\rho=\rho(n,\alpha,M)>0$ such
  that if $x_0,y_0\in B_{\rho}'$ are free boundary points of $U$, then
\begin{align}\label{eq:var-blowup-est-u}
\int_{\mfa_{x_0}\partial B_1} |(U_{x_0}^*)_{x_0,0}^\phi - (U_{y_0}^*)_{y_0,0}^\phi|&\leq
C|x_0-y_0|^\g,\\
\label{eq:var-blowup-est-u_x}\int_{\pa
  B_1}|(u_{x_0}^*)_{0}^\phi-(u_{y_0}^*)_{0}^\phi|&\le C|x_0-y_0|^\g,\\
\label{eq:var-blowup-est-thin}\int_{\pa B'_1}|(u_{x_0}^*)_0^\phi-(u_{y_0}^*)_0^\phi|&\le C|x_0-y_0|^\g,
\end{align}
with $C=C(n,\al,M,\|U\|_{W^{1,2}(B_{1})})$, $\g=\g(n,\alpha,M)>0$.
\end{lemma}

\begin{proof}\emph{Step 1.}
Let $d=|x_0-y_0|$ and $d^\tau\leq r\leq 2 d^\tau$ with
$\tau=\tau(\al)\in(0,1)$ to be determined later.

Next note that we can incorporate the weight ${\mu_{x_0}}/{\det
  \mfa_{x_0}}$ with $\mu_{x_0}$ as
in \eqref{eq:mu} in the integral on the left hand side of
\eqref{eq:var-blowup-est-u} because of the bounds
$$
\left(\frac\lambda\Lambda\right)^{1/2}\leq\frac{\mu_{x_0}}{\det
  \mfa_{x_0}}\leq \left(\frac\Lambda\lambda\right)^{1/2}.
$$
Then, by using Lemma~\ref{lem:var-blowup-rot-est}, we
  have
  \begin{multline} \label{eq:var-blowup-est-1}
    \int_{\mfa_{x_0}\partial B_1} |(U_{x_0}^*)_{x_0,0}^\phi -
    (U_{y_0}^*)_{y_0,0}^\phi|\frac{\mu_{x_0}}{\det \mfa_{x_0}}\\*
 \begin{aligned}
    & \le \int_{\mfa_{x_0}\pa B_1}\bigg(|(U_{x_0}^*)_{x_0,0}^\phi-(U_{x_0}^*)_{x_0,r}^\phi|+|(U_{x_0}^*)_{x_0,r}^\phi-(U_{x_0}^*)_{y_0,r}^\phi|\\
    &\qquad+|(U_{x_0}^*)_{y_0,r}^\phi-(U_{y_0}^*)_{y_0,r}^\phi|+|(U_{y_0}^*)_{y_0,r}^\phi-(U_{y_0}^*)_{y_0,0}^\phi|       \bigg)\frac{\mu_{x_0}}{\det \mfa_{x_0}}\\
&\qquad +\int_{\mfa_{y_0}\pa
  B_1}|(U_{y_0}^*)_{y_0,r}^\phi-(U_{y_0}^*)_{y_0,0}^\phi|\frac{\mu_{y_0}}{\det \mfa_{y_0}}\\
&\qquad
-\int_{\mfa_{y_0}\pa B_1}|(U_{y_0}^*)_{y_0,r}^\phi-(U_{y_0}^*)_{y_0,0}^\phi|\frac{\mu_{y_0}}{\det \mfa_{y_0}}\\
&\le 2Cr^{\de/2}+I_r+II_r+III_r\\
&\le Cd^{\tau\de/2}+I_r+II_r+III_r,
\end{aligned}
\end{multline}
where
\begin{align*}
I_r&=\int_{\mfa_{x_0}\pa B_1}|(U_{x_0}^*)_{x_0,r}^\phi-(U_{x_0}^*)_{y_0,r}^\phi|\frac{\mu_{x_0}}{\det \mfa_{x_0}},\\
II_r&=\int_{\mfa_{x_0}\pa B_1}|(U_{x_0}^*)_{y_0,r}^\phi-(U_{y_0}^*)_{y_0,r}^\phi|\frac{\mu_{x_0}}{\det \mfa_{x_0}},\\
III_r&=\int_{\mfa_{x_0}\pa
       B_1}|(U_{y_0}^*)^\phi_{y_0,r}-(U_{y_0}^*)^\phi_{y_0,0}|\frac{\mu_{x_0}}{\det
       \mfa_{x_0}}\\
  &\qquad-\int_{\mfa_{y_0}\pa B_1}|(U_{y_0}^*)^\phi_{y_0,r}-(U_{y_0}^*)^\phi_{y_0,0}|\frac{\mu_{y_0}}{\det \mfa_{y_0}}.
\end{align*}

 \medskip\noindent \emph{Step 2.}
By the definition of the almost homogeneous rescalings, we have
 \begin{align*}
    I_r \le \frac C{d^{\tau(n+1/2)}}\int_{\mfa_{x_0}\pa B_r}|U_{x_0}^*(z+x_0)-U_{x_0}^*(z+y_0)|dS_z.
 \end{align*}
This gives 
\begin{align*}
   & \frac 1{d^\tau}\int_{d^\tau}^{2d^\tau}I_r\,dr\\
   &\qquad\le \frac C{d^{\tau(n+3/2)}}\int_{d^\tau}^{2d^\tau}\int_{\mfa_{x_0}\pa B_r}|U_{x_0}^*(z+x_0)-U_{x_0}^*(z+y_0)|dS_zdr\\
    &\qquad\le \frac C{d^{\tau(n+3/2)}}\int_{\mfa_{x_0}(B_{2d^\tau}\sm B_{d^\tau})}|U_{x_0}^*(z+x_0)-U_{x_0}^*(z+y_0)|dz\\
    &\qquad=\frac C{d^{\tau(n+3/2)}}\int_{\mfa_{x_0}(B_{2d^\tau}\sm B_{d^\tau})}\left|\int_0^1\frac d{ds}\left[U_{x_0}^*(z+x_0(1-s)+y_0s)      \right]ds \right|dz\\
    &\qquad\le \frac C{d^{\tau(n+3/2)}}|x_0-y_0|\int_0^1\int_{\mfa_{x_0}(B_{2d^\tau}\sm B_{d^\tau})}|\D U_{x_0}^*(z+x_0(1-s)+y_0s)|dzds\\
    &\qquad\le \frac C{d^{\tau(n+3/2)-1}}\int_0^1\int_{\mfa_{x_0}B_{2d^\tau}+[x_0(1-s)+y_0s]}|\D U_{x_0}^*|dzds.
\end{align*}
Notice that the last integral is taken over
\begin{align*}
    \mfa_{x_0}B_{2d^\tau}+[x_0(1-s)+y_0s]&=\mfa_{x_0}[B_{2d^\tau}+s\mfa_{x_0}^{-1}(y_0-x_0)]+x_0\\
    &\subset \mfa_{x_0}B_{2d^\tau+\la^{-1/2}d}+x_0\subset E_{3d^\tau}(x_0),
\end{align*}
if $\rho=\rho(n,\al,M)$ is small so that $(2\rho)^{1-\tau}\le \la^{1/2}$ which readily implies $d^{1-\tau}\le\la^{1/2}$.
Thus, 
\begin{align*}
     \frac 1{d^\tau}\int_{d^\tau}^{2d^\tau}I_r\,dr
   &\le\frac C{d^{\tau(n+3/2)-1}}\int_0^1\int_{E_{3d^\tau}(x_0)}|\D U_{x_0}^*|dzds\\
   &\le\frac C{d^{\tau(n/2+3/2)-1}}\left(\int_{E_{3d^\tau}(x_0)}|\D U_{x_0}^*|^2\right)^{1/2}\\
   &\le C\|U\|_{W^{1,2}(B_{1})}d^{1-\tau},
\end{align*}
where the third inequality follows from
Lemma~\ref{lem:var-opt-growth}.

 \medskip\noindent \emph{Step 3.}
By the definition of rescalings and symmetrizations, we have
 \begin{align*}
    II_r &\le \frac C{d^{\tau(n+1/2)}}\int_{\mfa_{x_0}\pa B_r+y_0}|U_{x_0}^*(z)-U_{y_0}^*(z)|dS_z\\
    &\le\frac C{d^{\tau(n+1/2)}}\int_{\mfa_{x_0}\pa B_r+y_0}|U(P_{x_0}z)-U(P_{y_0}z)|dS_z.
\end{align*}
This gives
\begin{align*}
     & \frac 1{d^\tau}\int_{d^\tau}^{2d^\tau}II_r\,dr\\
   &\qquad\le \frac C{d^{\tau(n+3/2)}}\int_{\mfa_{x_0}(B_{2d^\tau}\sm B_{d^\tau})+y_0}|U(P_{x_0}z)-U(P_{y_0}z)|dz\\
   &\qquad\le \frac C{d^{\tau(n+3/2)}}\int_{\mfa_{x_0}(B_{2d^\tau}\sm B_{d^\tau})+y_0}\int_0^1\left|\frac{d}{ds}[U([(1-s)P_{x_0}+sP_{y_0}]z)]\right|ds dz\\
   &\qquad\le \frac {C|P_{x_0}-P_{y_0}|}{d^{\tau(n+3/2)}}\int_0^1\int_{\mfa_{x_0}(B_{2d^\tau}\sm B_{d^\tau})+y_0}|\D U([(1-s)P_{x_0}+sP_{y_0}]z)|dzds.
\end{align*}
Now we do the change of variables
$$
y=[(1-s)P_{x_0}+sP_{y_0}]z.
$$
Since $P_{x_0}$ and $P_{y_0}$ are upper-triangular matrices with diagonal entries $1,1,\dots,1,-1$, so is $(1-s)P_{x_0}+sP_{y_0}$. Thus
$$
\left|\det[(1-s)P_{x_0}+sP_{y_0}] \right|=1.
$$
Moreover, $y\in [(1-s)P_{x_0}+sP_{y_0}](\mfa_{x_0}B_{2d^\tau}+y_0)$. Since $$
\mfa_{x_0}B_{2d^\tau}+y_0\subset \mfa_{y_0}B_{2(\Ld/\la)^{1/2}d^\tau}+y_0=E_{2(\Ld/\la)^{1/2}d^\tau}(y_0),
$$
we have $$
P_{y_0}(\mfa_{x_0}B_{2d^\tau}+y_0)\subset P_{y_0}E_{2(\Ld/\la)^{1/2}d^\tau}(y_0)=E_{2(\Ld/\la)^{1/2}d^\tau}(y_0).
$$
Similarly, since \begin{align*}
    \mfa_{x_0}B_{2d^\tau}+y_0&=E_{2d^\tau}(x_0)+(y_0-x_0)\subset B_{2\Ld^{1/2}d^\tau}(x_0)+(y_0-x_0)\\
    &\subset B_{4\Ld^{1/2}d^\tau}(x_0)\subset E_{4(\Ld/\la)^{1/2}d^\tau}(x_0),
\end{align*}
we have $$
P_{x_0}(\mfa_{x_0}B_{2d^\tau}+y_0)\subset E_{4(\Ld/\la)^{1/2}d^\tau}(x_0).
$$
Thus \begin{align*}
    &y\in (1-s)P_{x_0}(\mfa_{x_0}B_{2d^\tau}+y_0)+sP_{y_0}(\mfa_{x_0}B_{2d^\tau}+y_0)\\
    &\qquad\subset (1-s)E_{4(\Ld/\la)^{1/2}d^\tau}(x_0)+sE_{2(\Ld/\la)^{1/2}d^\tau}(y_0)\\
    &\qquad \subset B_{6(\Ld/\la^{1/2})d^\tau}+x_0+s(y_0-x_0)\\
    &\qquad \subset B_{7(\Ld/\la^{1/2})d^\tau}+x_0\subset E_{7(\Ld/\la)d^\tau}(x_0).
\end{align*}
Therefore, \begin{align*}
    \frac 1{d^\tau}\int_{d^\tau}^{2d^\tau}II_r\,dr &\le \frac C{d^{\tau(n+3/2)-\al}}\int_0^1\int_{E_{7(\Ld/\la)d^\tau}(x_0)}|\D U|dzds\\
   &\le \frac C{d^{\tau(n/2+3/2)-\al}}\left(\int_{E_{7(\Ld/\la)d^\tau}(x_0)}|\D U|^2\right)^{1/2}\\
   &\le\frac C{d^{\tau(n/2+3/2)-\al}}\left(\int_{E_{7(\Ld/\la)d^\tau}(x_0)}|\D U_{x_0}^*|^2\right)^{1/2}\\
   &\le C\|U\|_{W^{1,2}(B_{1})}d^{\al-\tau},
\end{align*}
for small $\rho$, where the third inequality follows from the quasisymmetry property and the last inequality from Lemma~\ref{lem:var-opt-growth}.

\medskip\noindent \emph{Step 4.}
By the change of variables, we have
\begin{align*}
   III_r
  &=\int_{\pa B_1}|(U_{y_0}^*)_{y_0,r}^\phi(\mfa_{x_0}z)-(U_{y_0}^*)_{y_0,0}^\phi(\mfa_{x_0}z)|\\
  &\qquad-\int_{\pa
    B_1}|(U_{y_0}^*)_{y_0,r}^\phi(\mfa_{y_0}z)-(U_{y_0}^*)_{y_0,0}^\phi(\mfa_{y_0}z)|\\
  &\le \int_{\pa B_1}|(U^*_{y_0})^\phi_{y_0,r}(\mfa_{x_0}z)-(U^*_{y_0})^\phi_{y_0,r}(\mfa_{y_0}z)|\\
  &\qquad+\int_{\pa B_1}|(U^*_{y_0})^\phi_{y_0,0}(\mfa_{x_0}z)-(U^*_{y_0})^\phi_{y_0,0}(\mfa_{y_0}z)|\\
  &\le
    C\left(\|\D(U_{y_0}^*)_{y_0,r}^\phi\|_{L^\infty(B_{\Ld^{1/2}})}+\|\D(U_{y_0}^*)_{y_0,0}^\phi\|_{L^\infty(B_{\Ld^{1/2}})}
    \right)|\mfa_{x_0}-\mfa_{y_0}|,
\end{align*}
where we have used the fact that both $\mfa_{x_0}z$ and
$\mfa_{y_0}z$ are contained in $\overline {B_{\Lambda^{1/2}}}$ for $z\in\partial
B_1$. To estimate the gradients of rescalings
we first observe that by the inclusion $B_{r\Ld^{1/2}}(y_0)\subset E_{r(\Ld/\la)^{1/2}}(y_0)\subset B_{r\Ld/\la^{1/2}}(y_0)$, we have
$$
\|\D(U_{y_0}^*)_{y_0,r}^\phi\|_{L^\infty(B_{\Ld^{1/2}})}\le \frac C{r^{1/2}}\|\D U_{y_0}^*\|_{L^\infty(B_{r\Ld^{1/2}}(y_0))}\le \frac C{r^{1/2}}\|\D U\|_{L^\infty(B_{r\Ld/\la^{1/2}}(y_0))}.
$$
Let $U_{y_0,r}(x):=U(r(x-y_0)+y_0)$. Then, arguing as in the proof of Proposition~\ref{prop:var-exist-Alm-blowup}, we have $$
\|\D U_{y_0,r}\|_{L^\infty(B_{\Ld/\la^{1/2}}(y_0))}\le C(n,\al,M)\|U_{y_0,r}\|_{W^{1,2}(B_{2\Ld/\la^{1/2}}(y_0))}.
$$
Thus
\begin{align*}
  \|\D U\|_{L^\infty(B_{r\Ld/\la^{1/2}}(y_0))}
  & =\frac 1r\|\D U_{y_0,r}\|_{L^\infty(B_{\Ld/\la^{1/2}}(y_0))}\\
  &\le \frac Cr\|U_{y_0,r}\|_{W^{1,2}(B_{2\Ld/\la^{1/2}}(y_0))}\\
  &\le \frac C{r^{n/2+1}}\|U\|_{L^2(B_{2r\Ld/\la^{1/2}}(y_0))}+\frac
    C{r^{n/2}}\|\D U\|_{L^2(B_{2r\Ld/\la^{1/2}}(y_0))}\\
    &\le \frac C{r^{{n/2}+1}}\|U_{y_0}^*\|_{L^2(E_{2r\Ld/\la}(y_0))}+\frac C{r^{n/2}}\|\D U_{y_0}^*\|_{L^2(E_{2r\Ld/\la}(y_0))}\\
    &\le
    Cr^{1/2}\|U\|_{W^{1,2}(B_{1})},
\end{align*}
where we have used the inclusion
$B_{2r\Ld/\la^{1/2}}(y_0)\subset E_{2r\Ld/\la}(y_0)$ and the
quasisymmetry property in the third inequality and
Lemma~\ref{lem:var-opt-growth} in the forth. 
Therefore,
$$
\|\D(U_{y_0}^*)_{y_0,r}^\phi\|_{L^\infty(B_{\Ld^{1/2}})}\leq\frac
C{r^{1/2}}\|\D U\|_{L^\infty(B_{r\Ld/\la^{1/2}}(y_0))}\leq C\|U\|_{W^{1,2}(B_{1})}.
$$
Moreover, by $C_{\loc}^1$ convergence of $(U_{y_0}^*)_{y_0,r}^\phi$ to $(U_{y_0}^*)_{y_0,0}^\phi$, we also have \begin{align}\label{eq:hom-blowup-L-infinity-est}
\|\D(U_{y_0}^*)_{y_0,0}^\phi\|_{L^\infty(B_{\Ld^{1/2}})}=\lim_{r_j\ra 0+}\|\D(U_{y_0}^*)_{y_0,r_j}^\phi\|_{L^\infty(B_{\Ld^{1/2}})}\le C\|U\|_{W^{1,2}(B_{1})}.
\end{align}
Therefore,
\begin{align*}
    III_r &\le C|\mfa_{x_0}-\mfa_{y_0} |\|U\|_{W^{1,2}(B_{1})}\\
    &\le C\|U\|_{W^{1,2}(B_{1})} d^{\al}.
\end{align*}

\medskip\noindent \emph{Step 5.}
Now we are ready to prove \eqref{eq:var-blowup-est-u}.
Using the estimates in Steps 2--4 and taking the average over
$d^\tau\le r\le 2d^\tau$, we have
$$
\int_{\mfa_{x_0}\pa B_1}|(U_{x_0}^*)_{x_0,0}^\phi-(U_{y_0}^*)_{y_0,0}^\phi|\le C\|U\|_{W^{1,2}(B_{1})}(d^{\tau\de/2}+d^{1-\tau}+d^{\al-\tau}+d^{\al}).
$$
If we simply take $\tau=\al/2$, then we conclude $$
\int_{\mfa_{x_0}\pa B_1}|(U_{x_0}^*)_{x_0,0}^\phi-(U_{y_0}^*)_{y_0,0}^\phi|\le C|x_0-y_0|^\g,
$$
with $\g=\al\de/4$ and $C=C(n,\alpha,M,\|U\|_{W^{1,2}(B_{1})})$.

\medskip\noindent \emph{Step 6.}
To prove \eqref{eq:var-blowup-est-u_x}, we first observe that from \eqref{eq:var-blowup-est-u},
\begin{align*}
&\int_{\pa B_1}|(u_{x_0}^*)_{0}^\phi(z)-(u_{y_0}^*)_{0}^\phi(\bar{\mfa}_{y_0}^{-1}\bar{\mfa}_{x_0}z)|\\
&\qquad =\int_{\pa B_1}|(U_{x_0}^*)_{x_0,0}^\phi(\bar{\mfa}_{x_0}z)-(U_{y_0}^*)_{y_0,0}^\phi(\bar{\mfa}_{x_0}z)|\\
  &\qquad=\int_{\mfa_{x_0}
    \pa B_1}|(U_{x_0}^*)_{x_0,0}^\phi-(U_{y_0}^*)_{y_0,0}^\phi |\frac{\mu_{x_0}}{\det \mfa_{x_0}}\\
&\qquad\le C|x_0-y_0|^\g.
\end{align*}
On the other hand, 
\begin{align*}
    &\int_{\pa B_1}|(u_{y_0}^*)_{0}^\phi(z)-(u_{y_0}^*)_{0}^\phi(\bar{\mfa}_{y_0}^{-1}\bar{\mfa}_{x_0}z)|\\
    &\qquad =\int_{\mfa_{x_0}\pa B_1}|(u_{y_0}^*)_{0}^\phi(\bar{\mfa}_{x_0}^{-1}z)-(u_{y_0}^*)_{0}^\phi(\bar{\mfa}_{y_0}^{-1}z)|\frac{\mu_{x_0}}{\det \mfa_{x_0}}\\
    &\qquad\le C\|\D(u_{y_0}^*)_{0}^\phi\|_{L^\infty(B_{(\Ld/\la)^{1/2}})}|\bar{\mfa}_{x_0}^{-1}-\bar{\mfa}_{y_0}^{-1}|\\
    &\qquad\le C\|\D(U^*_{y_0})^\phi_{y_0,0}\|_{L^\infty(B_{\Ld/\la^{1/2}})}|x_0-y_0|^\al\\
    &\qquad \le C\|U\|_{W^{1,2}(B_{1})}|x_0-y_0|^\al,
\end{align*}
where the last inequality follows from \eqref{eq:hom-blowup-L-infinity-est}. (It is easy to see that we can enlarge the domain in \eqref{eq:hom-blowup-L-infinity-est}.)
Therefore, combining the preceding two estimates, we conclude
that
\begin{align*}
       \int_{\pa B_1}|(u_{x_0}^*)_{0}^\phi-(u_{y_0}^*)_{0}^\phi|
    & \le C|x_0-y_0|^\g.
\end{align*}

\medskip\noindent\emph{Step 7.} Finally, \eqref{eq:var-blowup-est-u_x}
implies \eqref{eq:var-blowup-est-thin}, by arguing precisely as in \cite{GarPetSVG16}*{Proposition~7.4}.
\end{proof}


\section{Regularity of the regular set}
\label{sec:var-regul-regul-set}

In this section we combine the uniqueness and H\"older continuous
dependence of $3/2$-homogeneous blowups of the symmetrized almost
minimizers $(U_{x_0}^*)_{x_0,0}^\phi$
(Lemma~\ref{lem:var-blowup-est}) with a classification of
such blowups 
at so-called regular points
(Proposition~\ref{prop:var-gap-32-hom-classif}) to prove one of the main results
of this paper, the $C^{1, \g}$ regularity of the regular set
(Theorem~\ref{thm:var-C1g-regset}). While some arguments follow
directly from those in the case $A\equiv I$  by a coordinate
transformation $\bar{T}_{x_0}$, the 
dependence of these transformations on $x_0$ creates an additional difficulty.

We start by defining the regular set.

\begin{definition}[Regular points]
  \label{def:var-reg-set}
  For an $A$-quasisymmetric almost minimizer $U$ for the $A$-Signorini problem  in $B_{1}$, we
  say that a free boundary point $x_0$ of $U$ is \emph{regular} if
  $$
  \kappa(x_0)=3/2.
  $$
We denote the set of all regular points of $U$ by $\mathcal{R}(U)$ and
call it the \emph{regular set}.
\end{definition}
We explicitly observe here that $3/2<2\leq\kappa_0$, so the fact $x_0\in\cR(U)$
  is independent of the choice of $\kappa_0\geq 2$, see Remark~\ref{rem:kax_0}.

\medskip
The proofs of the following two results (Lemma~\ref{lem:var-nondeg}
and Proposition~\ref{prop:var-gap-32-hom-classif}) are established
precisely as in  \cite{JeoPet19a}*{Lemma~9.2, Proposition~9.3} for the
transformed functions $u_{x_0}^*$. The equivalent statements for 
$U_{x_0}^*$ are obtained by changing back to the original variables.

\begin{lemma}[Nondegeneracy at regular points]\label{lem:var-nondeg} Let  $x_0\in B'_{1/2}\cap\cR(U)$ for an $A$-quasisymmetric almost
  minimizer $U$ for the $A$-Signorini problem  in $B_{1}$. Then, for
  $\kappa=3/2$,
  $$
  \liminf_{t\to 0} \int_{\mfa_{x_0}\partial B_1} ((U_{x_0}^*)^\phi_{x_0,
    t})^2\mu_{x_0}=\det \mfa_{x_0}\liminf_{t\to 0}\int_{\partial B_1}
((u_{x_0}^*)_{t}^\phi)^2>0.
  $$
\end{lemma}

\begin{proposition}\label{prop:var-gap-32-hom-classif}
  If $\kappa(x_0)<2$, then necessarily $\ka(x_0)=3/2$ and
\begin{align*}
  (u_{x_0}^*)_0^\phi(z)&=a_{x_0} \Re(z'\cdot \nu_{x_0}+i|z_n|)^{3/2},\\
  (U_{x_0}^*)_{x_0,0}^\phi(x)&=a_{x_0}\Re((\bar{\mfa}_{x_0}^{-1}x)'\cdot \nu_{x_0}+i|(\bar{\mfa}_{x_0}^{-1}x)_n|)^{3/2},
  \end{align*}
  for some $a_{x_0}>0$, $\nu_{x_0}\in \partial B_1'$.
\end{proposition}

The next two corollaries are obtained by repeating the same arguments
as in \cite{JeoPet19a}*{Corollaries~9.4 and 9.5}.

\begin{corollary}[Almgren's frequency gap]
  \label{cor:var-gap} Let $U$ and $x_0$ be as in Lemma~\ref{lem:var-nondeg}.  Then either
  $$
  \kappa(x_0)=3/2\quad\text{or}\quad \kappa(x_0)\geq 2.
  $$
\end{corollary}

\begin{corollary}\label{cor:reg-set-open} The regular set $\mathcal{R}(U)$ is a relatively
  open subset of the free boundary.
\end{corollary}

The combination of Proposition~\ref{prop:var-gap-32-hom-classif} and
Lemma~\ref{lem:var-blowup-est} implies the following lemma.

\begin{lemma}\label{lem:var-blowup-hol}
  Let $U$ and $x_0$ be as in Lemma~\ref{lem:var-nondeg}. Then there exists $\rho>0$, depending on
  $x_0$ such that $B_\rho'(x_0)\cap \Gamma(U)\subset \mathcal{R}(U)$
  and if
  $$
  (u_{\bx}^*)_0^\phi(z)=a_{\bx} \Re(z'\cdot\nu_{\bx}+i|z_n|)^{3/2}
  $$
  is the unique $3/2$-homogeneous blowup of $u_{\bx}^*$ at
  $\bx \in B'_{\rho}(x_0)\cap \Gamma(u)$, then
  \begin{align*}
    |a_{\bx}-a_{\by}| \le C_0|\bx-\by|^{\g},\\
    |\nu_{\bx}-\nu_{\by}| \le C_0|\bx-\by|^{\g},
  \end{align*}
  for any $\bx, \by\in B'_{\rho}(x_0) \cap \Gamma(u)$ with a constant
  $C_0$ depending on $x_0$.
\end{lemma}
\begin{proof} The proof follows by repeating the argument in Lemma~7.5
  in \cite{GarPetSVG16} with $(u_{\bx}^*)_0^\phi$, $(u_{\by}^*)_0^\phi$.
\end{proof}

Now we are ready to prove the main result on the regularity of the
regular set.

\begin{theorem}[$C^{1,\gamma}$ regularity of the regular set]
  \label{thm:var-C1g-regset}
  Let $U$ be an $A$-quasisymmetric almost minimizer for the $A$-Signorini problem  in $B_{1}$. Then, if $x_0\in B_{1/2}'\cap\mathcal{R}(U)$, there exists
  $\rho>0$, depending on $x_0$ such that, after a possible rotation of
  coordinate axes in $\R^{n-1}$, one has
  $B'_{\rho}(x_0)\cap \Gamma(U) \subset \mathcal{R}(U)$, and
  \begin{align*}
    B'_{\rho}(x_0) \cap \Gamma(U)
    &= B'_{\rho}(x_0)\cap \{x_{n-1}=g(x_1,\dots, x_{n-2})\},
  \end{align*}
  for $g\in C^{1, \g}(\R^{n-2})$ with an exponent
  $\g=\g(n, \al,M)\in (0, 1)$.

\end{theorem}

\begin{proof}
  The proof of the theorem is similar to those of in \cite{JeoPet19a}*{Theorem~9.7} and 
  \cite{GarPetSVG16}*{Theorem~1.2}. However, we provide full details since there are
  technical differences.

  \medskip\noindent \emph{Step 1.}  By relative openness of $\cR(U)$
  in $\Gamma(U)$, for small $\rho>0$ we have
  $B'_{2\rho}(x_0)\cap \Ga(U)\subset \cR(U)$.  We then claim that for
  any $\e>0$, there is $r_{\e}>0$ such that for
  $\bx\in B'_{\rho}(x_0)\cap \Ga(U)$, $r<r_{\e}$, we have that 
  $$
  \|(u_{\bx}^*)^{\phi}_{r}-(u_{\bx}^*)^{\phi}_{0}\|_{C^1(\overline{B_1^{\pm}})}<\e.
  $$
  Assuming the contrary, there is a sequence of points
  $\bx_j\in B'_{\rho}(x_0)\cap \Ga(U)$ and radii $r_j\ra 0$ such that
  $$
  \|(u_{\bx_j}^*)_{r_j}^\phi-(u_{\bx_j}^*)_{0}^\phi\|_{C^1(\overline{B_1^{\pm}})}\ge \e_0,
  $$
  for some $\e_0>0$. Taking a subsequence if necessary, we may assume
  $\bx_j\ra \bx_0\in \overline{B'_{\rho}(x_0)}\cap \Ga(U)$.  Using
  estimates \eqref{eq:hom-grad-est}--\eqref{eq:hom-grad-holder-est}, we can see that $\D(u^*_{\bx_j})_{r_j}^\phi$
  are uniformly bounded in $C^{0, \be}(B_2^{\pm}\cup B'_2)$. Since $(u^*_{\bx_j})_{r_j}^\phi(0)=0$, we also have that $(u^*_{\bx_j})_{r_j}^\phi$ is uniformly bounded in $C^{1, \be}(B_2^{\pm}\cup B'_2)$.  Thus, we
  may assume that for some $w$
  \begin{align*}
    (u^*_{\bx_j})_{r_j}^\phi\ra w \quad \text{in }C^1(\overline{B_1^{\pm}}).
  \end{align*}
  By arguing as in the proof of
  Proposition~\ref{prop:var-exist-Alm-blowup}, we see that the limit $w$
  is a solution of the Signorini problem in $B_1$.  Further, by
  Lemma~\ref{lem:var-blowup-rot-est}, we have
$$
\|(u_{\bx_j}^*)_{r_j}^\phi-(u_{\bx_j}^*)_{0}^\phi\|_{L^1(\pa B_1)}\ra 0.
$$
On the other hand, by Lemma~\ref{lem:var-blowup-hol}, we have
$$
(u_{\bx_j}^*)^{\phi}_{0}\ra (u_{\bx_0}^*)^{\phi}_{0}\quad \text{in }
C^1(\overline{B^{\pm}_1}),
$$
and thus
$$
w=(u_{\bx_0}^*)^{\phi}_{0}\quad \text{on }\partial B_1.
$$
Since both $w$ and $(u_{\bx_0}^*)^{\phi}_{0}$ are solutions of the
Signorini problem, they must coincide also in $B_1$.  Therefore
$$
(u_{\bx_j}^*)_{r_j}^\phi\ra (u_{\bx_0}^*)^{\phi}_{0}\quad \text{in
}C^1(\overline{B_1^{\pm}}),
$$
implying also that
$$
\|(u_{\bx_j}^*)_{r_j}^\phi-(u_{\bx_j}^*)^{\phi}_{0}\|_{C^1(\overline{B_1^{\pm}})}\to 0,
$$
which contradicts our assumption.

\medskip\noindent \emph{Step 2.}  For a given
$\e>0$ and a unit vector $\nu\in \R^{n-1}$ define the cone
$$
\cC_{\e}(\nu)=\{x'\in \R^{n-1}: x'\cdot\nu > \e|x'|\}.
$$ By Lemma~\ref{lem:var-blowup-hol}, we may assume
$a_{\bx}\ge\frac{a_{x_0}}{2}$ for $\bx\in B'_{\rho}(x_0)\cap \Ga(U)$
by taking $\rho$ small. For such $\rho$, we then claim that for any
$\e>0$, there is $r_{\e}>0$ such that for any
$\bx\in B'_{\rho}(x_0)\cap \Ga(U)$, we have
$$\cC_{\e}(\nu_{\bx})\cap B'_{r_{\e}}\subset
\{u_{\bx}^*(\cdot, 0)>0\}.$$ Indeed, denoting
$\cK_{\e}(\nu)=\cC_{\e}\cap \partial B'_{1/2}$, we have for some
universal $C_{\e}>0$
\begin{align*}
  \cK_{\e}(\nu_{\bx})\Subset\{(u_{\bx}^*)_0^\phi(\cdot,0)>0\}
  \cap B'_1\quad \text{and}\quad (u_{\bx}^*)_0^\phi(\cdot, 0)\ge
  a_{\bx}C_{\e}\ge \frac{a_{x_0}}{2}C_{\e}\quad \text{on }
  \cK_{\e}(\nu_{\bx}). 
\end{align*}
Since $\frac{a_{x_0}}{2}C_{\e}$ is independent of $\bx$, by Step 1 we
can find $r_{\e}>0$ such that for $r<2r_{\e}$,
$$
(u_{\bx}^*)^{\phi}_{r}(\cdot, 0)>0\quad \text{on } \cK_{\e}(\nu_{\bx}).
$$
This implies that for $r<2r_{\e}$,
$$
u_{\bx}^*(\cdot, 0)>0\quad \text{on}\quad
r\cK_{\e}(\nu_{\bx})=\cC_{\e}(\nu_{\bx})\cap \partial
  B'_{r/2}.
$$
Taking the union over all $r<2r_{\e}$, we obtain
$$
    u_{\bx}^*(\cdot, 0)>0\quad \text{on }\cC_{\e}(\nu_{\bx})\cap
  B'_{r_{\e}}.
$$

\medskip\noindent \emph{Step 3.}  We claim that for given $\e>0$,
there exists $r_{\e}>0$ such that for any
$\bar{x}\in B_{\rho}'(x_0)\cap \Gamma(U)$, we have
$-\left(\cC_{\e}(\nu_{\bar{x}})\cap B'_{r_{\e}}\right)\subset
\{u_{\bx}^*(\cdot, 0)=0\}$.

Indeed, we first note that
$$
-\pa^+_{x_n}(u_{\bx}^*)_0^\phi\ge a_{\bx}C_{\e} >
\left(\frac{a_{x_0}}{2}\right)C_{\e}\quad \text{on
}-\mathcal{K}_{\e}(\nu_{\bx}),
$$ 
for a universal constant
$C_{\e}>0$. From Step 1, there exists $r_{\e}>0$ such that for
$r<2r_{\e}$,
$$
-\pa^+_{x_n}(u_{\bx}^*)^{\phi}_{r}(\cdot, 0)>0\quad \text{on
}-\mathcal{K}_{\e}(\nu_{\bx}).
$$
By arguing as in Step 2, we obtain
$$
-\pa^+_{x_n}u_{\bx}^*(\cdot, 0)>0\quad\text{on } -\left(\cC(\nu_{\bx})\cap
  B'_{r_{\e}}\right).
$$
By the complementarity condition in Lemma~\ref{u_ex0-comp-cond}, we
therefore conclude that
$$
-\left(\cC(\nu_{\bx})\cap B'_{r_{\e}}\right) \subset
\{-\pa^+_{x_n}u_{\bx}^*(\cdot, 0)>0\}\subset \{u_{\bx}^*(\cdot, 0)=0\}.
$$

\medskip\noindent \emph{Step 4.} By direct computation, we have
$$
\cC_{\Ld^{1/2}\la^{-1/2}\e}(\nu_{\bx}^A)\cap
B'_{\la^{1/2}r_\e}\subset \bar{\mfa}_{\bx}\left(\cC_\e(\nu_{\bx})\cap
  B'_{r_\e}\right), 
$$
where
$$
\nu_{\bx}^A:=\frac{(\bar{\mfa}_{x}^{-1})^{\rm{tr}}\nu_{\bx}}{|(\bar{\mfa}_{x}^{-1})^{\rm{tr}}\nu_{\bx}|}.
$$
(Here $(\cdot)^{\rm{tr}}$ stands for the transpose of the matrix.)
Indeed, if $y'\in \cC_{\Ld^{1/2}\la^{-1/2}\e}(\nu_{\bx}^A)\cap
B'_{\la^{1/2}r_\e}$, then
\begin{align*}
y'\in
  B'_{\la^{1/2}r_\e}
  &=\bar{\mfa}_{\bx}\left(\bar{\mfa}_{\bx}^{-1}B'_{\la^{1/2}r_\e}\right)\subset
    \bar{\mfa}_{\bx}B'_{r_\e},\\
\intertext{and}
  \langle \bar{\mfa}_{x}^{-1}y',\nu_{\bx}\rangle
  &=\langle y',(\bar{\mfa}_{x}^{-1})^{\rm{tr}}\nu_{\bx}\rangle=\langle
    y',\nu_{\bx}^A\rangle |(\bar{\mfa}_{x}^{-1})^{\rm{tr}}\nu_{\bx}|\\
  &\ge (\Ld^{1/2}\la^{-1/2}\e|y'|)(\Ld^{-1/2})\\
  &=\la^{-1/2}\e|y'|\ge \e|\bar{\mfa}_{\bx}^{-1}y'|.
\end{align*}
Combining this with Step 2 and Step 3, for $\bx\in B'_\rho(x_0)\cap \Ga(U)$, \begin{align*}
    \bx+\left(\cC_{\Ld^{1/2}\la^{-1/2}\e}(\nu_{\bx}^A)\cap B'_{\la^{1/2}r_\e}\right) &\subset \bx+\bar{\mfa}_{\bx}\left(\cC_\e(\nu_{\bx})\cap B'_{r_\e}\right)\\
    &\subset \{U_{\bx}^*(\cdot,0)>0\},\\
    \bx-\left(\cC_{\Ld^{1/2}\la^{-1/2}\e}(\nu_{\bx}^A)\cap B'_{\la^{1/2}r_\e}\right) &\subset\{U_{\bx}^*(\cdot,0)=0\}.
\end{align*}

\medskip\noindent \emph{Step 5.}  By rotation in $\R^{n-1}$ we may
assume $\nu_{x_0}^A=e_{n-1}$. For any $\e>0$, by
Lemma~\ref{lem:var-blowup-hol} and the H\"older continuity of $A$, we can take
$\rho_{\e}=\rho(x_0, \e, M)$, possibly smaller than $\rho$ in the
previous steps, such that
$$
\cC_{2\Ld^{1/2}\la^{-1/2}\e}(e_{n-1})\cap B'_{\la^{1/2}r_{\e}} \subset \cC_{\Ld^{1/2}\la^{-1/2}\e}(\nu_{\bx}^A)\cap
B'_{\la^{1/2}r_{\e}},
$$
for $\bx\in B'_{\rho_{\e}}(x_0)\cap \Ga(U)$.
By Step 4, we also have
\begin{align*}
  \bx+\left(\cC_{2\Ld^{1/2}\la^{-1/2}\e}(e_{n-1})\cap B'_{\la^{1/2}r_{\e}}\right) \subset \{U(\cdot, 0)>0\},\\
  \bx-\left(\cC_{2\Ld^{1/2}\la^{-1/2}\e}(e_{n-1})\cap B'_{\la^{1/2}r_{\e}}\right) \subset \{U(\cdot, 0)=0\}.
\end{align*}
Now, fixing $\e=\e_0$, by the standard arguments, we conclude that
there exists a Lipschitz function $g: \R^{n-2}\ra \R$ with
$|\D g|\le C_{n, M}/\e_0$ such that
\begin{align*}
  B'_{\rho_{\e_0}}(x_0) \cap \{U(\cdot, 0)=0\}=B'_{\rho_{\e_0}}(x_0)\cap \{x_{n-1}\le g(x'')\},\\
  B'_{\rho_{\e_0}}(x_0) \cap \{U(\cdot,
  0)>0\}=B'_{\rho_{\e_0}}(x_0)\cap
  \{x_{n-1}> g(x'')\}.
\end{align*}

\medskip\noindent \emph{Step 6.}  Taking $\e\ra 0$ in Step 5, $\Ga(U)$
is differentiable at $x_0$ with normal $\nu_{x_0}^A$. Recentering at any
$\bx\in B'_{\rho_{\e_0}}(x_0)\cap \Ga(U)$, we see that $\Ga(U)$ has a
normal $\nu_{\bx}^A$ at $\bx$. By noticing that $\bx\mapsto
\nu_{\bx}^A$ is $C^{0,\gamma}$, we conclude that the function $g$ in Step 5 is
$C^{1, \g}$. This completes the proof.
\end{proof}


\section{Singular points}
\label{sec:var-singular-points}

In this section we study another type of free boundary points for
almost minimizers, the so-called singular set $\Sigma(U)$. Because of the
machinery developed in the earlier sections, we are able to prove a
stratification type result for $\Sigma(U)$
(Theorem~\ref{thm:var-sing}), following a similar approach for the
minimizers and almost minimizers with $A=I$.

\begin{definition}[Singular points]
  Let $U$ be an $A$-quasisymmetric almost minimizer for the $A$-Signorini problem  in
  $B_{1}$. We say that a free boundary point $x_0$ is \emph{singular} if
  the coincidence set $\Lambda(U)=\{U(\cdot,0)=0\}\subset B'_{1}$ has
  zero $H^{n-1}$-density at $x_0$, i.e.,
$$
\lim_{r\ra 0+}\frac{H^{n-1}\left(\Ld(U)\cap
    B_r'(x_0)\right)}{H^{n-1}(B'_r)}=0.
$$
We denote the set of all singular points by $\Sigma(U)$ and call it
the \emph{singular set}.
\end{definition}

Denote by
$\bar{\mfa}_{x_0}'$ the $(n-1)\times(n-1)$ submatrix of $\bar{\mfa}_{x_0}$ formed by the first $(n-1)$ rows and columns. We then claim
that there are constants $C,c>0$ depending only on $n$, $\lambda$, and
$\Lambda$ such that
\begin{align}\label{eq:det-(n-1)-est}
  c\le|\det\bar{\mfa}_{x_0}'|\le C.
\end{align}
Indeed, this follows from the ellipticity of $\mfa_{x_0}$ and the
invariance of both   $\R^{n-1}\times\{0\}$ and $\{0\}\times\R$ under
$\bar{\mfa}_{x_0}$, since we have
$$
|\det\bar{\mfa}_{x_0}'(\bar{\mfa}_{x_0})_{nn}|=|\det\bar{\mfa}_{x_0}|
=|\det \mfa_{x_0}|
$$
and
$$
|(\bar{\mfa}_{x_0})_{nn}|=|\langle{\bar{\mfa}_{x_0}e_n,e_n}\rangle|=|\bar{\mfa}_{x_0}e_n|\in[\lambda^{1/2},\Lambda^{1/2}].
$$
Recall now that for $x_0\in \Ga(u)$, $u_{x_0}(y)=U(\bar{\mfa}_{x_0}y+x_0)$
and note that $\bar{\mfa}_{x_0}'B_r'+x_0=E'_r(x_0)$.
Thus,
\begin{equation}\label{eq:sing-set-measure}
  H^{n-1}(\Ld(U)\cap E_r'(x_0))
  =|\det\bar{\mfa}_{x_0}'|H^{n-1}(\Ld(u_{x_0}^*)\cap
  B'_r).
\end{equation}
Now, by \eqref{eq:sing-set-measure} and \eqref{eq:det-(n-1)-est}, together with $B_{\la^{1/2}r}(x_0)\subset E_r(x_0)\subset B_{\Ld^{1/2}r}(x_0)$, we have \begin{align*}
   \lim_{r\ra 0+}\frac{H^{n-1}\left(\Ld(U)\cap
    B_r'(x_0)\right)}{H^{n-1}(B'_r)}=0 &\Longleftrightarrow \lim_{r\ra 0+}\frac{H^{n-1}\left(\Ld(U)\cap
    E_r'(x_0)\right)}{H^{n-1}(E_r'(x_0))}=0\\
    &\Longleftrightarrow\lim_{r\ra 0+}\frac{H^{n-1}\left(\Ld(u_{x_0}^*)\cap
    B_r'\right)}{H^{n-1}(B'_r)}=0.
\end{align*}
In terms of Almgren rescalings $(u_{x_0}^*)_r^I$, we can rewrite the condition above as 
$$
\lim_{r\ra 0+}H^{n-1}\left(\Ld((u_{x_0}^*)_r^I)\cap
    B_1'\right)=0.
$$
We then have the following characterization of singular points.

\begin{proposition}[Characterization of singular points]\label{prop:var-sing-char}
  Let $U$ be an $A$-quasi\-sym\-met\-ric almost minimizer for the $A$-Signorini problem  in $B_{1}$,
  and $x_0\in B_{1/2}'\cap \Gamma(U)$ be such that
  $\ka(x_0)=\ka<\kappa_0$. Then the following statements
  are equivalent.
  \begin{enumerate}[label=\textup{(\roman*)},leftmargin=*,widest=iii]
  \item $x_0\in \Sigma(U)$.
  \item any Almgren blowup $(u_{x_0}^*)_0^I$ of $u_{x_0}^*$ at $0$ is a nonzero polynomial
    from the class
    \begin{multline*}
      \mathcal{Q}_{\ka}=\{q: \text{$q$ is homogeneous polynomial of
        degree
        $\kappa$ such that}\\
      \La q=0,\ q(y', 0)\ge 0,\ q(y', y_n)=q(y', -y_n)\}.
    \end{multline*}
  \item any Almgren blowup $(U_{x_0}^*)^A_{x_0,0}$ of $U_{x_0}^*$ at $x_0$ is a nonzero polynomial
    from the class
    \begin{multline*}
      \mathcal{Q}^{A,x_0}_{\ka}=\{p: \text{$p$ is homogeneous polynomial of
        degree
        $\kappa$ such that}\\
      \dv(A(x_0)\D p)=0,\ p(x', 0)\ge 0,\ p(x)=p(P_{x_0}x)\}.
    \end{multline*}   
  \item $\ka(x_0)=2m$ for some $m\in \mathbb{N}$.
  \end{enumerate}
\end{proposition}
\begin{proof}This is the analogue of \cite{JeoPet19a}*{Proposition~10.2} in the case $A\equiv I$.

Clearly, (ii) and (iii) are equivalent. By Proposition~\ref{prop:var-exist-Alm-blowup}, any Almgren blowup $(u_{x_0}^*)_0^I$ of $u_{x_0}^*$ at $0$ is a nonzero global solution of the Signorini problem, homogeneous of degree $\ka$. Moreover, $(u_{x_0}^*)_0^I$ is a $C^1_{\loc}$ limit of Almgren rescalings $(u_{x_0}^*)_{t_j}^I$ in $\R^n_\pm\cup \R^{n-1}$. Since $u_{x_0}^*$ also satisfies the complementarity condition in Lemma~\ref{u_ex0-comp-cond}, the equivalence among (i), (ii) and (iv) follows by repeating the arguments in \cite{JeoPet19a}*{Proposition~10.2}.
\end{proof}

In order to proceed with the blowup analysis at singular points, we need to remove the logarithmic term from the growth estimates in Lemma~\ref{lem:var-almost-opt-growth}. This was achieved in \cite{JeoPet19a}*{Lemma~10.8} in the case $A\equiv I$ by using a bootstrapping argument \cite{JeoPet19a}*{Lemmas~10.4--10.6, Corollary~10.7}, based on the $\log$-epiperimetric inequality of \cite{ColSpoVel20}. All the arguments above work directly for $u_{x_0}^*$ (and then for $U_{x_0}^*$, by deskewing) and we obtain the following optimal growth estimate.

\begin{lemma}[Optimal growth estimate at singular points]
  \label{lem:var-opt-est}
  Let $U$ be an $A$-quasi\-sym\-met\-ric almost minimizer for the
  $A$-Signorini problem  in $B_{1}$. If $x_0\in B_{1/2}'\cap \Ga(U)$
  and $\ka(x_0)=\kappa<\ka_0$, $\ka=2m$, $m\in \mathbb{N}$, then there
  are $t_0$ and $C$, depending on $n$, $\al$, $M$, $\ka$, $\ka_0$, $\|U\|_{W^{1,2}(B_{1})}$, such that for $0<t<t_0$,
\begin{alignat*}{2}
   \int_{\pa B_t}(u_{x_0}^*)^2&\le
   Ct^{n+2\ka-1},&\quad\int_{B_t}|\D u_{x_0}^*|^2&\le
   Ct^{n+2\ka-2},\\
    \int_{\pa E_t(x_0)}(U_{x_0}^*)^2&\le
    Ct^{n+2\ka-1},&\quad\int_{E_t(x_0)}|\D U_{x_0}^*|^2&\le
    Ct^{n+2\ka-2}.
  \end{alignat*}
\end{lemma}

With this growth estimate at hand, we now proceed as in the beginning of Section~\ref{sec:var-32-homog-blow} but with $\kappa=2m<\kappa_0$ in place of $\kappa=3/2$. Namely, for such $\kappa$, let 
$$
\phi(r)=\phi_\ka(r):=e^{-\left(\frac{\ka b}\al\right)r^{\al}}r^\ka,\quad 0<r<t_0,
$$ 
where $b=\frac{M(n+2\ka_0)}\al$ is as in Weiss-type monotonicity formula. Then, define the \emph{$\ka$-almost homogeneous rescalings} of a function $v$ at $x_0$ by $$
v^\phi_{x_0,r}(x):=\frac{v(rx+x_0)}{\phi(r)}.
$$
Again, when $x_0=0$, we simply write $v_{0,r}^\phi=v_r^\phi$.

The growth estimates in Lemma~\ref{lem:var-opt-est} enable us to consider \emph{$\ka$-homogeneous blowups} \begin{align*}
(u_{x_0}^*)^\phi_{t_j}\to(u_{x_0}^*)^\phi_{0}\quad\text{in } C^1_{\loc}(\R^n_\pm\cup\R^{n-1}),\\
(U_{x_0}^*)^\phi_{x_0,t_j}\to(U_{x_0}^*)^\phi_{x_0,0}\quad\text{in } C^1_{\loc}(\R^n_\pm\cup\R^{n-1}),
\end{align*}
for $t=t_j\to 0+$, similar to $3/2$-homogeneous blowups in Section~\ref{sec:var-32-homog-blow}. 

Furthermore, the arguments in \cite{JeoPet19a}*{Proposition~10.10} also go through for $u_{x_0}^*$ (and then for $U_{x_0}^*$, by deskewing), and we obtain the following rotation estimate for almost homogeneous rescalings.

\begin{proposition}[Rotation estimate]\label{prop:var-blowup-rot-est} For $U$ and $x_0$ as in Lemma~\ref{lem:var-opt-est}, there exist $C>0$ and $t_0>0$ such
  that
\begin{align*} \int_{\pa B_1}|(u_{x_0}^*)^{\phi}_{t}-(u_{x_0}^*)^{\phi}_{s}|&\le C\left(\log\frac
  1t\right)^{-\frac{1}{n-2}},\\
   \int_{\mfa_{x_0}\pa B_1}|(U_{x_0}^*)^{\phi}_{x_0,t}-(U_{x_0}^*)^{\phi}_{x_0,s}|&\le C\left(\log\frac
    1t\right)^{-\frac{1}{n-2}},
\end{align*}
for $0<s<t<t_0$.
In particular, the blowups $(u_{x_0}^*)^{\phi}_{0}$ and $(U_{x_0}^*)^{\phi}_{x_0,0}$ are unique.
\end{proposition}

We next show that the rotation estimate as above holds uniformly for $u_{x_0}^*$ replaced with its Almgren rescalings $(u^*_{x_0})^I_{r}$, $0<r<1$. (Note that the objects $\left[(u_{x_0}^*)_{r}^I\right]^\phi_t$ in the proposition below are $\kappa$-almost homogeneous rescalings of Almgren rescalings.)

\begin{proposition}\label{prop:var-alm-blowup-rot-est} For $U$ and $x_0$ as in Lemma~\ref{lem:var-opt-est} and $0<r<1$, there are $C>0$ and $t_0>0$, independent of $r$ such that 
$$
\int_{\pa B_1}\left|\left[(u_{x_0}^*)_{r}^I\right]^\phi_t-\left[(u_{x_0}^*)_{r}^I\right]^\phi_s\right|\le C\left(\log\frac 1t\right)^{-\frac{1}{n-2}},
$$
for $0<s<t<t_0$. In particular, the $\kappa$-homogeneous blowup $\left[(u_{x_0}^*)_{r}^I\right]^\phi_0$ is unique.
\end{proposition}

\begin{proof}
We first observe that since $u_{x_0}^*$ has the almost Signorini property
at $0$, $(u_{x_0}^*)^I_{r}$ also has the almost Signorini property at
$0$. This implies that $W_\ka(\rho,(u_{x_0}^*)^I_{r},0)$ and
$\widehat{N}_{\ka_0}(\rho,(u_{x_0}^*)^I_{r},0)$ are monotone
nondecreasing on $\rho$. Thus
\begin{align*}
  \widehat{N}_{\ka_0}(0+,(u_{x_0}^*)^I_{r},0)
  &=\lim_{\rho\to
    0}\widehat{N}_{\ka_0}(\rho,(u_{x_0}^*)^I_{r},0)=\lim_{\rho\to
    0}\widehat{N}_{\ka_0}(\rho r,u_{x_0}^*,0)\\
  &=\ka(x_0)=\ka.
 \end{align*}
Fix $R>1$. If $t$ is small, then we can argue as in the proof
of Proposition~\ref{prop:var-exist-Alm-blowup}
to obtain that for any $K\Subset B^\pm_R\cup B
_R'$, $$
\left\| \left[(u_{x_0}^*)^I_{r}\right]^\phi_t\right\|_{C^{1,\be}(K)}\le C(n,\al,M,R,K)\left\|\left[(u_{x_0}^*)^I_{r}\right]^\phi_t\right\|_{W^{1,2}(B_R)}.
$$
Those are all we need to proceed all the arguments with
$(u_{x_0}^*)^I_{r}$
as in Lemmas~10.4--10.6, Corollary~10.7, Lemma~10.8, and Proposition~10.10 in \cite{JeoPet19a}. This completes the proof.
\end{proof}

Once we have Proposition~\ref{prop:var-alm-blowup-rot-est}, we can argue as in \cite{JeoPet19a}*{Lemma~10.11} to obtain the nondegeneracy for $u^*_{x_0}$, and also for $U^*_{x_0}$.
\begin{lemma}[Nondegeneracy at singular points]\label{lem:var-non-deg} Let $U$ and $x_0$ be as in Lemma~\ref{lem:var-opt-est}. Then \begin{align*}
\liminf_{t\to 0}\int_{\pa B_1}((u_{x_0}^*)^\phi_{t})^2=\liminf_{t\to 0}\frac 1{t^{n+2\ka-1}}\int_{\pa B_t}(u_{x_0}^*)^2>0,\\
\liminf_{t\to 0}\int_{\mfa_{x_0}\pa B_1}((U_{x_0}^*)^\phi_{x_0,t})^2=\liminf_{t\to 0}\frac 1{t^{n+2\ka-1}}\int_{\pa E_t(x_0)}(U_{x_0}^*)^2>0.
\end{align*}
\end{lemma}

\medskip To state our main result on the singular set, we need to
introduce certain subsets of $\Sigma(U)$.
For $\ka=2m<\kappa_0$, $m\in \mathbb{N}$, let
$$
\Sigma_{\ka}(U):=\{x_0\in \Si(U):\kappa(x_0)=\ka\}=\Gamma_\kappa(U).
$$
Note that the last equality follows from the implication (iv)
$\Rightarrow$ (i) in 
Proposition~\ref{prop:var-sing-char}.

\begin{lemma}\label{lem:var-sing-closed}
  The set $\Sigma_{\ka}(U)$ is of topological type $F_{\si}$; i.e., it
  is a countable union of closed sets.
\end{lemma}

\begin{proof}
For $j\in\mathbb{N}$, $j\geq 2$, let
\begin{multline*}
  F_j:=\Bigl\{x_0\in \Sigma_{\ka}(U)\cap \overline{B_{1-1/j}}:\\
  \frac1j\le \frac{1}{\rho^{n+2\ka-1}}\int_{\pa E_{\rho}(x_0)}(U_{x_0}^*)^2 \le
  j\ \text{for}\ 0<\rho<\frac{1}{2j}\Bigr\}.
\end{multline*}
Note that if $x_j\to x_0$, then by the local uniform continuity of $U$ and $A$, $$
\int_{\pa E_\rho(x_i)}(U_{x_i}^*)^2\to \int_{\pa E_\rho(x_0)}(U_{x_0}^*)^2.
$$
Using this, together with Lemma~\ref{lem:var-opt-est}, Lemma~\ref{lem:var-non-deg} and Lemma~\ref{lem:var-almost-opt-growth}, we can argue as in \cite{JeoPet19a}*{Lemma~10.12} to prove that $\Sigma_\ka(U)=\cup_{j=2}^\infty F_j$ and each $F_j$ is closed.
\end{proof}

Next, for $\ka=2m<\kappa_0$,
$m\in \mathbb{N}$ and $x_0\in \Sigma_\kappa(U)$, we define
$$
d^{(\ka)}_{x_0}:=\dim\{\xi\in \R^{n-1}:\xi\cdot
\D_{y'}(u_{x_0}^*)_{0}^\phi(y', 0)\equiv 0\ \text{on}\ \R^{n-1}\},
$$
which has the meaning of the dimension of $\Sigma_\kappa(u_{x_0}^*)$ at $0$,
and where $(u_{x_0}^*)_{0}^\phi$ is the unique $\kappa$-homogeneous blowup of $u_{x_0}^*$ at
$0$. We note here that $d_{x_0}^{(\kappa)}$ can only take the values
$0,1,\dots, n-2$. Indeed, otherwise $(u_{x_0}^*)_{0}^\phi$ would
vanish identically on $\Pi$ and consequently on $\R^n$, since it is a
solution of the Signorini problem, even symmetric with respect to
$\Pi$ (see \cite{GarPet09}). However, that would contradict the
nondegeneracy Lemma~\ref{lem:var-non-deg}.
Then, for $d=0,1,\dots,n-2$, let
$$
\Si^d _{\ka}(U):=\{x_0\in \Si_{\ka}(U): d^{(\ka)}_{x_0}=d\}.
$$

\begin{theorem}[Structure of the singular set]\label{thm:var-sing}
  Let $U$ be an $A$-quasisymmetric almost minimizer for the $A$-Signorini problem  in
  $B_{1}$. Then for every $\ka=2m<\kappa_0$, $m\in \mathbb{N}$, and
  $d=0,1,\dots,n-2$, the set $\Sigma_{\ka}^d(U)$ is contained in the
  union of countably many submanifolds of dimension $d$ and class
  $C^{1, \log}$.
\end{theorem}

\begin{proof}
We follow the idea in \cite{JeoPet19a}*{Theorem~10.13}. For $x_0\in\Sigma_\ka(U)\cap B'_{1/2}$, let $q_{x_0}\in \mathcal{Q}_\ka$ denote the unique $\ka$-homogeneous blowup of $u_{x_0}^*$ at $0$. By the optimal growth (Lemma~\ref{lem:var-opt-est}) and the nondegeneracy (Lemma~\ref{lem:var-non-deg}), we can write 
  $$
  q_{x_0}=\eta_{x_0} q_{x_0}^I,\quad \eta_{x_0}>0,\quad
  \|q_{x_0}^I\|_{L^2(\partial B_1)}=1,
  $$
where $q_{x_0}^I\in \mathcal{Q}_\ka$ is the corresponding Almgren
blowup. If $x_1$, $x_2\in \Sigma_\ka(U)\cap B'_{1/2}$, for $t>0$, to be
chosen below, we can write
\begin{align}\label{eq:var-sing-1}
  \|q_{x_1}-q_{x_2}\|_{L^1(\pa B_1)}
  &\begin{multlined}[t]                                                            \le \|q_{x_1}-(u_{x_1}^*)_t^\phi\|_{L^1(\pa B_1)}+\|(u_{x_1}^*)_t^\phi-(u_{x_2}^*)_t^\phi\|_{L^1(\pa B_1)}\\
    +\|q_{x_2}-(u_{x_2}^*)_t^\phi\|_{L^1(\pa B_1}
  \end{multlined}
\\\nonumber
&\le C\left(\log\frac 1t\right)^{-\frac
  1{n-2}}+\|(u_{x_1}^*)_t^\phi-(u_{x_2}^*)_t^\phi\|_{L^1(\pa
  B_1)},
\end{align}
where we have used Proposition~\ref{prop:var-blowup-rot-est} in the
second inequality. Moreover, we have
\begin{multline}\label{eq:var-sing-2}
  \|(u_{x_1}^*)_t^\phi-(u_{x_2}^*)_t^\phi\|_{L^1(\pa B_1)}\\*
\begin{aligned}  
&=\frac 1{2\phi(t)}\int_{\pa B_1}|U(t\bar{\mfa}_{x_1}y+x_1)+U(P_{x_1}(t\bar{\mfa}_{x_1}y+x_1))\\
&\qquad\qquad -U(t\bar{\mfa}_{x_2}y+x_2)-U(P_{x_2}(t\bar{\mfa}_{x_2}y+x_2))|\,dS_y\\
&\le \frac C{t^\ka}\int_{\pa B_1}\Big(|U(t\bar{\mfa}_{x_1}y+x_1)-U(t\bar{\mfa}_{x_2}y+x_2)|\\
&\qquad\qquad +|U(P_{x_1}(t\bar{\mfa}_{x_1}y+x_1))-U(P_{x_1}(t\bar{\mfa}_{x_2}y+x_2))|\\
&\qquad\qquad +|U(P_{x_1}(t\bar{\mfa}_{x_2}y+x_2))-U(P_{x_2}(t\bar{\mfa}_{x_2}y+x_2))|\Big)\,dS_y\\
&\le \frac C{t^\ka}\|\D
U\|_{L^\infty(B_1)}\left(|\bar{\mfa}_{x_1}-\bar{\mfa}_{x_2}|+|x_1-x_2|+|P_{x_1}-P_{x_2}|\right)\\
&\le C\frac{|x_1-x_2|^\al}{t^\ka}= C|x_1-x_2|^{\al/2},
\end{aligned}
\end{multline}
if we choose $t=|x_1-x_2|^{\frac\al{2\ka}}$ and have $|x_1-x_2|<(1/4\Ld^{-1}\la^{1/2})^{\frac{2\ka}\al}$. Combining \eqref{eq:var-sing-1} and \eqref{eq:var-sing-2}, we obtain $$
\|q_{x_1}-q_{x_2}\|_{L^1(\pa B_1)}\le C\left(\log\frac 1{|x_1-x_2|}\right)^{-\frac 1{n-2}}.
$$
After this, we can repeat the argument in the proof of \cite{JeoPet19a}*{Theorem~10.13} to obtain the estimates that for $x_0\in\Sigma_\ka(U)\cap B'_{1/2}$, there is $\de=\de(x_0)>0$ such that \begin{align*}
    |\eta_{x_1}-\eta_{x_2}|&\le C\left(\log\frac 1{|x_1-x_2|}\right)^{-\frac 1{2(n-2)}},\\
    \|q_{x_1}^I-q_{x_2}^I\|_{L^\infty(B_1)}&\le C\left(\log\frac 1{|x_1-x_2|}\right)^{-\frac 1{2(n-2)}}, \quad x_1, x_2\in \Sigma_\ka(U)\cap B_\de(x_0).
\end{align*}
Now, we also have the similar result for $U_{x_0}^*$. For
$x_0\in\Sigma_\ka(U)\cap B'_{1/2}$, where $\ka=2m$, $m\in \mathbb{N}$,
let $p_{x_0}\in \mathcal{Q}_\ka^{A,x_0}$ be the unique $\ka$-homogeneous
blowup of $U_{x_0}^*$ at $x_0$. Then we can write
$$
p_{x_0}=\eta_{x_0}^A p_{x_0}^A,\quad \eta_{x_0}^A>0,\quad
\|p_{x_0}^A\|_{L^2(\pa B_1)}=1,
$$
where $p_{x_0}^A\in \mathcal{Q}^{A,x_0}_\ka$ is the corresponding Almgren
blowup of $U_{x_0}^*$. Using that
$$
q_{x_0}^I(z)=\left(\det\mfa_{x_0}\right)^{1/2}p_{x_0}^A(\mfa_{x_0}z),\quad
q_{x_0}(z)=p_{x_0}(\mfa_{x_0}z),
$$
together with the ellipticity and H\"older continuity of $\mfa_{x_0}$ and the homogeneity of blowups, we easily
conclude that for $x_0\in\Sigma_\ka(U)\cap B'_{1/2}$, there is
$\de=\de(x_0)>0$ such that
\begin{align*} 
    |\eta^A_{x_1}-\eta^A_{x_2}|&\le C\left(\log\frac 1{|x_1-x_2|}\right)^{-\frac 1{2(n-2)}},\\
    \|p_{x_1}^A-p_{x_2}^A\|_{L^\infty(B_1)}&\le C\left(\log\frac 1{|x_1-x_2|}\right)^{-\frac 1{2(n-2)}}, \quad x_1, x_2\in \Sigma_\ka(U)\cap B_\de(x_0).
\end{align*}
Once we have these estimates, as well as Lemma~\ref{lem:var-sing-closed}, we can apply the Whitney
Extension Theorem of Fefferman \cite{Fef09}, to complete the proof, similar to that of Theorem~1.7 in \cite{ColSpoVel20}.
\end{proof}


\appendix
\section{Example of almost minimizers}

\begin{example}\label{ex:drift}
Let $U$ be a solution of the $A$-Signorini problem in $B_1$ with velocity field $b\in L^p(B_1)$, $p>n$: 
  \begin{align*}
    -\dv(A\D U)+ \langle b(x),\nabla U\rangle=0
    \quad\text{in }B_1^{\pm},\\
    \begin{multlined}U\geq 0,\quad \langle A\D U,\nu^+\rangle+\langle A\D U,\nu^-\rangle\geq 0,\\
      \qquad\qquad\qquad\qquad\qquad U(\langle A\D U,\nu^+\rangle+\langle A\D U,\nu^-\rangle)=0\quad\text{on }B_1',
    \end{multlined}
  \end{align*}
where $\nu^\pm=\mp e_n$ and $\langle A\D U,\nu^\pm\rangle$ on $B_1'$ are understood as the limits from inside $B_1^\pm$.
We interpret this in the weak sense that
  $U$ satisfies the variational inequality
$$
\int_{B_1}\mean{A\nabla U, \nabla(W-U)}+\langle b,\nabla U\rangle (W-U)\geq 0,
$$
for any competitor $W\in\mathfrak{K}_{0,U}(B_1,\Pi)$. Then $U$ is an almost minimizer of the $A$-Signorini problem in $B_1$ with thin obstacle $\psi=0$ on $\Pi=\R^{n-1}\times\{0\}$ and a gauge function $\om(r)=Cr^{1-n/p}$, $C=C(n, p, \la, \Ld)\|b\|^2_{L^p(B_1)}$.
\end{example}

\begin{proof}
For any $E_r(x_0)\Subset B_1$ and $W\in
\mathfrak{K}_{0, U}(E_r(x_0),\Pi)$, we extend $W$ as equal to $U$ in
$B_1\sm E_r(x_0)$ to obtain
\begin{equation}\label{ex-u-var-ineq}
\int_{E_r(x_0)}\langle A\D U, \D(W-U)\rangle+\langle b,\D U\rangle (W-U)\ge 0.
\end{equation}
Let $V$ be the minimizer of the energy functional 
$$
\int_{E_r(x_0)}\langle A\D V, \D V\rangle\quad \text{on }\mathfrak{K}_{0, U}(E_r(x_0),\Pi).
$$
Then it follows from a standard variation argument that $V$ satisfies the variational inequality \begin{equation}\label{ex-v-var-ineq}
\int_{E_r(x_0)}\langle A\D V, \D(W-V)\rangle \ge 0\quad\text{for any }W\in \mathfrak{K}_{0, U}(E_r(x_0), \Pi).
\end{equation}
Taking $W=U\pm(U-V)^+$ in \eqref{ex-u-var-ineq} and $W=V+(U-V)^+$ in \eqref{ex-v-var-ineq}, we obtain $$
\int_{E_r(x_0)}\langle A\D(U-V)^+, \D(U-V)^+\rangle \le -\int_{E_r(x_0)}\langle b,\D U\rangle (U-V)^+.
$$
Similarly, taking $W=U+(V-U)^+$ in \eqref{ex-u-var-ineq} and $W=V\pm(V-U)^+$ in \eqref{ex-v-var-ineq}, we get $$
\int_{E_r(x_0)}\langle A\D(V-U)^+, \D(V-U)^+\rangle\le \int_{E_r(x_0)}\langle b,\D U\rangle (V-U)^+.
$$
These two inequalities give $$
\int_{E_r(x_0)}\langle A\D(U-V), \D(U-V)\rangle \le \int_{E_r(x_0)}|b||\D U||U-V|.
$$
Applying H\"older's inequality,
\begin{align*}
\int_{E_r(x_0)}|\D(U-V)|^2&\le \lambda^{-1}\int_{E_r(x_0)}\langle A\D(U-V), \D(U-V)\rangle\\
&\le \lambda^{-1}\|b\|_{L^p(E_r(x_0))}\|\D U\|_{L^2(E_r(x_0))}\|U-V\|_{L^{p^*}(E_r(x_0))},
\end{align*}
with $p^*=2p/(p-2)$. Since $U-V\in W_0^{1, 2}(E_r(x_0))$ and $\operatorname{diam}(E_r(x_0))\le 2\Ld^{1/2}r$, from the Sobolev's inequality, $$
\|U-V\|_{L^{p^*}(E_r(x_0))}\le C(n, p, \lambda,\Lambda)r^{1-n/p}\|\D(U-V)\|_{L^2(E_r(x_0))}.
$$
Now we have
\begin{equation}\label{ex-(u-v)-estimate}
\int_{E_r(x_0)}|\D(U-V)|^2\le Cr^{2(1-n/p)}\int_{E_r(x_0)}|\D U|^2,
\end{equation}
with $C=C(n, p, \la, \Ld)\|b\|^2_{L^p(B_1)}$.
Thus,
\begin{multline*}
\int_{E_r(x_0)}\langle A\D U, \D U\rangle-\int_{E_r(x_0)}\langle A\D V, \D
V\rangle= \int_{E_r(x_0)}\langle A\D(U+V), \D(U-V)\rangle \\
\begin{aligned}
  &\leq C\int_{E_r(x_0)}|\D(U+V)||\D(U-V)|\\
&\le Cr^{\g}\int_{E_r(x_0)}\left(|\D U|^2+|\D V|^2\right)+Cr^{-\g}\int_{E_r(x_0)}|\D(U-V)|^2\\
&
\begin{multlined}
\le Cr^{\g}\int_{E_r(x_0)}\langle A\D U, \D U\rangle+Cr^{\g}\int_{E_r(x_0)}\langle A\D V, \D V\rangle\\
+ Cr^{2(1-n/p)-\g}\int_{E_r(x_0)}\langle A\D U, \D U\rangle,
\end{multlined}
\end{aligned}
\end{multline*}
where we applied Young's inequality and used \eqref{ex-(u-v)-estimate} at the end. We choose $\g=1-n/p$ to complete the proof.
\end{proof}


\end{document}